\documentclass[a4paper,11pt,reqno]{amsart}

\usepackage[margin=1in,includehead,includefoot]{geometry}
\usepackage[utf8]{inputenc}
\usepackage[T1]{fontenc}
\usepackage{lmodern,microtype}
\usepackage{upref}
\usepackage{mathtools,amssymb,amsthm}
\usepackage[foot]{amsaddr}
\usepackage{enumitem}
\usepackage{tikz}
\usepackage{graphicx}
\usepackage{wrapfig}
\usepackage{hyperref}
\usepackage[b]{esvect} 
\usepackage{textcomp}
\usepackage{mdframed}
\usepackage{xpatch}
\usepackage{bookmark}
\usepackage{multirow}
\usepackage{cases}
\usepackage[margin=5mm]{caption}
\usepackage{array}
\usepackage{arydshln}
\usepackage{mycommands}

\usepackage{todonotes}

\graphicspath{{./graphics/}}

\makeatletter
\let\old@setaddresses\@setaddresses
\def\@setaddresses{\bigskip{\parindent 0pt\let\scshape\relax\let\ttfamily\relax\old@setaddresses}}
\makeatother

\newcommand{\tikzrect}[1]{}
\newcommand{\tikzgener}[1]{}
\tikzstyle{wall}=[line width=1pt]

\definecolor{p1col1}{rgb}{0.266122, 0.486664, 0.802529}
\definecolor{p1col2}{rgb}{0.513417, 0.72992, 0.440682}
\definecolor{p1col3}{rgb}{0.863512, 0.670771, 0.236564}
\definecolor{p1col4}{rgb}{0.857359, 0.131106, 0.132128}
\definecolor{p1col5}{rgb}{0.52734, 0.24218, 0.24218}
\definecolor{p1col6}{rgb}{0.45, 0.45, 0.45}

\definecolor{p2col1}{rgb}{0.264135, 0.201429, 0.745889}
\definecolor{p2col2}{rgb}{0.256326, 0.430921, 0.808553}
\definecolor{p2col3}{rgb}{0.324106, 0.60897, 0.708341}
\definecolor{p2col4}{rgb}{0.439128, 0.704968, 0.52925}
\definecolor{p2col5}{rgb}{0.597181, 0.742185, 0.36771}
\definecolor{p2col6}{rgb}{0.764712, 0.728302, 0.273608}
\definecolor{p2col7}{rgb}{0.88018, 0.631684, 0.227665}
\definecolor{p2col8}{rgb}{0.897354, 0.41824, 0.185007}
\definecolor{p2col9}{rgb}{0.857359, 0.131106, 0.132128}

\newtheorem{theorem}{Theorem}

\newtheorem{lemma}[theorem]{Lemma}
\theoremstyle{remark}

\newtheorem{remark}[theorem]{Remark}
\theoremstyle{definition}


\hypersetup{
  pdftitle={Combinatorial generation via permutation languages. III. Rectangulations},
  pdfauthor={Arturo Merino, Torsten M\"utze}
}

\title{Combinatorial generation via permutation languages. \\ III. Rectangulations}

\author{Arturo Merino}
\address[Arturo Merino]{Department of Mathematics, TU Berlin, Germany}
\email{merino@math.tu-berlin.de}

\author{Torsten M\"utze}
\address[Torsten M\"utze]{Department of Computer Science, University of Warwick, United Kingdom}
\email{torsten.mutze@warwick.ac.uk}

\thanks{An extended abstract of this paper appeared in the Proceedings of SoCG~2021.
Arturo Merino was supported by ANID Becas Chile 2019-72200522. Torsten M\"utze is also affiliated with the Faculty of Mathematics and Physics, Charles University Prague, Czech Republic, and he was supported by Czech Science Foundation grant GA~19-08554S. Both authors were supported by German Science Foundation grant~413902284.}

\begin{document}

\begin{abstract}
A generic rectangulation is a partition of a rectangle into finitely many interior-disjoint rectangles, such that no four rectangles meet in a point.
In this work we present a versatile algorithmic framework for exhaustively generating a large variety of different classes of generic rectangulations.
Our algorithms work under very mild assumptions, and apply to a large number of rectangulation classes known from the literature, such as generic rectangulations, diagonal rectangulations, 1-sided/area-universal, block-aligned rectangulations, and their guillotine variants, including aspect-ratio-universal rectangulations.
They also apply to classes of rectangulations that are characterized by avoiding certain patterns, and in this work we initiate a systematic investigation of pattern avoidance in rectangulations.
Our generation algorithms are efficient, in some cases even loopless or constant amortized time, i.e., each new rectangulation is generated in constant time in the worst case or on average, respectively.
Moreover, the Gray codes we obtain are cyclic, and sometimes provably optimal, in the sense that they correspond to a Hamilton cycle on the skeleton of an underlying polytope.
These results are obtained by encoding rectangulations as permutations, and by applying our recently developed permutation language framework.
\end{abstract}


\maketitle

\section{Introduction}
\label{sec:intro}

Partitioning a geometric shape into smaller shapes is a fundamental theme in discrete and combinatorial geometry.
In this paper we consider \emph{rectangulations}, i.e., partitions of a rectangle into finitely many interior-disjoint rectangles.
Such partitions have an abundance of practical applications, which motivates their combinatorial and algorithmic study.
For example, rectangulations are an appealing way to represent geographic information as a cartogram.
This is a map where each country is represented as a rectangle, the adjacencies between rectangles correspond to those between countries, and the areas of the rectangles are determined by some geographic variable, such as population size~\cite{MR2331063}.
If the rectangulation is \emph{area-universal}~\cite{MR2967467} or \emph{aspect-ratio-universal}~\cite{felsner_nathenson_toth_2021}, respectively, then such an adjacency-preserving cartogram can be drawn for any assignment of area values or aspect ratios to the rectangles.
Another important use of rectangulations is as floorplans in VLSI design and architectural design.
These problems often involve additional constraints on top of adjacency, such as extra space for wires~\cite{DBLP:conf/dac/Otten82} or proportion limits for the rooms~\cite{mitchell_steadman_liggett_1976}.
An important notion in this context are \emph{slicing} floorplans~\cite{DBLP:conf/dac/Otten82}, also known as \emph{guillotine} floorplans, i.e., floorplans that can be subdivided into their constituent rectangles by a sequence of straight vertical or horizontal cuts.

Rectangulations have rich combinatorial properties, and a task that has received a lot of attention is counting, i.e., determining the number of rectangulations of a particular type with $n$ rectangles, either exactly as a function of~$n$~\cite{DBLP:journals/todaes/YaoCCG03} or asymptotically as $n$ grows~\cite{shen_chu_2003}.
This led to several beautiful bijections of rectangulations with pattern-avoiding permutations~\cite{MR2233287,MR2864445,MR3084577} or with twin binary trees~\cite{DBLP:journals/todaes/YaoCCG03}.
The focus of this paper is on another fundamental algorithmic task, which is more fine-grained than counting, namely exhaustive generation, meaning that every rectangulation from a given class must be produced exactly once.
While such generation algorithms are known for many other discrete objects such as permutations, combinations, subsets, trees etc.\ and covered in standard textbooks such as Knuth's~\cite{MR3444818}, much less is known about the generation of geometric objects such as rectangulations.

The ultimate goal for a generation algorithm is to produce each new object in time~$\cO(1)$, which requires that consecutively generated objects differ only by a `small local change'.
Such a minimum change listing of combinatorial objects is often called a \emph{Gray code}~\cite{MR1491049}.
If the time bound~$\cO(1)$ for producing the next object holds in every step, then the algorithm is called \emph{loopless}~\cite{MR0366085}, and if it holds on average it is called \emph{constant amortized time} (CAT)~\cite{MR3523863}.
The Gray code problem entails the definition of a \emph{flip graph}, which has as nodes all the combinatorial objects to be generated, and an edge between any two objects that differ in the specified small way.
Clearly, computing a Gray code ordering of the objects is equivalent to traversing a Hamilton path or cycle in the corresponding flip graph.
It turns out that some interesting flip graphs arising from rectangulations can be equipped with a natural lattice structure~\cite{meehan_2019}, analogous to the Tamari lattice on triangulations, and realized as polytopes in high-dimensional space~\cite{MR2871762}, analogous to the associahedron (see \cite{padrol_pilaud_ritter_2021} for generalizations).
This ties in the Gray code problem with deep methods and results from lattice and polytope theory.

\subsection{Our results}
\label{sec:results}

The main contribution of this paper is a versatile algorithmic framework for generating a large variety of different classes of generic rectangulations, i.e., rectangulations with the property that no four rectangles meet in a point.
In particular, we obtain efficient generation algorithms for several interesting classes known from the literature, in some cases loopless or CAT algorithms; see Table~\ref{tab:rect}.

The initialization time and memory requirement for all these algorithms is linear in the number of rectangles.
The classes of rectangulations shown in the table arise from generic rectangulations by imposing structural constraints, such as the guillotine property or forbidden configurations, or by equivalence relations, and they will be defined in Section~\ref{sec:flips}.
We implemented the algorithms generating the classes of rectangulations from the table in C++, and we made the code available for download and experimentation on the Combinatorial Object Server~\cite{cos_rect}.

\begin{table}[t!]
\small
\caption{Classes of rectangulations that can be generated by our algorithms.
The second column gives a description of the class in terms of forbidden rectangulation patterns (n/a means not applicable), and one or more bijectively equivalent classes of pattern-avoiding permutations.
Underlined permutation patterns are so-called vincular patterns; see~\cite{perm_series_i} and the papers referenced in the table.
The last column specifies the obtained runtime bound for generating each rectangulation, where $n$ is the number of rectangles.
These are all worst case bounds that apply in every step (in particular, LL=loopless), with the exception of the $\cO(1)$~bound for generic rectangulations, which holds on average (CAT=constant amortized time).
For more extensive counting results on pattern-avoiding rectangulations, see Section~\ref{sec:counting}.}
\label{tab:rect}
\renewcommand{\arraystretch}{1.1}
\setlength\tabcolsep{3pt}
\hspace{-22mm}
\parbox{15cm}{
\begin{tabular}{|p{3mm}:p{22mm}|p{46mm}|p{53mm}|p{20mm}|p{16mm}|}
\hline
\multicolumn{2}{|l|}{\textbf{Class}} & \textbf{Forbidden patterns} & \textbf{Counts/OEIS} \cite{oeis} & \textbf{Refs.} & \textbf{Runtime} \\ \hline
\multicolumn{2}{|l|}{generic} & $\emptyset$ \newline $\{3\ul{51}24,3\ul{51}42,24\ul{51}3,42\ul{51}3\}$: \newline 2-clumped permutations
         & $1,2,6,24,116,642,3938,26194,\ldots$ \newline \href{https://oeis.org/A342141}{A342141} & \cite{MR2864445,meehan_2019} & $\cO(1)$ CAT \newline Thm.~\ref{thm:next-gen} \\ \hline
\multicolumn{2}{|l|}{\parbox[t]{23mm}{diagonal \newline =mosaic floorpl. \newline /R-equivalence}} & \raisebox{-3.5pt}{$\big\{\zvu,\zhr\big\}$} \newline $\{2\ul{41}3,3\ul{14}2\}$: Baxter \newline $\{2\ul{41}3,3\ul{41}2\}$: twisted Baxter \newline $\{2\ul{14}3,3\ul{14}2\}$
         & $1,2,6,22,92,422,2074,10754,\ldots$ \newline \href{https://oeis.org/A001181}{A001181} (Baxter numbers) & \cite{DBLP:journals/todaes/YaoCCG03,MR2233287,MR2871762,MR3878132} & $\cO(1)$ LL \newline Thm.~\ref{thm:next-diag} \\ \hline
\multicolumn{2}{|l|}{\parbox[t]{26mm}{1-sided \newline =area-universal}} & \raisebox{-3.5pt}{$\big\{\zvu,\zhr,\zvd,\zhl\big\}$} \newline $\{2\ul{41}3,3\ul{14}2,2\ul{14}3,3\ul{41}2\}$ & $1,2,6,20,72,274,1088,4470,\ldots$ & \cite{MR2967467,leifheit_2021} & $\cO(n)$ \newline Thm.~\ref{thm:nextjumpP-RD} \\ \hline
\multicolumn{2}{|l|}{\parbox[t]{23mm}{block-aligned \\ /S-equivalence}} & n/a \newline $\{2\ul{14}3,3\ul{41}2\}$
         & $1,1,2,6,22,88,374,1668,7744,\ldots$ \newline \href{https://oeis.org/A214358}{A214358} & \cite{MR3084577} & $\cO(1)$ LL \newline Thm.~\ref{thm:next-block} \\ \hline\hline
\multirow{4}{*}{\rotatebox{90}{guillotine\hspace{25mm}}} & generic & \raisebox{-3.5pt}{$\big\{\millr,\milll\big\}$} \newline & $1,2,6,24,114,606,3494,21434,\ldots$ & &  $\cO(n)$ \newline Thm.~\ref{thm:nextjumpP-RD} \\ \cline{2-6}
 & diagonal \newline =slicing fl.pl. \newline /R-equiv. & \raisebox{-3.5pt}{$\big\{\millr,\milll,\zvu,\zhr\big\}$} \newline $\{2413,3142\}$: separable
           & $1,2,6,22,90,394,1806,8558,\ldots$ \newline \href{https://oeis.org/A006318}{A006318} (Schr\"oder numbers) & \cite{DBLP:journals/todaes/YaoCCG03,MR2233287,MR2601798,MR3084577} & $\cO(n)$ \newline Thm.~\ref{thm:nextjumpP-RD} \\ \cline{2-6}
 & 1-sided \newline =aspect-ratio-universal & \raisebox{-3.5pt}{$\big\{\millr,\milll,\zvu,\zhr,\zvd,\zhl\big\}$} \newline $\{2413,3142,2\ul{14}3,3\ul{41}2\}$
           & $1,2,6,20,70,254,948,3618,\ldots$ \newline \href{https://oeis.org/A078482}{A078482} & \cite{MR2601798,felsner_nathenson_toth_2021,leifheit_2021} & $\cO(n)$ \newline Thm.~\ref{thm:nextjumpP-RD} \\ \cline{2-6}
 & & \raisebox{-3.5pt}{$\big\{\millr,\milll,\zvu,\zhr,\zvd,\zhl,$} \newline \raisebox{-3.5pt}{$\hvert,\hhor\big\}$} \newline $\{2413,3142,2143,3412\}$
           & $1,2,6,20,68,232,792,2704,\ldots$ \newline \href{https://oeis.org/A006012}{A006012} & \cite{MR2601798} & $\cO(n^2)$ \newline Thm.~\ref{thm:nextjumpP-RD} \\ \cline{2-6}
 & block-aligned \newline /S-equiv. & n/a \newline $\{2413,3142,2\ul{14}3,3\ul{41}2\}$
           & $1,1,2,6,20,70,254,948,3618,\ldots$ \newline \href{https://oeis.org/A078482}{A078482} & \cite{MR3084577} & $\cO(n)$ \newline Thm.~\ref{thm:nextjumpP-S} \\ \hline
\end{tabular}
}
\end{table}

The classes of rectangulations that our algorithms can generate are not limited to the examples shown in Table~\ref{tab:rect}, but can be described by the following \emph{closure property}; see Figure~\ref{fig:closure}.
Given an infinite class of rectangulations~$\cC$, we require that if a rectangulation~$R$ is contained in~$\cC$, then the rectangulation obtained from~$R$ by deleting the bottom-right rectangle is also in~$\cC$, and the two rectangulations obtained from~$R$ by inserting a new rectangle at the bottom or right, respectively, are also in~$\cC$ (formal definitions of deletion and insertion are given in Section~\ref{sec:prelim}).
If $\cC$ satisfies this property, then our algorithms allow generating the set $\cC_n\seq\cC$ of all rectangulations from~$\cC$ with exactly~$n$ rectangles, for every $n\geq 1$, by so-called \emph{jumps}, a minimum change operation that generalizes simple flips, T-flips, and wall slides studied in~\cite{MR2864445,MR3878132} (the formal definition of jumps is in Section~\ref{sec:jumps}).
Moreover, if the class~$\cC$ is symmetric, i.e., if $R$ is in~$\cC$ then the rectangulation obtained from~$R$ by reflection at the diagonal from top-left to bottom-right is also in~$\cC$, then the jump Gray code for~$\cC_n$ is cyclic, i.e., the last rectangulation differs from the first one only by a jump.
In other words, we not only obtain a Hamilton path in the corresponding flip graph, but a Hamilton cycle.
In fact, all the classes of rectangulations listed in Table~\ref{tab:rect} satisfy the aforementioned closure and symmetry properties, so in all those cases we obtain cyclic jump Gray codes.

Generic rectangulations and diagonal rectangulations, shown in the first two rows of Table~\ref{tab:rect}, have an underlying lattice and polytope structure~\cite{MR2871762,meehan_2019}, and in those two cases our Gray codes form a Hamilton cycle on the skeleton of this polytope, i.e., jumps are provably optimal minimum change operations.
The Gray codes for these two rectangulation classes with $n=1,\ldots,5$ rectangles are shown in the appendix.

\begin{wrapfigure}{r}{0.6\textwidth}
\centering
\includegraphics{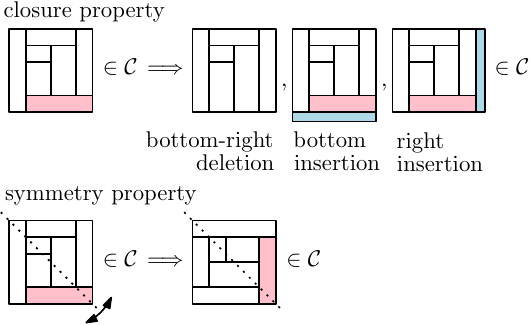}
\caption{Closure property and symmetry property.}
\label{fig:closure}
\end{wrapfigure}
It turns out that many interesting classes of rectangulations can be characterized by pattern avoidance; see the second column in Table~\ref{tab:rect}.
Under very mild conditions on the patterns, these classes satisfy the aforementioned closure property, and can hence be generated by our framework.
In this work we initiate a systematic investigation of pattern avoidance in rectangulations, and we obtain the first counting results for many known and new classes; see the third column in Table~\ref{tab:rect} and the more extensive tables in Section~\ref{sec:counting}.

Our generation framework for rectangulations consists of two main algorithms.
The first is a simple greedy algorithm that generates a jump Gray code ordering for any set of rectangulations~$\cC_n\seq\cC$ for which~$\cC$ satisfies the aforementioned closure property; see Algorithm~\Jrect{} and Theorem~\ref{thm:jump-rect} in Section~\ref{sec:algo}.
The second is a memoryless version of the first algorithm, which computes the same ordering of rectangulations; see Algorithm~\Mrect{} and Theorem~\ref{thm:algo-rect} in Section~\ref{sec:efficient}.
This algorithm can be fine-tuned to derive efficient algorithms for several known rectangulation classes such as the ones listed in Table~\ref{tab:rect}, by providing corresponding jump oracles for the class~$\cC$.

To prove Theorems~\ref{thm:jump-rect} and~\ref{thm:algo-rect}, we encode rectangulations by permutations as described by Reading~\cite{MR2864445}, and we then apply our framework for exhaustively generating permutation languages presented in~\cite{DBLP:conf/soda/HartungHMW20,perm_series_i,perm_series_ii}.
The minimum change operations on permutations used in that framework translate to jumps on rectangulations.
Generating different classes of rectangulations efficiently is thus another major new application of our permutation language framework, and in this paper we flesh out the details of this application.

\subsection{Related work}

There has been some prior work on generating a few special classes of rectangulations, all based on Avis and Fukuda's reverse search method~\cite{MR1380066}.
Specifically, Nakano~\cite{MR1917735} described a CAT generation algorithm for generic rectangulations, which does not produce a Gray code, however.
This algorithm has been adapted by Takagi and Nakano~\cite{DBLP:journals/scjapan/TakagiN04} to generate generic rectangulations with bounds on the number of rectangles that do not touch the outer face.
Yoshii, Chigira, Yamanaka and Nakano~\cite{DBLP:journals/ieicet/YoshiiCYN06} gave a Gray code for generic rectangulations based on a generating tree that is different from ours, resulting in a loopless algorithm.
Their Gray code changes at most 3 edges of the rectangulation in each step, whereas our algorithm changes only 1~edge in each step for generic and for diagonal rectangulations.
Consequently, none of the listings produced by these earlier algorithms corresponds to a walk along the skeleton of the underlying polytope.

There has been a lot of work on combinatorial properties of rectangulations.
Yao, Chen, Cheng and Graham~\cite{DBLP:journals/todaes/YaoCCG03} showed that diagonal rectangulations are counted by the Baxter numbers and that guillotine diagonal rectangulations are counted by the Schr\"oder numbers, using a bijection between diagonal rectangulations and twin binary trees.
Ackerman, Barequet and Pinter~\cite{MR2233287} presented another bijection between diagonal rectangulations and Baxter permutations, which also yields a bijection between guillotine diagonal rectangulations and separable permutations.
Leifheit~\cite{leifheit_2021} showed that this bijection can be restricted to the 1-sided variants of these two rectangulation classes by adding two permutation patterns; see Table~\ref{tab:rect}.
Shen and Chu~\cite{shen_chu_2003} provided asymptotic estimates for diagonal rectangulations and their guillotine variant.
Moreover, He~\cite{MR3192492} presented an optimal encoding of diagonal rectangulations with $n$ rectangles using only $3n-3$ bits, which is optimal.

The term `generic rectangulation' was coined by Reading~\cite{MR2864445}, who established a bijection between generic rectangulations and 2-clumped permutations, proving that these permutations are representatives of equivalence classes of a lattice congruence of the weak order on the symmetric group.
Earlier, generic rectangulations had been studied under the name `rectangular drawings' by Amano, Nakano and Yamanaka~\cite{amano_nakano_yamanaka_2007} and by Inoue, Takahashi and Fujimaki~\cite{fujimaki_inoue_takahashi_2009,DBLP:journals/ieicet/InoueTF09}, who established recursion formulas and asymptotic bounds for their number.
More general classes of rectangular partitions were analyzed by Conant and Michaels~\cite{MR3167602}.

Ackerman, Barequet and Pinter~\cite{MR2244135} considered the setting where we are given a set of $n$ points in general position in a rectangle, and the goal is to partition the rectangle into smaller rectangles by $n$ walls, such that each point from the set lies on a distinct wall.
They showed that for every set of points that forms a separable permutation in the plane, the number of possible rectangulations is the $(n+1)$st Baxter number, and for every point set the number of possible guillotine rectangulations is the $n$th Schr\"oder number.
They also presented a counting and generation procedure based on simple flips and T-flips using reverse search, which was later improved by Yamanaka, Rahman and Nakano~\cite{DBLP:journals/ieicet/YamanakaRN18}.

\subsection{Outline of this paper}

In Section~\ref{sec:prelim} we provide basic definitions and concepts that will be used throughout the paper.
In Section~\ref{sec:algo} we present a greedy algorithm for generating a set of rectangulations by jumps, and we provide a sufficient condition for the algorithm to succeed.
In Section~\ref{sec:avoidance} we show that the algorithm applies to a large number of rectangulation classes that are characterized by pattern avoidance.
In Section~\ref{sec:efficient} we demonstrate how to make our generation algorithm memoryless and efficient.
The implementation details for our algorithms are provided in Sections~\ref{sec:implementation} and~\ref{sec:oracles}.
The proofs of Theorems~\ref{thm:jump-rect} and~\ref{thm:algo-rect} are presented in Section~\ref{sec:proofs}, by establishing a connection between rectangulations and permutations and by applying our permutation language framework.
The results for one special class of rectangulations mentioned in Table~\ref{tab:rect} are deferred to Section~\ref{sec:Sequiv}.
In Section~\ref{sec:counting} we report on our computer experiments about counting pattern-avoiding rectangulations.
We conclude the paper with some interesting open questions in Section~\ref{sec:open}.
Several visualizations of Gray codes produced by our algorithms are shown in the appendix.

\section{Preliminaries}
\label{sec:prelim}

\subsection{Generic rectangulations}

A \emph{generic rectangulation}, or rectangulation for short, is a partition of a rectangle into finitely many interior-disjoint axis-aligned rectangles, such that no four rectangles of the partition have a point in common; see Figure~\ref{fig:rect}.
In other words, every point where three rectangles meet, or where two rectangles meet the outer face forms a T-joint with the incident rectangle boundaries.
Given rectangles $r$ and~$s$, we say that $r$ is \emph{left} of~$s$, and $s$ is \emph{right} of~$r$, if the right side of~$r$ intersects the left side of~$s$ (necessarily in a line segment, rather than a single point).
Similarly, we say that $r$ is \emph{below}~$s$, and $s$ is \emph{above}~$r$, if the top side of~$r$ intersects the bottom side of~$s$.
We consider generic rectangulations up to equivalence that preserves the left/right and below/above relations between rectangles, and we use $\cR_n$, $n\geq 1$, to denote the set of all rectangulations with $n$ rectangles.
We write $\square$ for the unique rectangulation in~$\cR_1$, i.e., the rectangulation consisting of a single rectangle.

\begin{wrapfigure}{r}{0.4\textwidth}
\centering
\includegraphics[page=1]{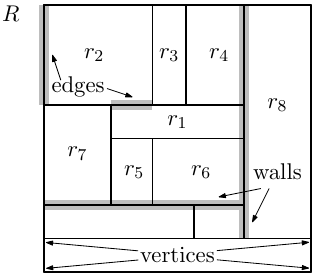}
\caption{Generic rectangulation~$R$ with 11 rectangles.
The rectangle~$r_1$ is below $r_2$, $r_3$ and $r_4$, above $r_5$ and~$r_6$, right of~$r_7$ and left of~$r_8$.}
\label{fig:rect}
\vspace{-8mm}
\end{wrapfigure}
We refer to every rectangle corner in a rectangulation as a \emph{vertex}, to every minimal line segment between two vertices as an \emph{edge}, and to every maximal line segment between two vertices that are not corners of the rectangulation as a \emph{wall}.
The \emph{type} of a vertex that is not a corner of the rectangulation describes the shape of the T-joint at this vertex, and it is one of $\bottomT$, $\rightT$, $\topT$, or $\leftT$.

\subsection{Flip operations and classes of rectangulations}
\label{sec:flips}

Our Gray codes use three types of local change operations on rectangulations; see Figure~\ref{fig:flips}.

\begin{wrapfigure}{r}{0.40\textwidth}
\centering
\vspace{-3mm}
\includegraphics[page=1]{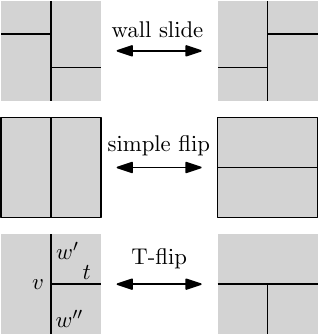}
\caption{Local change operations on rectangulations.}
\label{fig:flips}
\vspace{-1mm}
\end{wrapfigure}
A \emph{wall slide} swaps the order of two neighboring vertices of types~$\rightT$ and~$\leftT$ along a vertical wall, or of types~$\bottomT$ and~$\topT$ along a horizontal wall.
A \emph{simple flip} swaps the orientation of a wall that separates two rectangles.
Given a vertex~$v$ that belongs to three rectangles, we consider the wall~$w$ that goes through~$v$ and the wall~$t$ that ends at~$v$, and we let $w'$ and~$w''$ be the two halves of~$w$ meeting in~$v$.
If $w'$ or~$w''$ is an edge, respectively, then a \emph{T-flip} swaps the orientation of this edge so that it merges with~$t$.

We now define various interesting subclasses of generic rectangulations that have been studied in the literature and that appear in Table~\ref{tab:rect}.
Examples illustrating these classes are in Figure~\ref{fig:classes}.
A \emph{diagonal} rectangulation is one in which every rectangle intersects the \emph{main diagonal} that goes from the top-left to the bottom-right corner of the rectangulation.
We write $\cD_n\seq\cR_n$ for the set of all diagonal rectangulations with $n$ rectangles.
Diagonal rectangulations are characterized by avoiding the wall patterns~$\zvu$ and~$\zhr$~\cite{MR3878132}.
Consider the equivalence relation on~$\cR_n$ obtained from wall slides, sometimes referred to as \emph{R-equivalence}~\cite{MR3084577}.
The equivalence classes are referred to as \emph{mosaic floorplans}, and every equivalence class contains exactly one diagonal rectangulation, obtained by repeatedly destroying occurrences of~$\zvu$ or~$\zhr$ by wall slides~\cite{MR3878132}.
Consequently, in a diagonal rectangulation, along every vertical wall, all $\rightT$-vertices are below all $\leftT$-vertices, and along every horizontal wall, all $\topT$-vertices are to the left of all $\bottomT$-vertices.

In a \emph{1-sided} rectangulation, every wall is the side of at least one rectangle, i.e., these rectangulations are characterized by avoiding the four patterns~\zvu,\zhr,\zvd and~\zhl.
The notion of 1-sidedness was introduced by Eppstein, Mumford, Speckmann, and Verbeek~\cite{MR2967467} to characterize \emph{area-universal} rectangulations, i.e., for any assignment of areas to the rectangles, the rectangulation can be drawn so that each rectangle has the prescribed area.

Asinowski et al.~\cite{MR3084577} also considered the equivalence relation on~$\cR_n$ obtained from wall slides and simple flips, and they called it \emph{S-equivalence}.
By definition, S-equivalence is a coarser relation than R-equivalence, i.e., the equivalence classes are obtained by identifying mosaic floorplans that differ in simple flips.
In Section~\ref{sec:Sequiv} we introduce \emph{block-aligned} rectangulations, which are a subset of diagonal rectangulations with the property that every equivalence class of S-equivalence contains exactly one block-aligned rectangulation.

A rectangulation is \emph{guillotine}, if each of its rectangles can be cut out from the entire rectangulation by a sequence of straight vertical or horizontal cuts.
Guillotine rectangulations are characterized by avoiding the windmill patterns~$\millr$ and~$\milll$, which is a folklore result.
Various special classes of guillotine diagonal rectangulations, characterized by the avoidance of certain wall configurations, were introduced by Asinowski and Mansour~\cite{MR2601798} (see Section~\ref{sec:avoidance} for precise definitions of these configurations).
Mosaic floorplans that are guillotine are also known as \emph{slicing} floorplans.

Felsner, Nathenson, and T\'oth~\cite{felsner_nathenson_toth_2021} showed that 1-sided guillotine rectangulations are precisely the \emph{aspect-ratio-universal} rectangulations, i.e., for any assignment of aspect ratios to the rectangles, the rectangulation can be drawn so that each rectangle has the prescribed aspect ratio.

\begin{figure}
\includegraphics[page=2]{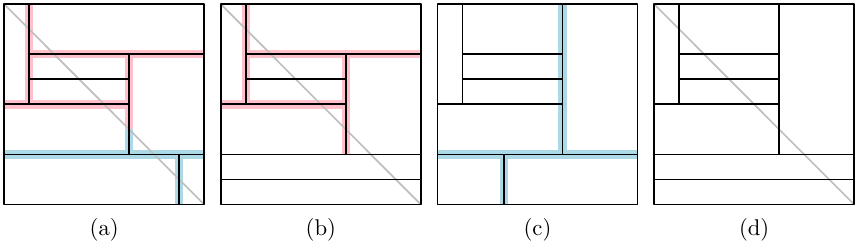}
\caption{Examples of different classes of rectangulations: (a) diagonal, but neither 1-sided nor guillotine (b) 1-sided, but not guillotine (c) guillotine, but not diagonal (d) guillotine and 1-sided.
Occurrences of the corresponding forbidden patterns are highlighted.}
\label{fig:classes}
\end{figure}

\begin{wrapfigure}{r}{0.35\textwidth}
\includegraphics{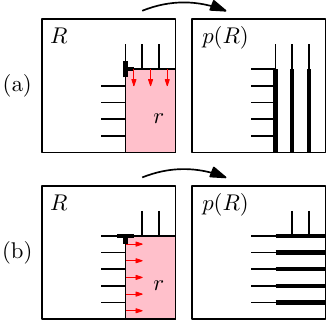}
\caption{Deletion operation.}
\label{fig:delete}
\end{wrapfigure}

\subsection{Deletion of rectangles}
\label{sec:deletion}

We now describe two operations on a generic rectangulation~$R$, namely deleting a rectangle and inserting a rectangle.
The resulting rectangulations will be denoted by~$p(R)$ and~$c_i(R)$, notations that refer to the parent and children of~$R$, in a tree structure that will be discussed shortly.
The deletion and insertion operations were introduced in~\cite{DBLP:conf/iccad/HongHCGDCG00} and heavily used e.g.\ in~\cite{MR2233287} and~\cite{MR1917735}.

The idea of deletion is to contract the rectangle in the bottom-right corner of the rectangulation.
Formally, given a rectangulation~$R\in\cR_n$, $n\geq 2$, we consider the rectangle~$r$ in the bottom-right corner, and we consider the top-left vertex of~$r$.
If this vertex has type~$\rightT$, then we collapse~$r$ by sliding its top side, which forms a wall, downwards until it merges with the bottom side of~$r$; see Figure~\ref{fig:delete}~(a).
Similarly, if this vertex has type~$\bottomT$, then we collapse~$r$ by sliding its left side, which forms a wall, to the right until it merges with the right side of~$r$; see Figure~\ref{fig:delete}~(b).
We denote the resulting rectangulation with $n-1$ rectangles by~$p(R)\in\cR_{n-1}$, and we say that $p(R)$ is obtained from~$R$ by \emph{deletion}.

\begin{figure}
\includegraphics{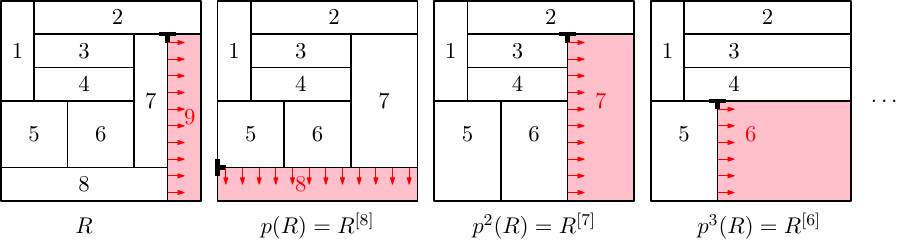}
\caption{A rectangulation and the indexing of its rectangles given by repeated deletion.}
\label{fig:numbering}
\end{figure}

Moreover, we denote the $n$ rectangles of~$R$ by $r_n,r_{n-1},\ldots,r_1$ in the order in which they are deleted when applying the deletion operation exhaustively; see Figure~\ref{fig:numbering}.
Clearly, if $r_i$ is deleted and its top-left vertex has type~$\rightT$, then the rightmost rectangle above~$r_i$ is~$r_{i-1}$.
Similarly, if the top-left vertex has type~$\bottomT$, then the lowest rectangle to the left of~$r_i$ is~$r_{i-1}$.

For any $R\in\cR_n$ and $i=1,\ldots,n$ we define $R^{[i]}:=p^{n-i}(R)$, i.e., this is the sub-rectangulation of~$R$ formed by the first $i$ rectangles; see Figure~\ref{fig:numbering}.

\subsection{Insertion of rectangles}
\label{sec:insertion}

The idea of insertion is to add a new rectangle into the bottom-right corner of the rectangulation.
Given a rectangulation~$R\in\cR_{n-1}$, we first define a set of points in~$R$ that can become the top-left corner of the newly added rectangle; see Figure~\ref{fig:insert-order}.

For any rectangle~$r$ in~$R\in\cR_{n-1}$, $n\geq 2$, that touches the bottom boundary of~$R$, we consider all edges forming the left side of~$r$, and from every such edge we select one interior point, and we refer to it as a \emph{vertical insertion point}.

\begin{wrapfigure}{r}{0.35\textwidth}
\centering
\vspace{-5mm}
\includegraphics[page=1]{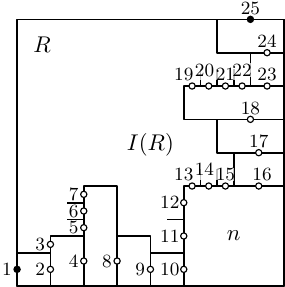}
\caption{Linear ordering of insertion points. First and last insertion point are filled.}
\vspace{-4mm}
\label{fig:insert-order}
\end{wrapfigure}
Similarly, for any rectangle~$r$ in~$R$ that touches the right boundary of~$R$, we consider the set of all edges forming the top side of~$r$, and from every such edge we select one interior point, and we refer to it as a \emph{horizontal insertion point}.
Combinatorially it does not make a difference which interior point of each edge is selected.

We order the insertion points linearly, by sorting all vertical insertion points lexicographically by their $(x,y)$-coordinates, followed by all horizontal insertion points sorted lexicographically by their $(y,x)$-coordinates; see Figure~\ref{fig:insert-order}.
We write $I(R)=(q_1,q_2,\ldots,q_\nu)$ for the sequence of all insertion points ordered in this linear order.
In particular, $\nu=\nu(R)$ denotes the number of insertion points.

\begin{lemma}
\label{lem:nu}
For any rectangulation~$R\in\cR_{n-1}$ we have $\nu(R)\leq n$.
\end{lemma}

\begin{proof}
Each rectangle in~$R$ has at most one vertical insertion point on its right side, and at most one horizontal insertion point on its bottom side.
Moreover, no rectangle has both, the bottom-right rectangle~$r_{n-1}$ has neither of the two, and exactly 2 insertion points lie on the boundary of~$R$.
Combining these observations shows that $\nu(R)\leq ((n-1)-1)+2=n$.
\end{proof}

Clearly, the upper bound in Lemma~\ref{lem:nu} is attained if every rectangle touches the bottom or right boundary of~$R$.

Given $R\in\cR_n$ and the sequence of insertion points $I(R)=(q_1,\ldots,q_\nu)$, for each $i=1,\ldots,\nu$ we define a rectangulation $c_i(R)\in\cR_n$ as follows:
If $q_i$ is a vertical insertion point, then $c_i(R)$ is obtained from~$R$ by inserting a new rectangle~$r_n$ in the bottom-right corner such that $r_n$ has above it exactly all rectangles which in~$R$ lie to the right of~$q_i$ and touch the bottom boundary of~$R$, and such that $r_n$ has to its left exactly all rectangles which in~$R$ touch the vertical wall through~$q_i$ below~$q_i$; see Figure~\ref{fig:insert}~(a).
Similarly, if $q_i$ is a horizontal insertion point, then $c_i(R)$ is obtained from~$R$ by inserting a new rectangle~$r_n$ in the bottom-right corner such that $r_n$ has to its left exactly all rectangles which in~$R$ lie below~$q_i$ and touch the right boundary of~$R$, and such that $r_n$ has above it exactly all rectangles which in~$R$ touch the horizontal wall through~$q_i$ to the right of~$q_i$; see Figure~\ref{fig:insert}~(b).
We say that $c_i(R)$ is obtained from~$R$ by \emph{insertion}.

\begin{figure}
\includegraphics[page=2]{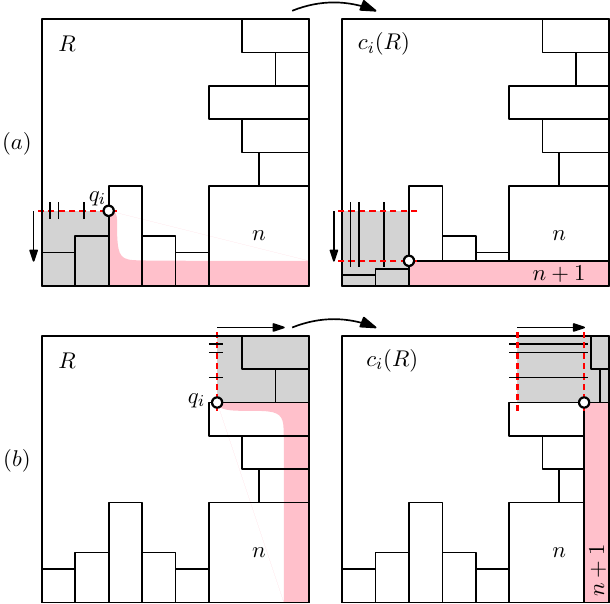}
\caption{Insertion operation.}
\label{fig:insert}
\end{figure}

By these definitions, the operations of deletion and insertion are inverse to each other, which we record in the following lemma.

\begin{lemma}
\label{lem:del-ins}
For any rectangulation $R\in\cR_{n-1}$ and any two distinct insertion points~$q_i$ and~$q_j$ from~$I(R)$, the rectangulations~$c_i(R)\in\cR_n$ and~$c_j(R)\in\cR_n$ are distinct, and we have $R=p(c_i(R))=p(c_j(R))$.
Moreover, for any $R'\in\cR_n$ with $p(R')=R$ there is an insertion point~$q_i$ in~$I(R)$ such that~$c_i(R)=R'$.
\end{lemma}

The first and last insertion point play a special role in our arguments, which is why they are highlighted in Figure~\ref{fig:insert}.
We say that $R$ is \emph{bottom-based} if $R$ has a rectangle whose bottom side is the entire bottom boundary of~$R$, and $R$ is \emph{right-based} if $R$ has a rectangle whose right side is the entire right boundary of~$R$.
Note that the rectangulation $\square\in\cR_1$ is both bottom-based and right-based, and if $n\geq 2$, then $R\in\cR_n$ is bottom-based if and only if $R=c_1(p(R))$ and right-based if and only if $R=c_{\nu(p(R))}(p(R))$.

\section{The basic algorithm}
\label{sec:algo}

In this section we present the basic algorithm that we use to generate a set of rectangulations~$\cC_n\seq\cR_n$.

\subsection{Jumps in rectangulations}
\label{sec:jumps}

To state the algorithm, we first introduce a local change operation that generalizes the three kinds of flips introduced in Section~\ref{sec:flips} (recall Figure~\ref{fig:flips}) and that will be applied when moving from one rectangulation in~$\cC_n$ to the next in the algorithm.
A \emph{jump} changes the insertion point for exactly one rectangle of the rectangulation.
Formally, for a rectangulation~$R\in\cR_n$, we say that $R'\in\cR_n$ differs from~$R$ by a \emph{right jump of rectangle~$r_j$ by $d$ steps}, denoted $R'=\rvec{J}(R,j,d)$, where $2\leq j\leq n$ and $d>0$, if one of the following conditions holds; see Figure~\ref{fig:jumps-def}:
\begin{itemize}[leftmargin=5mm, noitemsep, topsep=3pt plus 3pt]
\item $j=n$, and we have $p(R)=p(R')=:P\in\cR_{n-1}$, $R=c_k(P)$ and $R'=c_{k+d}(P)$ for some $k>0$;
\item $j<n$, and $R$ and $R'$ are either both bottom-based or both right-based, and $p(R')$ differs from $p(R)$ in a right jump of rectangle~$r_j$ by $d$ steps.
\end{itemize}
In words, the first condition asserts that the first $n-1$ rectangles in~$R$ and~$R'$ form the same rectangulation~$P\in\cR_{n-1}$, and~$R$ and~$R'$ are obtained by insertion from~$P$ using the $k$th and $(k+d)$th insertion point, respectively.
The second condition asserts that~$R$ and~$R'$ agree in the rectangle~$r_n$, which either forms the bottom boundary or the right boundary of those rectangulations, and $p(R')$ differs from $p(R)$ in a right jump with the same parameters.

\begin{wrapfigure}{r}{0.4\textwidth}
\centering
\includegraphics[page=2]{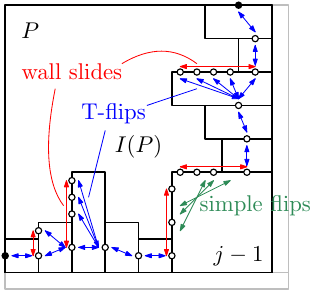}
\caption{Jumps generalize wall slides, simple flips and T-flips.}
\vspace{-2mm}
\label{fig:jumps-flips}
\end{wrapfigure}
A right jump as before is called \emph{minimal} w.r.t.\ to a set of rectangulations~$\cC_n\seq\cR_n$, if in the first condition above there is no index~$\ell$ with $k<\ell<k+d$ such that $c_\ell(P)\in\cC_n$.

A \emph{(minimal) left jump}, denoted $R'=\lvec{J}(R,j,d)$, is defined analogously by replacing $c_{k+d}$ by $c_{k-d}$ and $k<\ell<k+d$ by $k>\ell>k-d$ in the definitions above.
Clearly, if $R'$ differs from~$R$ by a right jump of rectangle~$r_j$ by $d$~steps, then $R$ differs from~$R'$ by a left jump of rectangle~$r_j$ by $d$~steps, and vice versa, i.e., we have $R'=\rvec{J}(R,j,d)$ if and only if $R=\lvec{J}(R',j,d)$.
We sometimes simply say that $R$ and $R'$ differ in a jump, without specifying the direction left or right.

We state the following simple observations for further reference; see Figure~\ref{fig:jumps-flips}.

\begin{lemma}
\label{lem:jumps-flips}
Consider two rectangulations $R,R'\in\cR_n$ that differ in a jump of rectangle~$r_j$, define $P:=R^{[j-1]}=R'^{[j-1]}\in\cR_{j-1}$, and let $q_k$ and $q_\ell$ be the insertion points in~$I(P)$ such that $R^{[j]}=c_k(P)$ and $R'^{[j]}=c_\ell(P)$.
\begin{enumerate}[label=(\alph*),leftmargin=7mm, noitemsep, topsep=3pt plus 3pt]
\item If $q_k$ and $q_\ell$ are consecutive (w.r.t.~$I(P)$) on a common wall of~$P$, then $R$ and $R'$ differ in a wall slide.
\item If $q_k$ lies on the last vertical wall and $q_\ell$ on the first horizontal wall of~$P$ (w.r.t.~$I(P)$), then $R$ and $R'$ differ in a simple flip.
\item If $q_k$ lies on a vertical wall and $q_\ell$ is the first insertion point on the next vertical wall of~$P$ (w.r.t.~$I(P)$), or if $q_k$ lies on a horizontal wall and $q_\ell$ is the last insertion point on the previous horizontal wall, then $R$ and $R'$ differ in a T-flip.
\end{enumerate}
\end{lemma}

For any rectangulation $R\in\cR_n$, we say that two insertion points from~$I(R)$ belong to the same~\emph{vertical or horizontal group}, if they lie on the same vertical or horizontal wall in~$R$, respectively.
In the sequence~$I(R)$, insertion points belonging to the same group appear consecutively.

\begin{figure}[h!]
\includegraphics[page=1]{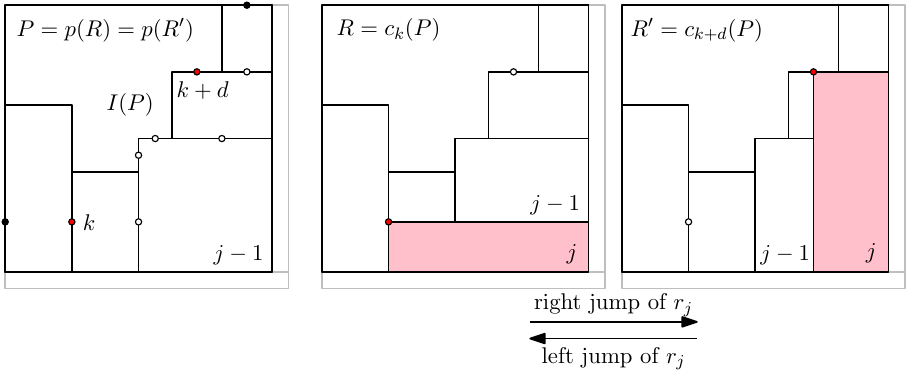}
\caption{Illustration of jumps.}
\label{fig:jumps-def}
\end{figure}

\subsection{Generating rectangulations by minimal jumps}

Consider the following algorithm that attempts to greedily generate a set of rectangulations $\cC_n\seq\cR_n$ using minimal jumps.

\begin{algo}{\bfseries Algorithm~\Jrect{}}{Greedy minimal jumps}
This algorithm attempts to greedily generate a set of rectangulations $\cC_n\subseteq \cR_n$ using minimal jumps starting from an initial rectangulation $R_0\in\cC_n$.
\begin{enumerate}[label={\bfseries J\arabic*.}, leftmargin=9mm, noitemsep, topsep=3pt plus 3pt]
\item{} [Initialize] Visit the initial rectangulation~$R_0$.
\item{} [Jump] Generate an unvisited rectangulation from~$\cC_n$ by performing a minimal jump of the rectangle with largest possible index in the most recently visited rectangulation.
If no such jump exists, or the jump direction is ambiguous, then terminate.
Otherwise visit this rectangulation and repeat~J2.
\end{enumerate}
\end{algo}

To illustrate how Algorithm~\Jrect{} works, we consider the set of five rectangulations $\cC_4=\{R_1,\ldots,R_5\}\seq\cR_4$ shown in Figure~\ref{fig:example}.
If initialized with $R_0:=R_1$, then the algorithm performs a left jump of rectangle~4 by one step (a right jump of rectangle~4 is impossible) to reach $R_2$, i.e., we have $R_2=\lvec{J}(R_1,4,1)$.
In~$R_2$, there are two options, either a right jump of rectangle~4 by one step, leading back to~$R_1$, which has been visited before, or a left jump of rectangle~4 by two steps, leading to~$R_3$, so we visit $R_3=\lvec{J}(R_2,4,2)$.
In~$R_3$, the jumps involving rectangle~4 lead to rectangulations that were visited before ($R_1$ and~$R_2$).
Moreover, a jump of rectangle~3 does not lead to a rectangulation in~$\cC_4$.
However, a right jump of rectangle~2 by one step leads to~$R_4$ (a left jump of rectangle~2 is impossible), so we visit $R_4=\rvec{J}(R_3,2,1)$.
Finally, in $R_4$ a right jump of rectangle~4 by two steps leads to~$R_5=\rvec{J}(R_4,4,2)$ (a left jump of rectangle~4 is impossible).
In this example, Algorithm~\Jrect{} successfully visits every rectangulation from~$\cC_4$ exactly once.

\begin{figure}[h!]
\includegraphics{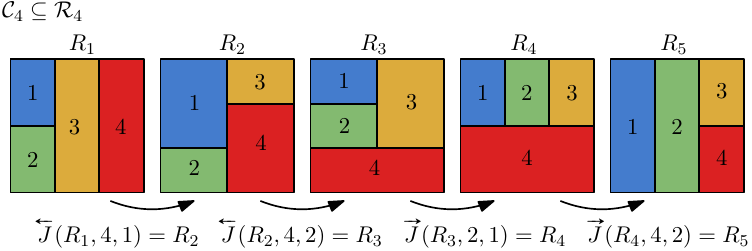}
\caption{Example execution of Algorithm~\Jrect{}.}
\label{fig:example}
\end{figure}

On the other hand, suppose we instead initialize the algorithm with~$R_0:=R_3$.
The algorithm will then visit $R_2:=\rvec{J}(R_3,4,2)$ followed by $R_1:=\rvec{J}(R_2,4,1)$, and then terminates without success, as from~$R_1$ no jump leads to an unvisited rectangulation from~$\cC_4$.
Lastly, suppose we initialize Algorithm~\Jrect{} with $R_0:=R_2$.
As before, in~$R_2$, there are two possibilities, either a right jump or a left jump of rectangle~4, both leading to an unvisited rectangulation from~$\cC_4$.
Both are minimal jumps in opposite directions, and as the jump direction is ambiguous, the algorithm terminates immediately without success.

\begin{remark}
\label{rem:inefficient}
\emph{We do not recommend using Algorithm~\Jrect{} in the stated form to generate a set of rectangulations efficiently!}
This is because the algorithm requires to maintain the list of all previously visited rectangulations (possibly exponentially many), and to look up this list in each step to check whether a rectangulation obtained by a jump from the current one has been visited before.
For us, Algorithm~\Jrect{} is merely a tool to define a Gray code ordering of the rectangulations in the given set~$\cC_n$ in way that is easy to remember (cf.~\cite{DBLP:conf/wads/Williams13}).
In fact, in Section~\ref{sec:efficient} we will present a modified algorithm that dispenses with the costly lookup operations, and that computes the very same sequence of rectangulations.
\end{remark}

\subsection{A guarantee for success}

By definition, Algorithm~\Jrect{} visits every rectangulation from a given set~$\cC_n\seq\cR_n$ at most once, but it may terminate before having visited all.
We now provide a sufficient condition guaranteeing that Algorithm~\Jrect{} visits every rectangulation from~$\cC_n$ exactly once.

A set of generic rectangulations $\cC_n\seq \cR_n$ is called \emph{zigzag}, if either $n=1$ and $\cC_1=\{\square\}$, or if $n\geq 2$ and $\cC_{n-1}:=\{p(R)\mid R\in\cC_n\}$ is zigzag and for every $R\in\cC_{n-1}$ we have $c_1(R)\in\cC_n$ and $c_{\nu(R)}(R)\in\cC_n$.
In words, the set~$\cC_n$ must be closed under repeatedly deleting bottom-right rectangles and replacing them by rectangles inserted either below or to the right of the remaining ones; recall Figure~\ref{fig:closure}.
The name `zigzag' does not refer to the shape of a rectangulation, but to the order in which they are visited by Algorithm~\Jrect{}, which will become clear momentarily.
We also say that $\cC_n$ is \emph{symmetric}, if reflection at the main diagonal is an involution of~$\cC_n$, i.e., if $R\in\cC_n$, then the rectangulation obtained from~$R$ by reflection at the main diagonal is also in~$\cC_n$.
We write~\idrect{} for the rectangulation that consists of $n$ vertically stacked rectangles.

\begin{theorem}
\label{thm:jump-rect}
Given any zigzag set of rectangulations~$\cC_n$ and initial rectangulation $R_0=\idrect$, Algorithm~\Jrect{} visits every rectangulation from~$\cC_n$ exactly once.
Moreover, if $\cC_n$ is symmetric, then the ordering of rectangulations generated by Algorithm~\Jrect{} is cyclic, i.e., the first and last rectangulation differ in a minimal jump.
\end{theorem}

The proof of Theorem~\ref{thm:jump-rect} is provided in Section~\ref{sec:proofs}.

Note that the rectangulation $R_0=\idrect$ is contained in every zigzag set by definition, so this is a valid initialization for Algorithm~\Jrect{}.
We write $\Jrectm(\cC_n)$ for the sequence of rectangulations generated by Algorithm~\Jrect{} for a zigzag set~$\cC_n$ when initialized with $R_0=\idrect$.

It is easy to see that the number of distinct zigzag sets of generic rectangulations is at least $2^{|\cR_n|(1-o(1))}\geq 2^{\Omega(11.56^n)}$ (the latter estimate uses the best known lower bound on~$|\cR_n|$ from~\cite{amano_nakano_yamanaka_2007}), i.e., at least double-exponential in~$n$.
In other words, Algorithm~\Jrect{} exhaustively generates a given set of generic rectangulations in a vast number of cases.
Moreover, many natural classes of rectangulations are in fact zigzag.
In particular, \emph{all} the different classes introduced in Section~\ref{sec:flips} and shown in Table~\ref{tab:rect} satisfy the aforementioned closure property.
Moreover, all of these classes are symmetric, so for each of them we obtain cyclic jump orderings.
Several such Gray codes are visualized in the appendix.

\subsection{Tree of rectangulations}
\label{sec:tree}

The notion of zigzag sets and the operation of Algorithm~\Jrect{} can be interpreted combinatorially in the so-called \emph{tree of rectangulations}, which is an infinite rooted tree, defined recursively as follows; see Figure~\ref{fig:tree}:
The root of the tree is a single rectangle $\square\in\cR_1$.
For any node $R\in\cR_{n-1}$, $n\geq 2$, of the tree we consider all insertion points of the rectangulation~$R$, and the set of children of~$R$ in the tree is $\{c_i(R)\in\cR_n\mid i=1,\ldots,\nu(R)\}$.
Conversely, the parent of each $R\in\cR_n$, $n\geq 2$, is $p(R)\in\cR_{n-1}$.
In words, insertion leads to the children of a node, and deletion leads to the parent of a node.
By Lemma~\ref{lem:del-ins}, each generic rectangulation appears exactly once in the tree, and the set of nodes at distance~$n$ from the root of the tree is precisely the set~$\cR_{n+1}$ of generic rectangulations with $n+1$ rectangles.
We emphasize that this tree is \emph{unordered}, i.e., there is no specified ordering among the children of a node.

By Lemma~\ref{lem:nu}, a node $R\in\cR_n$ in the tree has at most $n+1$ children, i.e., we have~$|\cR_n|\leq n!$.
As we see from Figure~\ref{fig:tree}, this inequality is tight up to $n=4$, but starting from $n=4$, there are nodes~$R\in\cR_n$ with strictly less than $n+1$ children, i.e., we have $|\cR_5|<5!$.
In fact, it was shown in~\cite{amano_nakano_yamanaka_2007} that $|\cR_n|=\cO(28.3^n)$.

A subset $\cC_n\seq \cR_n$ of nodes in depth~$n-1$ of this tree is zigzag, if and only if it arises from the full tree of rectangulations by pruning some subtrees whose roots are neither bottom-based nor right-based rectangulations.
In Figure~\ref{fig:tree}, all bottom-based or right-based rectangulations are highlighted by gray boxes, and can therefore not be pruned, while all other nodes can possibly be pruned.
If no nodes are pruned, then we have $\cC_n=\cR_n$, and if all possible nodes are pruned, then $\cC_n$ is the set~$\cB_n$ of $2^{n-1}$ rectangulations obtained by repeatedly stacking a new rectangle either below or to the right of the previous ones, i.e., $\cB_n=\{c_1(R),c_{\nu(R)}(R)\mid R\in\cB_{n-1}\}$ for $n\geq 2$ and $\cB_1=\{\square\}$.
Moreover, we have $\cB_n\seq\cC_n\seq\cR_n$ for any zigzag set~$\cC_n$.

\begin{figure}
\includegraphics[scale=0.7]{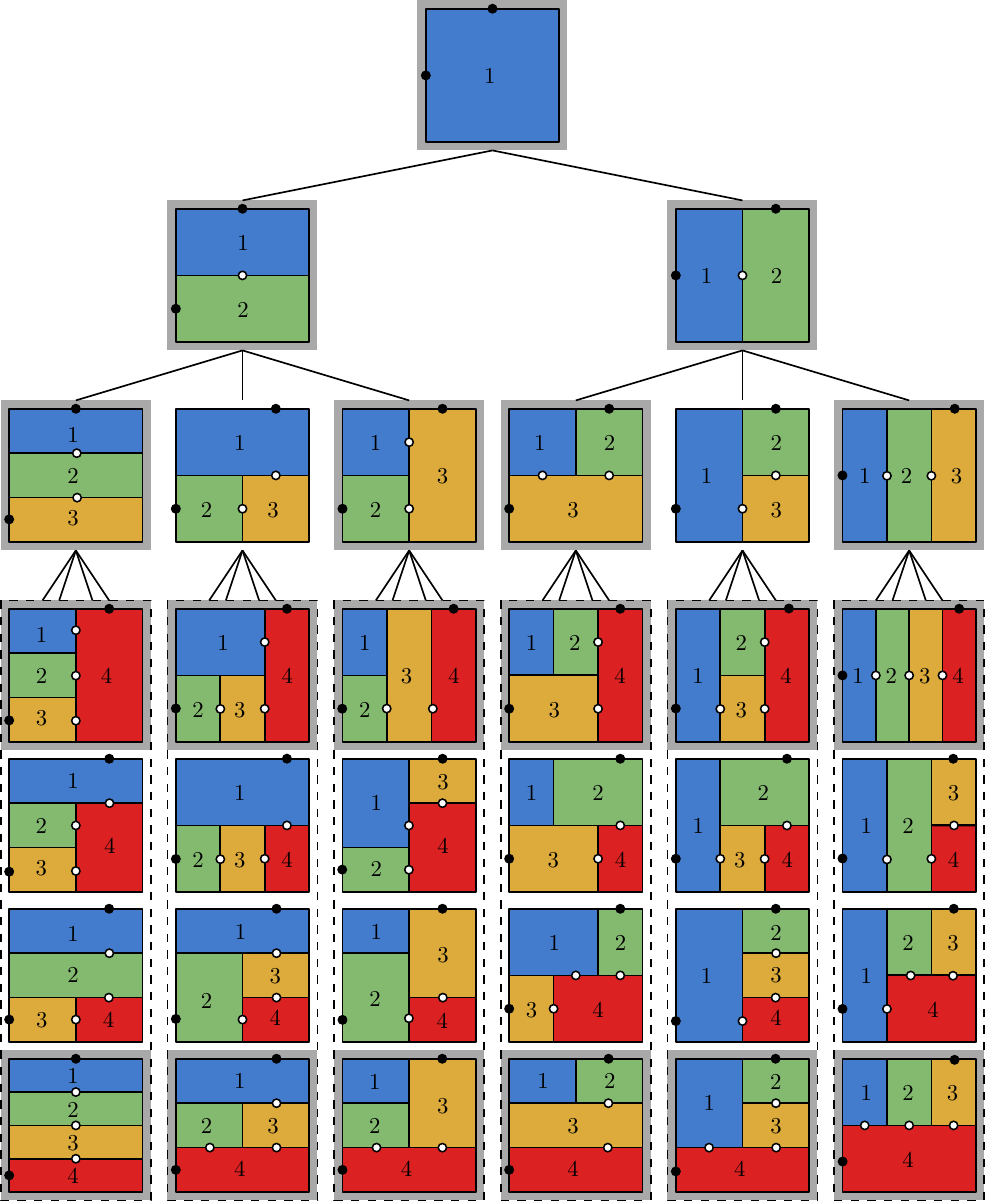}
\caption{Tree of generic rectangulations up to depth~3 with insertion points highlighted, where first and last insertion point are filled.
The rectangulations in the dashed boxes at the bottom level~$\cR_4$ are stacked on top of each other due to space constraints, but they are children of a common parent node.
Bottom- or right-based rectangulations, corresponding to insertion at the first or last insertion point, are marked by gray boxes.
}
\label{fig:tree}
\end{figure}

The operation of Algorithm~\Jrect{} for a zigzag set~$\cC_n$ as input can be interpreted as follows:
Given the pruned tree corresponding to~$\cC_n$, we consider the set of nodes on all previous levels of the tree, i.e., the sets $\cC_{i-1}:=\{p(R)\mid R\in\cC_i\}$ for $i=n,n-1,\ldots,2$, which are all zigzag sets by definition.
Moreover, we consider the orderings $\Jrectm(\cC_i)$, $i=1,\ldots,n$, defined by Algorithm~\Jrect{} for each of these sets.
These sequences turn the unordered tree corresponding to~$\cC_n$ into an ordered tree, where the children~$c_i(R)$ of each node~$R$ from left to right appear alternatingly in increasing order $i=1,\ldots,\nu(R)$ or in decreasing order $i=\nu(R),\nu(R)-1,\ldots,1$.
Consequently, in the sequence $\Jrectm(\cC_i)$, $i\geq 2$, which forms the left-to-right sequence of all nodes in depth~$i-1$ of this ordered tree, the rectangle~$r_i$ alternatingly jumps left and right between the first and last insertion point, which motivates the name `zigzag' set; see also the animations provided in~\cite{cos_rect}.

It is important to realize that these orderings are not consistent with respect to taking subsets, i.e., if we have two zigzag sets $\cC_n'\seq \cC_n$, then the entries of the sequence~$\Jrectm(\cC_n')$ do not necessarily appear in the same relative order in the sequence~$\Jrectm(\cC_n)$.

\section{Pattern-avoiding rectangulations}
\label{sec:avoidance}

In this section we show that Algorithm~\Jrect{} applies to a large number of rectangulation classes that are defined by pattern avoidance, under some very mild conditions on the patterns; recall Table~\ref{tab:rect}.

A \emph{rectangulation pattern} is a configuration of walls with prescribed directions and incidences.
For example, the windmill patterns~$\millr$ and~$\milll$ describe four walls such that when considering the walls in clockwise or counterclockwise order, respectively, the end vertex of one wall lies in the interior of the next wall.
We can also think of a pattern as the rectangulation formed by the given walls and incidences.
For example, we can think of the windmill patterns as rectangulations with~5 rectangles.
We say that a rectangulation~$R$ \emph{contains} the pattern~$P$, if $R$ contains a subset of walls with the directions and incidences specified by~$P$.
Otherwise we say that~$R$ \emph{avoids}~$P$.
For any set of rectangulation patterns~$\cP$ and for any set of rectangulations~$\cC$, we write~$\cC(\cP)$ for the rectangulations from~$\cC$ that avoid each pattern from~$\cP$.
For example, diagonal rectangulations are given by $\cD_n=\cR_n\big(\big\{\zvu,\zhr\big\}\big)$.
Examples of rectangulations containing and avoiding various patterns are shown in Figure~\ref{fig:classes}.

We say that a rectangulation pattern~$P$ is \emph{tame}, if for any rectangulation~$R$ that avoids~$P$, we also have that~$c_1(R)$ and~$c_{\nu(R)}(R)$ avoid~$P$.
In words, inserting a new rectangle below~$R$ or to the right of~$R$ must not create the forbidden pattern~$P$.
The next lemma follows directly from these definitions.

\begin{lemma}
\label{lem:tame}
If a rectangulation pattern is neither bottom-based nor right-based, then it is tame.
In particular, each of the patterns $\millr,\milll,\zvu,\zhr,\zvd,\zhl,\hvert,\hhor$ is tame.
\end{lemma}

The following powerful theorem allows to obtain many new zigzag sets of rectangulations from a given zigzag set~$\cC_n\seq\cR_n$ by forbidding one or more tame patterns.
All of these zigzag sets can then be generated by our Algorithm~\Jrect{}.

\begin{theorem}
\label{thm:pattern}
Let $\cC_n\seq \cR_n$ be a zigzag set of rectangulations, and let $\cP$ be a set of tame rectangulation patterns.
Then $\cC_n(\cP)$ is a zigzag set of rectangulations.
Moreover, if $\cP$ is symmetric, then $\cC_n(\cP)$ is symmetric.
\end{theorem}

Recall that $\cP$ is symmetric if for each pattern~$P\in\cP$, we have that the pattern obtained from~$P$ by reflection at the main diagonal is also in~$\cP$.
The significance of the second part of the theorem is that if $\cC_n(\cP)$ is symmetric, then the ordering of rectangulations of~$\cC_n(\cP)$ generated by Algorithm~\Jrect{} is cyclic by Theorem~\ref{thm:jump-rect}.

\begin{proof}
As $\cC_n$ is a zigzag set of rectangulations, we know that $\cC_{i-1}:=\{p(R)\mid R\in\cC_i\}$ for $i=n,n-1,\ldots,2$ are also zigzag sets.
We argue by induction that~$\cC_i(\cP)$ is also a zigzag set for all $i=1,\ldots,n$.
For the induction basis $i=1$ note that the rectangulation~$\square$ that consists of a single rectangle has no walls, so it avoids any pattern, showing that $\cC_1(\cP)=\cC_1=\{\square\}$.
For the induction step we assume that $\cC_i(\cP)$, $i\in\{1,\ldots,n-1\}$, is a zigzag set, and we prove it for~$\cC_{i+1}(\cP)$.
Note that $\{p(R)\mid R\in\cC_{i+1}(\cP)\}=\cC_i(\cP)$, and so we only need to check that~$c_1(R)$ and~$c_{\nu(R)}(R)$ are in~$\cC_{i+1}(\cP)$ for all $R\in\cC_i(\cP)$, which is guaranteed by the assumption that each pattern~$P\in\cP$ is tame.
This proves the first part of the theorem.

It remains to prove the second part.
If $R\in\cC_n(\cP)$, then $R$ avoids every pattern from~$\cP$.
Let $R'$ be the rectangulation obtained from~$R$ by reflection at the main diagonal.
$R'$ must also avoid every pattern from~$\cP$, because if it contained a pattern $P$ from~$\cP$, then $R$ would contain the corresponding reflected pattern~$P'$, which is in~$\cP$ because of the assumption that~$\cP$ is symmetric.
It follows that $R'\in\cC_n(\cP)$, completing the proof.
\end{proof}

\section{Efficient computation}
\label{sec:efficient}

Recall from Remark~\ref{rem:inefficient} that Algorithm~\Jrect{} in its stated form is unsuitable for efficient implementation.
We now discuss how to make the algorithm efficient, so as to achieve the time bounds claimed in Table~\ref{tab:rect} for several interesting classes of rectangulations.

\subsection{Memoryless algorithm}

Consider Algorithm~\Mrect{} below, which takes as input a zigzag set of rectangulations~$\cC_n\seq\cR_n$ and generates them exhaustively by minimal jumps in the same order as Algorithm~\Jrect{}, i.e., in the order~$\Jrectm(\cC_n)$.
After initialization in line~M1, the algorithm loops over lines M2--M5, visiting the current rectangulation~$R$ at the beginning of each iteration (line~M2), until it terminates (line~M3).

The key idea of the algorithm is to track explicitly which rectangle jumps in each step, and the direction of the jump.
With this information, the jump is determined by the condition that it must be minimal w.r.t.~$\cC_n$, i.e., starting from the current insertion point of the given rectangle, we choose the first insertion point (w.r.t.\ their linear ordering) for that rectangle in the given direction that creates the next rectangulation from~$\cC_n$.

\begin{algo}{\bfseries Algorithm~\Mrect{}}{Memoryless minimal jumps}
This algorithm generates all rectangulations of a zigzag set $\cC_n\seq\cR_n$ by minimal jumps in the same order as Algorithm~\Jrect{}.
It maintains the current rectangulation in the variable~$R$, and auxiliary arrays $o=(o_1,\ldots,o_n)$ and $s=(s_1,\ldots,s_n)$.
\begin{enumerate}[label={\bfseries M\arabic*.}, leftmargin=8mm, noitemsep, topsep=3pt plus 3pt]
\item{} [Initialize] Set $R\gets\idrect$, and $o_j\gets \dirl$, $s_j\gets j$ for $j=1,\ldots,n$.
\item{} [Visit] Visit the current rectangulation~$R$.
\item{} [Select rectangle] Set $j\gets s_n$, and terminate if $j=1$.
\item{} [Jump rectangle] In the current rectangulation~$R$, perform a jump of rectangle~$r_j$ that is minimal w.r.t.~$\cC_n$, where the jump direction is left if $o_j=\dirl$ and right if $o_j=\dirr$.
\item{} [Update $o$ and $s$] Set $s_n\gets n$.
If $o_j=\dirl$ and $R^{[j]}$ is bottom-based set $o_j\gets\dirr$, or if $o_j=\dirr$ and $R^{[j]}$ is right-based set $o_j\gets \dirl$, and in both cases set $s_j\gets s_{j-1}$ and $s_{j-1}\gets j-1$. Go back to~M2.
\end{enumerate}
\end{algo}

\begin{table}
\caption{Execution of Algorithm~\Mrect{} for the set~$\cC_4=\cD_4$ of diagonal rectangulations with 4~rectangles.
Empty entries in the $o$ and~$s$ column are unchanged compared to the previous row.}
\label{tab:algo}
\renewcommand{\arraystretch}{1.1}
\setlength\tabcolsep{3pt}
\begin{tabular}{>{\raggedleft}b{3mm}@{\hskip 2mm}c|c|c|cc>{\raggedleft}b{3mm}@{\hskip 2mm}c|c|c|c}
 & $\Jrectm(\cC_4)$ & jump & $o_1o_2o_3o_4$ & $s_1s_2s_3s_4$ & \hspace{2mm} &  & $\Jrectm(\cC_4)$ & jump & $o_1o_2o_3o_4$ & $s_1s_2s_3s_4$ \\ \cline{1-5}\cline{7-11}
1  & \raisebox{-6pt}{\begin{tikzpicture}[scale=24/100,font=\tiny]
\fill[p1col1] (0,4) rectangle (1,0);
\fill[p1col2] (1,4) rectangle (2,0);
\fill[p1col3] (2,4) rectangle (3,0);
\fill[p1col4] (3,4) rectangle (4,0);
\draw (0.50,3.50) node{1};
\draw (1.50,2.50) node{2};
\draw (2.50,1.50) node{3};
\draw (3.50,0.50) node{4};
\draw (0,4)--(4,0);
\end{tikzpicture}
} & $\lvec{J}(R,4,1)$ & \biv{$\dirl$}{$\dirl$}{$\dirl$}{$\dirl$} & \biv{1}{2}{3}{4} & & 12 & \raisebox{-6pt}{\begin{tikzpicture}[scale=24/100,font=\tiny]
\fill[p1col4] (0,1) rectangle (4,0);
\fill[p1col3] (0,2) rectangle (4,1);
\fill[p1col2] (0,3) rectangle (4,2);
\fill[p1col1] (0,4) rectangle (4,3);
\draw (0.50,3.50) node{1};
\draw (1.50,2.50) node{2};
\draw (2.50,1.50) node{3};
\draw (3.50,0.50) node{4};
\draw (0,4)--(4,0);
\end{tikzpicture}
} & $\rvec{J}(R,4,1)$ & \biv{}{$\dirr$}{}{} & \biv{1}{1}{}{4} \\
2  & \raisebox{-6pt}{\begin{tikzpicture}[scale=24/100,font=\tiny]
\fill[p1col1] (0,4) rectangle (1,0);
\fill[p1col2] (1,4) rectangle (2,0);
\fill[p1col4] (2,1) rectangle (4,0);
\fill[p1col3] (2,4) rectangle (4,1);
\draw (0.50,3.50) node{1};
\draw (1.50,2.50) node{2};
\draw (2.50,1.50) node{3};
\draw (3.50,0.50) node{4};
\draw (0,4)--(4,0);
\end{tikzpicture}
} & $\lvec{J}(R,4,1)$ & \biv{}{}{}{} & \biv{}{}{}{4} & & 13 & \raisebox{-6pt}{\begin{tikzpicture}[scale=24/100,font=\tiny]
\fill[p1col3] (0,2) rectangle (3,0);
\fill[p1col4] (3,2) rectangle (4,0);
\fill[p1col2] (0,3) rectangle (4,2);
\fill[p1col1] (0,4) rectangle (4,3);
\draw (0.50,3.50) node{1};
\draw (1.50,2.50) node{2};
\draw (2.50,1.50) node{3};
\draw (3.50,0.50) node{4};
\draw (0,4)--(4,0);
\end{tikzpicture}
} & $\rvec{J}(R,4,1)$ & \biv{}{}{}{} & \biv{}{}{}{4} \\
3  & \raisebox{-6pt}{\begin{tikzpicture}[scale=24/100,font=\tiny]
\fill[p1col1] (0,4) rectangle (1,0);
\fill[p1col4] (1,1) rectangle (4,0);
\fill[p1col2] (1,4) rectangle (2,1);
\fill[p1col3] (2,4) rectangle (4,1);
\draw (0.50,3.50) node{1};
\draw (1.50,2.50) node{2};
\draw (2.50,1.50) node{3};
\draw (3.50,0.50) node{4};
\draw (0,4)--(4,0);
\end{tikzpicture}
} & $\lvec{J}(R,4,1)$ & \biv{}{}{}{} & \biv{}{}{}{4} & & 14 & \raisebox{-6pt}{\begin{tikzpicture}[scale=24/100,font=\tiny]
\fill[p1col3] (0,2) rectangle (3,0);
\fill[p1col2] (0,3) rectangle (3,2);
\fill[p1col4] (3,3) rectangle (4,0);
\fill[p1col1] (0,4) rectangle (4,3);
\draw (0.50,3.50) node{1};
\draw (1.50,2.50) node{2};
\draw (2.50,1.50) node{3};
\draw (3.50,0.50) node{4};
\draw (0,4)--(4,0);
\end{tikzpicture}
} & $\rvec{J}(R,4,1)$ & \biv{}{}{}{} & \biv{}{}{}{4} \\
4  & \raisebox{-6pt}{\begin{tikzpicture}[scale=24/100,font=\tiny]
\fill[p1col4] (0,1) rectangle (4,0);
\fill[p1col1] (0,4) rectangle (1,1);
\fill[p1col2] (1,4) rectangle (2,1);
\fill[p1col3] (2,4) rectangle (4,1);
\draw (0.50,3.50) node{1};
\draw (1.50,2.50) node{2};
\draw (2.50,1.50) node{3};
\draw (3.50,0.50) node{4};
\draw (0,4)--(4,0);
\end{tikzpicture}
} & $\lvec{J}(R,3,1)$ & \biv{}{}{}{$\dirr$} & \biv{}{}{3}{3} & & 15 & \raisebox{-6pt}{\begin{tikzpicture}[scale=24/100,font=\tiny]
\fill[p1col3] (0,2) rectangle (3,0);
\fill[p1col2] (0,3) rectangle (3,2);
\fill[p1col1] (0,4) rectangle (3,3);
\fill[p1col4] (3,4) rectangle (4,0);
\draw (0.50,3.50) node{1};
\draw (1.50,2.50) node{2};
\draw (2.50,1.50) node{3};
\draw (3.50,0.50) node{4};
\draw (0,4)--(4,0);
\end{tikzpicture}
} & $\rvec{J}(R,3,1)$ & \biv{}{}{}{$\dirl$} & \biv{}{}{3}{3} \\
5  & \raisebox{-6pt}{\begin{tikzpicture}[scale=24/100,font=\tiny]
\fill[p1col4] (0,1) rectangle (4,0);
\fill[p1col1] (0,4) rectangle (1,1);
\fill[p1col3] (1,2) rectangle (4,1);
\fill[p1col2] (1,4) rectangle (4,2);
\draw (0.50,3.50) node{1};
\draw (1.50,2.50) node{2};
\draw (2.50,1.50) node{3};
\draw (3.50,0.50) node{4};
\draw (0,4)--(4,0);
\end{tikzpicture}
} & $\rvec{J}(R,4,1)$ & \biv{}{}{}{} & \biv{}{}{}{4} & & 16 & \raisebox{-6pt}{\begin{tikzpicture}[scale=24/100,font=\tiny]
\fill[p1col2] (0,3) rectangle (2,0);
\fill[p1col3] (2,3) rectangle (3,0);
\fill[p1col1] (0,4) rectangle (3,3);
\fill[p1col4] (3,4) rectangle (4,0);
\draw (0.50,3.50) node{1};
\draw (1.50,2.50) node{2};
\draw (2.50,1.50) node{3};
\draw (3.50,0.50) node{4};
\draw (0,4)--(4,0);
\end{tikzpicture}
} & $\lvec{J}(R,4,1)$ & \biv{}{}{}{} & \biv{}{}{}{4} \\
6  & \raisebox{-6pt}{\begin{tikzpicture}[scale=24/100,font=\tiny]
\fill[p1col1] (0,4) rectangle (1,0);
\fill[p1col4] (1,1) rectangle (4,0);
\fill[p1col3] (1,2) rectangle (4,1);
\fill[p1col2] (1,4) rectangle (4,2);
\draw (0.50,3.50) node{1};
\draw (1.50,2.50) node{2};
\draw (2.50,1.50) node{3};
\draw (3.50,0.50) node{4};
\draw (0,4)--(4,0);
\end{tikzpicture}
} & $\rvec{J}(R,4,1)$ & \biv{}{}{}{} & \biv{}{}{}{4} & & 17 & \raisebox{-6pt}{\begin{tikzpicture}[scale=24/100,font=\tiny]
\fill[p1col2] (0,3) rectangle (2,0);
\fill[p1col3] (2,3) rectangle (3,0);
\fill[p1col4] (3,3) rectangle (4,0);
\fill[p1col1] (0,4) rectangle (4,3);
\draw (0.50,3.50) node{1};
\draw (1.50,2.50) node{2};
\draw (2.50,1.50) node{3};
\draw (3.50,0.50) node{4};
\draw (0,4)--(4,0);
\end{tikzpicture}
} & $\lvec{J}(R,4,1)$ & \biv{}{}{}{} & \biv{}{}{}{4} \\
7  & \raisebox{-6pt}{\begin{tikzpicture}[scale=24/100,font=\tiny]
\fill[p1col1] (0,4) rectangle (1,0);
\fill[p1col3] (1,2) rectangle (3,0);
\fill[p1col4] (3,2) rectangle (4,0);
\fill[p1col2] (1,4) rectangle (4,2);
\draw (0.50,3.50) node{1};
\draw (1.50,2.50) node{2};
\draw (2.50,1.50) node{3};
\draw (3.50,0.50) node{4};
\draw (0,4)--(4,0);
\end{tikzpicture}
} & $\rvec{J}(R,4,1)$ & \biv{}{}{}{} & \biv{}{}{}{4} & & 18 & \raisebox{-6pt}{\begin{tikzpicture}[scale=24/100,font=\tiny]
\fill[p1col2] (0,3) rectangle (2,0);
\fill[p1col4] (2,1) rectangle (4,0);
\fill[p1col3] (2,3) rectangle (4,1);
\fill[p1col1] (0,4) rectangle (4,3);
\draw (0.50,3.50) node{1};
\draw (1.50,2.50) node{2};
\draw (2.50,1.50) node{3};
\draw (3.50,0.50) node{4};
\draw (0,4)--(4,0);
\end{tikzpicture}
} & $\lvec{J}(R,4,1)$ & \biv{}{}{}{} & \biv{}{}{}{4} \\
8  & \raisebox{-6pt}{\begin{tikzpicture}[scale=24/100,font=\tiny]
\fill[p1col1] (0,4) rectangle (1,0);
\fill[p1col3] (1,2) rectangle (3,0);
\fill[p1col2] (1,4) rectangle (3,2);
\fill[p1col4] (3,4) rectangle (4,0);
\draw (0.50,3.50) node{1};
\draw (1.50,2.50) node{2};
\draw (2.50,1.50) node{3};
\draw (3.50,0.50) node{4};
\draw (0,4)--(4,0);
\end{tikzpicture}
} & $\lvec{J}(R,3,1)$ & \biv{}{}{}{$\dirl$} & \biv{}{}{3}{3} & & 19 & \raisebox{-6pt}{\begin{tikzpicture}[scale=24/100,font=\tiny]
\fill[p1col4] (0,1) rectangle (4,0);
\fill[p1col2] (0,3) rectangle (2,1);
\fill[p1col3] (2,3) rectangle (4,1);
\fill[p1col1] (0,4) rectangle (4,3);
\draw (0.50,3.50) node{1};
\draw (1.50,2.50) node{2};
\draw (2.50,1.50) node{3};
\draw (3.50,0.50) node{4};
\draw (0,4)--(4,0);
\end{tikzpicture}
} & $\rvec{J}(R,3,1)$ & \biv{}{}{}{$\dirr$} & \biv{}{}{3}{3} \\
9  & \raisebox{-6pt}{\begin{tikzpicture}[scale=24/100,font=\tiny]
\fill[p1col3] (0,2) rectangle (3,0);
\fill[p1col1] (0,4) rectangle (1,2);
\fill[p1col2] (1,4) rectangle (3,2);
\fill[p1col4] (3,4) rectangle (4,0);
\draw (0.50,3.50) node{1};
\draw (1.50,2.50) node{2};
\draw (2.50,1.50) node{3};
\draw (3.50,0.50) node{4};
\draw (0,4)--(4,0);
\end{tikzpicture}
} & $\lvec{J}(R,4,1)$ & \biv{}{}{$\dirr$}{} & \biv{}{2}{2}{4} & & 20 & \raisebox{-6pt}{\begin{tikzpicture}[scale=24/100,font=\tiny]
\fill[p1col4] (0,1) rectangle (4,0);
\fill[p1col2] (0,3) rectangle (2,1);
\fill[p1col1] (0,4) rectangle (2,3);
\fill[p1col3] (2,4) rectangle (4,1);
\draw (0.50,3.50) node{1};
\draw (1.50,2.50) node{2};
\draw (2.50,1.50) node{3};
\draw (3.50,0.50) node{4};
\draw (0,4)--(4,0);
\end{tikzpicture}
} & $\rvec{J}(R,4,1)$ & \biv{}{}{$\dirl$}{} & \biv{}{2}{1}{4} \\
10 & \raisebox{-6pt}{\begin{tikzpicture}[scale=24/100,font=\tiny]
\fill[p1col3] (0,2) rectangle (3,0);
\fill[p1col1] (0,4) rectangle (1,2);
\fill[p1col4] (3,2) rectangle (4,0);
\fill[p1col2] (1,4) rectangle (4,2);
\draw (0.50,3.50) node{1};
\draw (1.50,2.50) node{2};
\draw (2.50,1.50) node{3};
\draw (3.50,0.50) node{4};
\draw (0,4)--(4,0);
\end{tikzpicture}
} & $\lvec{J}(R,4,2)$ & \biv{}{}{}{} & \biv{}{}{}{4} & & 21 & \raisebox{-6pt}{\begin{tikzpicture}[scale=24/100,font=\tiny]
\fill[p1col2] (0,3) rectangle (2,0);
\fill[p1col1] (0,4) rectangle (2,3);
\fill[p1col4] (2,1) rectangle (4,0);
\fill[p1col3] (2,4) rectangle (4,1);
\draw (0.50,3.50) node{1};
\draw (1.50,2.50) node{2};
\draw (2.50,1.50) node{3};
\draw (3.50,0.50) node{4};
\draw (0,4)--(4,0);
\end{tikzpicture}
} & $\rvec{J}(R,4,2)$ & \biv{}{}{}{} & \biv{}{}{}{4} \\
11 & \raisebox{-6pt}{\begin{tikzpicture}[scale=24/100,font=\tiny]
\fill[p1col4] (0,1) rectangle (4,0);
\fill[p1col3] (0,2) rectangle (4,1);
\fill[p1col1] (0,4) rectangle (1,2);
\fill[p1col2] (1,4) rectangle (4,2);
\draw (0.50,3.50) node{1};
\draw (1.50,2.50) node{2};
\draw (2.50,1.50) node{3};
\draw (3.50,0.50) node{4};
\draw (0,4)--(4,0);
\end{tikzpicture}
} & $\lvec{J}(R,2,1)$ & \biv{}{}{}{$\dirr$} & \biv{}{}{3}{2} & & 22 & \raisebox{-6pt}{\begin{tikzpicture}[scale=24/100,font=\tiny]
\fill[p1col2] (0,3) rectangle (2,0);
\fill[p1col1] (0,4) rectangle (2,3);
\fill[p1col3] (2,4) rectangle (3,0);
\fill[p1col4] (3,4) rectangle (4,0);
\draw (0.50,3.50) node{1};
\draw (1.50,2.50) node{2};
\draw (2.50,1.50) node{3};
\draw (3.50,0.50) node{4};
\draw (0,4)--(4,0);
\end{tikzpicture}
} & & \biv{}{}{}{$\dirl$} & \biv{}{}{3}{1} \\
\end{tabular}
\end{table}

Specifically, the jump directions are maintained by an array $o=(o_1,\ldots,o_n)$, where $o_j=\dirl$ means that rectangle~$r_j$ performs a left jump in the next step, and $o_j=\dirr$ means that rectangle~$r_j$ performs a right jump in the next step (line~M4).
All sub-rectangulations of the initial rectangulation~$\idrect$ are right-based, so the initial jump directions are $o_j=\dirl$ for $j=1,\ldots,n$ (line~M1).
Whenever rectangle~$r_j$ jumps left and reaches the first insertion point, which means that $R^{[j]}$ is bottom-based, or if it jumps right and reaches the last insertion point, which means that $R^{[j]}$ is right-based, then the jump direction~$o_j$ is reversed (line~M5).

The array $s=(s_1,\ldots,s_n)$ is used to determine which rectangle jumps in each step.
Specifically, the last entry~$s_n$ determines the rectangle that jumps in the current iteration (line~M3).
This array simulates a stack in a loopless fashion, following an idea first used by Bitner, Ehrlich, and Reingold~\cite{MR0424386}.
The stack is initialized by $(s_1,\ldots,s_n)=(1,2,\ldots,n)$ (line~M1), with $s_n$ being the value on the top of the stack.
The stack is popped (by the instruction $s_j\gets s_{j-1}$ in line~M5) when rectangle~$r_j$ reaches its first or last insertion point in this step, meaning that this rectangle is not eligible to jump in the next step, but becomes eligible again after the next step, which is achieved by pushing the value~$j$ on the stack again (by the instructions $s_n\gets n$ and $s_{j-1}\gets j-1$ in line~M5).

Table~\ref{tab:algo} shows the execution of Algorithm~\Mrect{} with input~$\cC_4=\cD_4$ being the set of all diagonal rectangulations with 4~rectangles.

\begin{theorem}
\label{thm:algo-rect}
For any zigzag set of rectangulations $\cC_n\seq\cR_n$, Algorithm~\Mrect{} visits every rectangulation from~$\cC_n$ exactly once, in the order $\Jrectm(\cC_n)$ defined by Algorithm~\Jrect{}.
\end{theorem}

The proof of Theorem~\ref{thm:algo-rect} is provided in Section~\ref{sec:proofs}.

To make meaningful statements about the running time of Algorithm~\Mrect{}, we need to specify the data structures used to represent the current rectangulation~$R$, and the operations on this data structure to perform jumps in line~M4 and to check the bottom-based and right-based property in line~M5.
Most importantly, we will develop oracles which efficiently compute the next minimal jump w.r.t.~$\cC_n$ for some interesting zigzag sets~$\cC_n$.
One should think of~$\cC_n$ here as a class of rectangulations specified by some properties or forbidden patterns, such as `diagonal guillotine rectangulations', and not as a large precomputed set of rectangulations.
All of these details are described in the following sections, and they are part of our C++ implementation provided in~\cite{cos_rect}.

\section{Implementation details}
\label{sec:implementation}

In the following we describe the data structures we use to represent and manipulate generic rectangulations, and the efficient implementation of jump operations using those data structures.

\subsection{Data structures}
\label{sec:data}

We represent a generic rectangulation with $n$ rectangles as follows; see Figure~\ref{fig:data}:
Rectangles are stored in the variables $r_1,\ldots,r_n$, indexed by the reverse deletion order described in Section~\ref{sec:deletion} (recall Figure~\ref{fig:numbering}).
Vertices and edges are stored in variables $v_1,\dots,v_{2n+2}$ and $e_1,\ldots,e_{3n+1}$, respectively (indexed in no particular order).

\begin{figure}
\setlength{\tabcolsep}{2pt}
\renewcommand{\arraystretch}{0.9}
\includegraphics{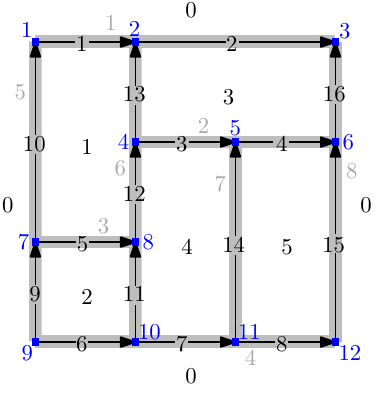}
\begin{tabular}{cccc}
\parbox{5cm}{
\begin{tabular}[b]{c|llllllll}
$e_i$ & \rotatebox[origin=l]{90}{\dir} & \rotatebox[origin=l]{90}{\etail} & \rotatebox[origin=l]{90}{\ehead} & \rotatebox[origin=l]{90}{\eprev} & \rotatebox[origin=l]{90}{\enext} & \rotatebox[origin=l]{90}{\eleft} & \rotatebox[origin=l]{90}{\eright} & \rotatebox[origin=l]{90}{\ewall} \\ \hline
1 & $\dirr$ & 1 & 2 & 0 & 2 & 0 & 1 & 1 \\
2 & $\dirr$ & 2 & 3 & 1 & 0 & 0 & 3 & 1 \\
3 & $\dirr$ & 4 & 5 & 0 & 4 & 3 & 4 & 2 \\
4 & $\dirr$ & 5 & 6 & 3 & 0 & 3 & 5 & 2 \\
5 & $\dirr$ & 7 & 8 & 0 & 0 & 1 & 2 & 3 \\
6 & $\dirr$ & 9 & 10 & 0 & 7 & 2 & 0 & 4 \\
7 & $\dirr$ & 10 & 11 & 6 & 8 & 4 & 0 & 4 \\
8 & $\dirr$ & 11 & 12 & 7 & 0 & 5 & 0 & 4 \\
9 & $\diru$ & 9 & 7 & 0 & 10 & 0 & 2 & 5 \\
10 & $\diru$ & 7 & 1 & 9 & 0 & 0 & 1 & 5 \\
11 & $\diru$ & 10 & 8 & 0 & 12 & 2 & 4 & 6 \\
12 & $\diru$ & 8 & 4 & 11 & 13 & 1 & 4 & 6 \\
13 & $\diru$ & 4 & 2 & 12 & 0 & 1 & 3 & 6 \\
14 & $\diru$ & 11 & 5 & 0 & 0 & 4 & 5 & 7 \\
15 & $\diru$ & 12 & 6 & 0 & 16 & 5 & 0 & 8 \\
16 & $\diru$ & 6 & 3 & 15 & 0 & 3 & 0 & 8 \\
\end{tabular}
}
&
\parbox{3.5cm}{
\begin{tabular}[b]{c|lllll}
$v_i$ & \rotatebox[origin=l]{90}{\vnorth} & \rotatebox[origin=l]{90}{\veast} & \rotatebox[origin=l]{90}{\vsouth} & \rotatebox[origin=l]{90}{\vwest} & \rotatebox[origin=l]{90}{\vtype} \\ \hline
1 & 0 & 1 & 10 & 0 & 0 \\
2 & 0 & 2 & 13 & 1 & $\bottomT$ \\
3 & 0 & 0 & 16 & 2 & 0 \\
4 & 13 & 3 & 12 & 0 & $\rightT$ \\
5 & 0 & 4 & 14 & 3 & $\bottomT$ \\
6 & 16 & 0 & 15 & 4 & $\leftT$ \\
7 & 10 & 5 & 9 & 0 & $\rightT$ \\
8 & 12 & 0 & 11 & 5 & $\leftT$ \\
9 & 9 & 6 & 0 & 0 & 0 \\
10 & 11 & 7 & 0 & 6 & $\topT$ \\
11 & 14 & 8 & 0 & 7 & $\topT$ \\
12 & 15 & 0 & 0 & 8 & 0 \\
\end{tabular}
}
&
\parbox{2.7cm}{
\begin{tabular}[b]{c|ll}
$w_i$ & \rotatebox[origin=l]{90}{\wfirst} & \rotatebox[origin=l]{90}{\wlast} \\ \hline
1 & 1 & 3 \\
2 & 4 & 6 \\
3 & 7 & 8 \\
4 & 9 & 12 \\
5 & 9 & 1 \\
6 & 10 & 2 \\
7 & 11 & 5 \\
8 & 12 & 3 \\
\end{tabular}
}
&
\parbox{3cm}{
\begin{tabular}[b]{c|lllll}
$r_i$ & \rotatebox[origin=l]{90}{\rne} & \rotatebox[origin=l]{90}{\rse} & \rotatebox[origin=l]{90}{\rsw} & \rotatebox[origin=l]{90}{\rnw} \\ \hline
1 & 2 & 8 & 7 & 1 \\
2 & 8 & 10 & 9 & 7 \\
3 & 3 & 6 & 4 & 2 \\
4 & 5 & 11 & 10 & 4 \\
5 & 6 & 12 & 11 & 5 \\
\end{tabular}
}
\end{tabular}
\caption{Data structures used to represent a generic rectangulation with $n=5$ rectangles.
Edges are labeled black, and walls are labeled gray.}
\label{fig:data}
\end{figure}

Each vertex~$v$ points to the edges incident to it in the four directions by $v.\vnorth$, $v.\veast$, $v.\vsouth$, and $v.\vwest$.
Some of these can be~0, indicating that no edge is incident.
This information determines the type~$v.\vtype$, which is one of $\bottomT$, $\rightT$, $\topT$, $\leftT$, or 0 at the corner vertices of the rectangulation.
We give all edges a default orientation from left to right, or from bottom to top.
The $\dir$ entry of each edge~$e$ specifies its direction, which is either $e.\dir=\dirr$ for a horizontal edge or $e.\dir=\diru$ for a vertical edge.
Each edge~$e$ points to its two end vertices, specifically to its tail by~$e.\etail$ and to its head by~$e.\ehead$ (with respect to the default orientation).
It also points to the previous and next edge, in the direction of its orientation, by $e.\eprev$ and $e.\enext$, respectively, which can be~0 if no such edge exists.
The rectangle to the left and right side of an edge~$e$, in the direction of its orientation, are stored in $e.\eleft$ and $e.\eright$, which can be~0 at the boundary of the rectangulation.
Each rectangle~$r$ points to its four corner vertices by $r.\rne$, $r.\rse$, $r.\rsw$, and $r.\rnw$ in the corresponding directions.

For some rectangulation classes it is useful to store information about walls, i.e., maximal sequences of edges between two vertices that are not corners of the rectangulation.
These are stored in $w_1,\ldots,w_{n+3}$, where for simplicity we also keep track of the four maximal line segments between corners of the rectangulation (which are not walls in our definition).
We also think of walls having a default orientation from left to right, or from bottom to top, and each wall~$w$ points to its first and last vertex by~$w.\wfirst$ and $w.\wlast$, respectively, in the direction of its orientation.
Moreover, each edge~$e$ has an entry $e.\ewall$ pointing to the wall that contains it.

\begin{remark}
The aforementioned data structures are natural in the sense that they also capture the dual graph of the rectangulation, i.e., the graph obtained by replacing every rectangle by a vertex, and by joining any two vertices that correspond to rectangles sharing a common edge.
This allows constructing the so-called \emph{transversal structure}~\cite{MR2509359} (also known as \emph{regular edge labeling}~\cite{MR1432861}), which is useful for computing a layout of the rectangulation; see Felsner's survey~\cite{MR3205156}.
Our data structures also allow to easily extract the twin binary tree representation of diagonal rectangulations described in~\cite{DBLP:journals/todaes/YaoCCG03}.
\end{remark}

We now use these data structures for implementing jumps efficiently.
Recall the conditions stated in Lemma~\ref{lem:jumps-flips} when a jump is one of the three flip operations shown in Figure~\ref{fig:flips}.
We refer to a jump as in~(a), (b) or~(c) in the lemma as a \emph{W-jump}, \emph{S-jump}, or \emph{T-jump}, respectively.
By these definitions, a W-jump is a special wall slide, an S-jump is a special simple flip, and a T-jump is a special T-flip.
We refer to W-, S- and T-jump as \emph{local} jumps collectively.
Moreover, a W-jump or T-jump between two horizontal insertion points, or between two vertical insertion points, is referred to as a \emph{horizontal or vertical W- or T-jump}, respectively.

Consider two rectangulations~$R$ and~$R'$ that differ in a jump of rectangle~$r_j$.
If the jump is an S-jump or a horizontal T-jump, we let $h(R,R')$ denote the number of horizontal insertion points of $I(R^{[j-1]})=I(R'^{[j-1]})$ that lie in the interior of the top side of rectangle~$r_j$ in~$R$ or~$R'$; see Figure~\ref{fig:jump-time}.
Similarly, if the jump is an S-jump or a vertical T-jump, we let $v(R,R')$ denote the number of vertical insertion points of $I(R^{[j-1]})=I(R'^{[j-1]})$ that lie in the interior of the left side of rectangle~$r_j$ in~$R$ or~$R'$.

\begin{figure}
\includegraphics[width=\textwidth]{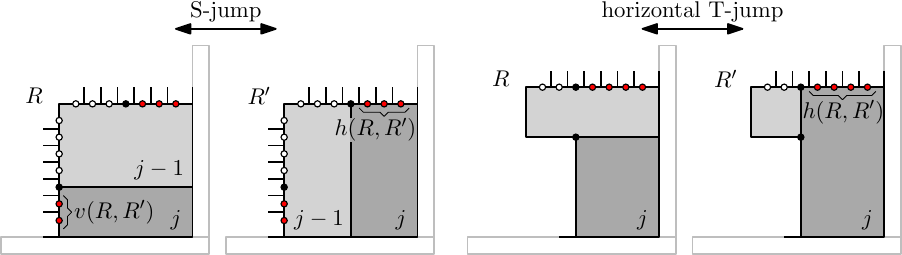}
\caption{Illustration of Lemma~\ref{lem:jump-time}.}
\label{fig:jump-time}
\end{figure}

\begin{lemma}
\label{lem:jump-time}
Local jumps can be implemented with the following time guarantees:
\begin{enumerate}[label=(\alph*),leftmargin=7mm, noitemsep, topsep=3pt plus 3pt]
\item A W-jump takes time~$\cO(1)$.
\item An S-jump between rectangulations $R$ and $R'$ takes time $\cO(h(R,R')+v(R,R')+1)$.
\item A horizontal T-jump between rectangulations $R$ and $R'$ takes time $\cO(h(R,R')+1)$ and a vertical T-jump takes time $\cO(v(R,R')+1)$.
\end{enumerate}
\end{lemma}
Clearly, every jump can be performed as a sequence of local jumps, and then the time bounds given by Lemma~\ref{lem:jump-time} can be added up.

\begin{proof}
The time bounds follow from the number of incidences that change during local jumps.
The crucial point is that during a jump of rectangle~$r_j$ between~$R$ and~$R'$, all rectangles~$r_k$ for $k=j+1,\ldots,n$ are right-based or bottom-based in~$R^{[k]}$ and~$R'^{[k]}$, entailing that only constantly many incidences with such rectangles have to be modified.
\end{proof}

\subsection{Auxiliary functions}

In Section~\ref{sec:jump-impl} we provide implementations of local jumps with the runtime guarantees stated in Lemma~\ref{lem:jump-time}.
Before doing so, we introduce some auxiliary functions to add and remove edges from a rectangulation.
These auxiliary functions only update the incidences between edges, vertices and walls, but not the incidences between rectangles and any other objects and the type of vertices (this will be done separately later).

The following function $\eremoveh(\beta)$ removes the edge~$e_\beta$ together with its head vertex.

\begin{algo}{{\bfseries Function} $\eremoveh(\beta)$}{Remove $e_\beta$ and its head}
\begin{enumerate}[label={\bfseries \arabic*.}, leftmargin=6mm, noitemsep, topsep=3pt plus 3pt]
\item{}[Prepare] Set $\alpha\gets e_\beta.\eprev$, $\gamma\gets e_\beta.\enext$, $a\gets e_\beta.\etail$.
\item{}[Update edges/vertices]
If $\alpha\neq 0$, set $e_\alpha.\enext\gets \gamma$.
If $\gamma\neq 0$, set $e_\gamma.\eprev\gets \alpha$ and $e_\gamma.\etail\gets a$.
If $e_\beta.\dir=\dirr$, set $v_a.\veast\gets \gamma$.
Otherwise we have $e_\beta.\dir=\diru$ and set $v_a.\vnorth\gets \gamma$.
\item{}[Update wall] Set $x\gets e_\beta.\ewall$.
If $e_\beta.\ehead = w_x.\wlast$, set $w_x.\wlast\gets a$.
\end{enumerate}
\end{algo}
After defining some auxiliary variables in the first step, the function $\eremoveh(\beta)$ updates the incidences between edges and vertices in the second step, and the incidences between walls and edges in the third step.
We also define an analogous function $\eremovet(\beta)$ that removes~$e_\beta$ and its tail instead of its head.
For details, see our C++ implementation~\cite{cos_rect}.

The following two functions $\einsertb(\beta,a,\gamma)$ and $\einserta(\alpha,a,\beta)$ insert the edge $e_\beta$ with head~$v_a$ or tail~$v_a$, respectively, before or after the edge~$e_\gamma$ or $e_\alpha$.

\begin{algo}{{\bfseries Function} $\einsertb(\beta,a,\gamma)$}{Insertion of $e_\beta$ with head~$v_a$ before $e_\gamma$}
\begin{enumerate}[label={\bfseries \arabic*.}, leftmargin=6mm, noitemsep, topsep=3pt plus 3pt]
\item{}[Prepare] Set $\alpha\gets e_\gamma.\eprev$ and $b\gets e_\gamma.\etail$.
\item{}[Update edges/vertices] Set $e_\beta.\etail \gets b$, $e_\beta.\ehead \gets a$, $e_\beta.\eprev\gets\alpha$, $e_\beta.\enext\gets\gamma$, $e_\gamma.\etail \gets a$, $e_\gamma.\eprev\gets\beta$, and if $\alpha\neq 0$ set $e_\alpha.\enext\gets\beta$.
If $e_\gamma.\dir=\dirr$, set $e_\beta.\dir\gets\dirr$, $v_a.\vwest \gets \beta$, $v_a.\veast \gets \gamma$ and $v_b.\veast \gets \beta$.
Otherwise we have $e_\gamma.\dir=\diru$ and set $e_\beta.\dir \gets \diru$, $v_a.\vsouth\gets \beta$, $v_a.\vnorth \gets \gamma$ and $v_b.\vnorth \gets \beta$.
\item{}[Update wall] Set $e_\beta.\ewall\gets e_\gamma.\ewall$.
\end{enumerate}
\end{algo}

\begin{algo}{{\bfseries Function} $\einserta(\alpha,a,\beta)$}{Insertion of $e_\beta$ with tail~$v_a$ after $e_\alpha$}
\begin{enumerate}[label={\bfseries \arabic*.}, leftmargin=6mm, noitemsep, topsep=3pt plus 3pt]
\item{}[Prepare] Set $\gamma\gets e_\alpha.\enext$ and $b\gets e_\alpha.\ehead$.
\item{}[Update edges/vertices] Set $e_\beta.\etail \gets a$, $e_\beta.\ehead \gets b$, $e_\beta.\eprev\gets\alpha$, $e_\beta.\enext\gets\gamma$, $e_\alpha.\ehead\gets a$, $e_\alpha.\enext\gets\beta$, and if $\gamma\neq 0$ set $e_\gamma.\eprev\gets\beta$.
If $e_\alpha.\dir=\dirr$, set $e_\beta.\dir \gets \dirr$, $v_a.\vwest \gets \alpha$, $v_a.\veast \gets \beta$ and $v_b.\vwest \gets \beta$.
Otherwise we have $e_\alpha.\dir=\diru$ and set $e_\beta.\dir \gets \diru$, $v_a.\vsouth \gets \alpha$, $v_a.\vnorth \gets \beta$ and $v_b.\vsouth \gets \beta$.
\item{}[Update wall] Set $e_\beta.\ewall\gets e_\alpha.\ewall$.
\end{enumerate}
\end{algo}

\subsection{Local jumps}
\label{sec:jump-impl}

Armed with these auxiliary functions, we now tackle the implementation of local jumps with the time guarantees stated in Lemma~\ref{lem:jump-time}.
Each of the functions $\Wjump(R,j,d,\alpha)$, $\Sjump(R,j,d,\alpha)$ and $\Tjump(R,j,d,\alpha)$ below takes as input the current rectangulation~$R$ in which the jump is performed, the index~$j$ of the rectangle~$r_j$ to be jumped, the direction $d \in \{\dirl,\dirr\}$ of the jump, and the index~$\alpha$ of the edge $e_\alpha$ which contains the insertion point of~$R^{[j-1]}$ that will become the top-left vertex of~$r_j$ after the jump.
In the pseudocode of these algorithms, all references to rectangles~$r_i$, edges~$e_i$, vertices~$v_i$ or walls~$w_i$ are with respect to the current rectangulation~$R$.

We first present the implementation of W-jumps.
For simplicity, we only show the implementation of left horizontal W-jumps in the function $\Wjumph(R,j,\dirl,\alpha)$ below; see Figure~\ref{fig:swtjump}~(a).
The implementation of right horizontal W-jumps, and of left and right vertical W-jumps in a function $\Wjumpv(R,j,d,\alpha)$ is very similar; we omit the details here.

\begin{algo}{{\bfseries Function} $\Wjumph(R,j,\dirl,\alpha)$}{left horizontal W-jump}
\begin{enumerate}[label={\bfseries \arabic*.}, leftmargin=6mm, noitemsep, topsep=3pt plus 3pt]
\item{}[Prepare] Set $a\gets r_j.\rnw$, $\beta\gets v_a.\vwest$ and $k\gets e_\alpha.\eleft$.
\item{}[Flip and update rectangles] Call $\eremoveh(\beta)$ and $\einserta(\alpha,a,\beta)$.
Then set $e_\beta.\eleft\!\gets\! k$ and $e_\beta.\eright\gets j$.
\end{enumerate}
\end{algo}
The running time of $\Wjumph(R,j,\dirl,\alpha)$ is clearly $\cO(1)$, as claimed in part~(a) of Lemma~\ref{lem:jump-time}.

We proceed with the implementation of S-jumps.
For simplicity, we only provide the implementation of left S-jumps in the function $\Sjump(R,j,\dirl,\alpha)$ below; see Figure~\ref{fig:swtjump}~(b).
The implementation of right S-jumps is very similar.

\begin{figure}
\includegraphics[page=1]{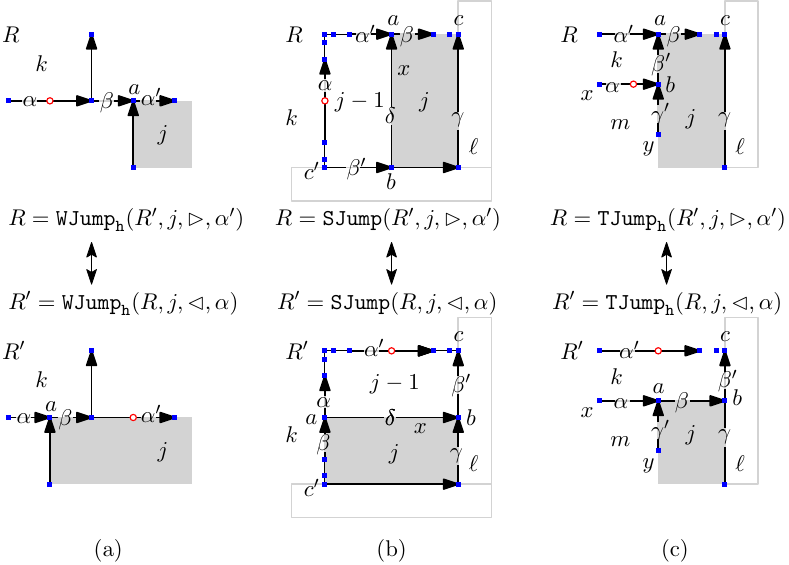}
\caption{Implementation of local jumps: (a) horizontal W-jumps; (b) S-jumps; (c) horizontal T-jumps.
Vertices are drawn as squares, insertion points as circles.
}
\label{fig:swtjump}
\end{figure}

\begin{algo}{{\bfseries Function} $\Sjump(R,j,\dirl,\alpha)$}{left S-jump}
\begin{enumerate}[label={\bfseries \arabic*.}, leftmargin=6mm, noitemsep, topsep=3pt plus 3pt]
\item{}[Prepare] Set $a\gets r_j.\rnw$, $b\gets r_j.\rsw$, $c\gets r_j.\rne$, $\alpha'\gets v_a.\vwest$, $\beta\gets v_a.\veast$, $\beta'\gets v_b.\vwest$, $\gamma\gets v_c.\vsouth$, $\delta\gets v_a.\vsouth$, $c'\gets e_{\beta'}.\etail$, $k\gets e_\alpha.\eleft$, $\ell\gets e_\gamma.\eright$ and $x\gets e_\delta.\ewall$.
\item{}[Flip] Call $\eremovet(\beta)$, $\eremoveh(\beta')$, $\einsertb(\beta,a,\alpha)$ and $\einserta(\gamma,b,\beta')$.
Then set $e_\delta.\dir\gets \dirr$, $e_\delta.\etail\gets a$, $e_\delta.\ehead\gets b$, $v_a.\veast\gets \delta$, $v_a.\vwest\gets 0$, $v_a.\vtype\gets\rightT$, $v_b.\veast\gets 0$, $v_b.\vwest\gets \delta$, $v_b.\vtype\gets\leftT$, $w_x.\wfirst\gets a$ and $w_x.\wlast\gets b$.
\item{}[Update rectangles]
Set $r_j.\rne\gets b$, $r_j.\rsw\gets c'$, $r_{j-1}.\rne\gets c$, and $r_{j-1}.\rsw\!\gets\! a$.
Set $\nu\gets v_c.\vwest$, and while $\nu\neq\alpha'$ repeat $e_\nu.\eright \gets j-1$ and $\nu\gets e_\nu.\eprev$.
Set $\nu\gets v_{c'}.\vnorth$, and while $\nu\neq\alpha$ repeat $e_\nu.\eright\gets j$ and $\nu\gets e_\nu.\enext$.
Also set $e_\beta.\eleft\gets k$, $e_\beta.\eright\gets j$, $e_{\beta'}.\eleft\gets j-1$ and $e_{\beta'}.\eright \gets \ell$.
\end{enumerate}
\end{algo}
Let $R'$ be the rectangulation obtained from~$R$ by one call of $\Sjump(R,j,\dirl,\alpha)$.
The running time of this call is $\cO(h(R,R')+v(R,R')+1)$, as claimed in part~(b) of Lemma~\ref{lem:jump-time}.
This time is incurred by the while-loops in step~3.
Specifically, the first while-loop is iterated exactly $h(R,R')$ times, and the second while-loop is iterated exactly $v(R,R')+1$ times.

We complete this section by presenting the implementation of T-jumps; see Figure~\ref{fig:swtjump}~(c).
For simplicity, we only provide the implementation of left horizontal T-jumps in the function $\Tjumph(R,j,\dirl,\alpha)$ below.
The implementation of right horizontal T-jumps, and of left and right vertical T-jumps in a function $\Tjumpv(R,j,d,\alpha)$ is very similar.

\begin{algo}{{\bfseries Function} $\Tjumph(R,j,\dirl,\alpha)$}{left horizontal T-jump}
\begin{enumerate}[label={\bfseries \arabic*.}, leftmargin=6mm, noitemsep, topsep=3pt plus 3pt]
\item{}[Prepare] Set $a\gets r_j.\rnw$, $b\gets e_\alpha.\ehead$, $c\gets r_j.\rne$, $\alpha'\gets v_a.\vwest$, $\beta\gets v_a.\veast$, $\beta'\gets v_a.\vsouth$, $\gamma\gets v_c.\vsouth$, $\gamma'\gets v_b.\vsouth$, $k\gets e_{\beta'}.\eleft$, $\ell\!\gets\! e_\gamma.\eright$, $m\!\gets\! e_\alpha.\eright$, $x\gets e_\alpha.\ewall$ and $y\gets e_{\gamma'}.\ewall$.
\item{}[Flip] Call $\eremovet(\beta)$, $\eremovet(\beta')$, $\einserta(\alpha,a,\beta)$ and $\einserta(\gamma,b,\beta')$.
Then set $e_\beta.\ehead\gets b$, $e_{\gamma'}.\ehead\gets a$, $v_a.\vsouth\gets \gamma'$, $v_b.\vwest\gets \beta$, $w_x.\wlast\gets b$ and $w_y.\wlast\gets a$.
\item{}[Update rectangles]
Set $r_j.\rne\gets b$, $r_k.\rne\gets c$ and $r_m.\rne\gets a$.
Set $\nu\gets v_c.\vwest$, and while $\nu\neq \alpha'$ repeat $e_\nu.\eright\gets k$ and $\nu\gets e_\nu.\eprev$.
Also set $e_\beta.\eleft\gets k$ and $e_{\beta'}.\eright\gets \ell$.
\end{enumerate}
\end{algo}
Let $R'$ be the rectangulation obtained from~$R$ by one call of $\Tjumph(R,j,\dirl,\alpha)$.
The running time of this call is $\cO(h(R,R')+1)$, as claimed in part~(c) of Lemma~\ref{lem:jump-time}.
This time is incurred by the while-loop in step~3, which is iterated exactly $h(R,R')+1$ times.

\section{Minimal jump oracles}
\label{sec:oracles}

A \emph{minimal jump oracle} is a function that is called in line~M4 of Algorithm~\Mrect{} to compute a jump in the current rectangulation~$R$ that is minimal respect to the given zigzag set of rectangulations~$\cC_n\seq\cR_n$.
In this section we specify such oracles for the zigzag sets~$\cC_n$ mentioned in Table~\ref{tab:rect}, which allows us to establish the runtime bounds for Algorithm~\Mrect{} stated in the last column of the table (except for block-aligned rectangulations, which are handled in Section~\ref{sec:Sequiv}).
A minimal jump oracle has the form~$\nextjump_{\cC_n}(R,j,d)$, and this function call performs in the current rectangulation~$R$ a jump of rectangle~$r_j$ in direction~$d$ that is minimal w.r.t.~$\cC_n$, and the function will modify~$R$ accordingly.
Depending on~$\cC_n$, our minimal jump oracles perform a suitable W-, S-, or T-jump, or a combination thereof, as implemented in the previous section.

\subsection{Generic rectangulations}
\label{sec:oracle-gen}

We first consider the case $\cC_n=\cR_n$ of generic rectangulations.
Given the current rectangulation~$R$, upon a jump of rectangle~$r_j$ in direction~$d$, \emph{every} insertion from $I(R^{[j-1]})$ is used, so we simply need to detect the next one.

By Lemma~\ref{lem:jumps-flips}, a W-jump occurs between any two consecutive (w.r.t.~$I(R^{[j-1]})$) insertion points belonging to the same horizontal or vertical group, an S-jump occurs between the last vertical insertion point and the first horizontal insertion point, and a T-jump occurs between the last insertion point of a group and the first insertion point of the next group, if both groups are vertical or horizontal.
Specifically, suppose there are $\lambda$ vertical groups and $\mu$ horizontal groups with cardinalities $g_k$, $k=1,\ldots,\lambda$, and $h_k$, $k=1,\ldots,\mu$, respectively in $I(R^{[j-1]})$ (note that $g_1=h_\mu=1$).
Then the jump sequence consisting of letters $\{W,S,T\}$ that specifies the types of jumps performed with rectangle~$r_j$ from the first to the last insertion point is
\begin{equation}
\label{eq:jump-seq-gen}
(TW^{g_2-1})(TW^{g_3-1})\cdots (TW^{g_\lambda-1})\,S\,(W^{h_1-1}T)(W^{h_2-1}T)\cdots(W^{h_{\mu-1}-1}T);
\end{equation}
see Figure~\ref{fig:next}~(a).
Of course, during Algorithm~\Mrect{}, these jump operations are not consecutive, but they are interleaved with the jump sequences of other rectangles~$r_k$, $k>j$.

The details are spelled out in the function $\nextjump_{\cR_n}(R,j,d)$.

\begin{algo}{$\nextjump_{\cR_n}(R,j,d)$}{Minimal jump oracle for generic rectangulations}
\begin{enumerate}[label={\bfseries N\arabic*.}, leftmargin=8mm, noitemsep, topsep=3pt plus 3pt]
\item{}[Prepare] Set $a\gets r_j.\rnw$.
If $d=\dirl$ and $v_a.\vtype=\bottomT$, set $\alpha\gets v_a.\vwest$, $\beta\gets v_a.\vsouth$, $b\gets e_\beta.\etail$ and $c\gets e_\alpha.\etail$ and goto~N2.
If $d=\dirr$ and $v_a.\vtype=\bottomT$, set $\alpha\gets v_a.\veast$, $b\gets e_\alpha.\ehead$ and goto~N3.
If $d=\dirr$ and $v_a.\vtype=\rightT$, set $\alpha\gets v_a.\vnorth$, $\beta\gets v_a.\veast$, $b\gets e_\beta.\ehead$ and $c\gets e_\alpha.\ehead$ and goto~N4.
If $d=\dirl$ and $v_a.\vtype=\rightT$, set $\alpha\gets v_a.\vsouth$, $b\gets e_\alpha.\etail$ and goto~N5.
\item{}[Horizontal left jump] If $v_c.\vtype=\topT$, set $\gamma\gets v_c.\vwest$ and call $\Wjumph(R,j,\dirl,\gamma)$.
Else if $v_b.\vtype=\leftT$, set $\gamma\gets v_b.\vwest$ and call $\Tjumph(R,j,\dirl,\gamma)$.
Otherwise we have $v_b.\vtype=\topT$, set $\gamma\gets v_c.\vsouth$ and call $\Sjump(R,j,\dirl,\gamma)$. Return.
\item{}[Horizontal right jump] If $v_b.\vtype=\topT$, set $\gamma\gets v_b.\veast$ and call $\Wjumph(R,j,\dirr,\gamma)$.
Otherwise we have $v_b.\vtype=\leftT$, set $k\gets e_\alpha.\eleft$, $c\gets r_k.\rnw$ and $\gamma\gets v_c.\veast$ and call $\Tjumph(R,j,\dirr,\gamma)$. Return.
\item{}[Vertical right jump] If $v_c.\vtype=\leftT$, set $\gamma\gets v_c.\vnorth$ and call $\Wjumpv(R,j,\dirr,\gamma)$.
Else if $v_b.\vtype=\topT$, set $\gamma\gets v_b.\vnorth$ and call $\Tjumpv(R,j,\dirr,\gamma)$.
Otherwise we have $v_b.\vtype=\leftT$, set $\gamma\gets v_c.\veast$ and call $\Sjump(R,j,\dirr,\gamma)$. Return.
\item{}[Vertical left jump] If $v_b.\vtype=\leftT$, set $\gamma\gets v_b.\vsouth$ and call $\Wjumpv(R,j,\dirl,\gamma)$.
Otherwise we have $v_b.\vtype=\topT$, set $k\gets e_\alpha.\eleft$, $c\gets r_k.\rnw$ and $\gamma\gets v_c.\vsouth$ and call $\Tjumpv(R,j,\dirl,\gamma)$. Return.
\end{enumerate}
\end{algo}
The four distinct cases treated in lines~N2--N4 come from the directions $d\in\{\dirl,\dirr\}$ and whether the jump is horizontal or vertical.
The latter condition is determined in line~N1 by querying the type of the top-left vertex of~$r_j$, which is either $\bottomT$ or~$\rightT$.
Note that the code in lines~N2 and~N4 is symmetric by reflecting all directions at the main diagonal.
The same holds for the code in lines~N3 and~N5.

\begin{lemma}
\label{lem:gen-time}
Consider a rectangulation~$P\in\cR_{n-1}$ with $\nu=\nu(P)$ insertion points.
Then calling $\nextjump_{\cR_n}(R,n,\dirr)$ exactly $\nu-1$ times with initial rectangulation $R=c_1(P)$, yields $c_i(P)$ for $i=1,\ldots,\nu$, and the total time for these calls is~$\cO(\nu)$.
An analogous statement holds for $\nextjump_{\cR_n}(R,n,\dirl)$.
\end{lemma}

\begin{proof}
If the sequence of insertion points~$I(P)$ has $\lambda$ vertical groups and $\mu$ horizontal groups of cardinalities $g_k$, $k=1,\ldots,\lambda$, and $h_k$, $k=1,\ldots,\mu$, respectively, then the sequence of jumps performed by the calls to $\nextjump_{\cR_n}$ has the form~\eqref{eq:jump-seq-gen}.
We clearly have $\nu=\sum_{k=1}^\lambda g_k+\sum_{k=1}^\mu h_k$.
We use Lemma~\ref{lem:jump-time} to bound the overall time to perform those operations; see Figure~\ref{fig:next}~(a).
The number of W-jumps in~\eqref{eq:jump-seq-gen} is $w:=\sum_{k=1}^\lambda (g_k-1)+\sum_{k=1}^\mu (h_k-1)\leq \nu$.
The sum of the terms $v(R,R')+1$ and $h(R,R')+1$ over any two consecutive rectangulations $R,R'$ in this sequence that differ in a T-jump is $t:=\sum_{k=1}^{\lambda-1}g_k$ and $t':=\sum_{k=2}^\mu h_k$, respectively.
The sum $v(R,R')+h(R,R')+1$ for the two consecutive rectangulations $R,R'$ in this sequence that differ in an S-jump is $s:=g_\lambda+h_1-1$.
Clearly, we have $s+t+t'\leq \nu$.
Consequently, the overall time for those operations is $\cO(w+s+t+t')=\cO(\nu)$, as claimed.
\end{proof}

\begin{theorem}
\label{thm:next-gen}
Algorithm~\Mrect{} with the minimal jump oracle $\nextjump_{\cR_n}$ takes time~$\cO(1)$ on average to visit each generic rectangulation.
\end{theorem}

\begin{figure}
\includegraphics{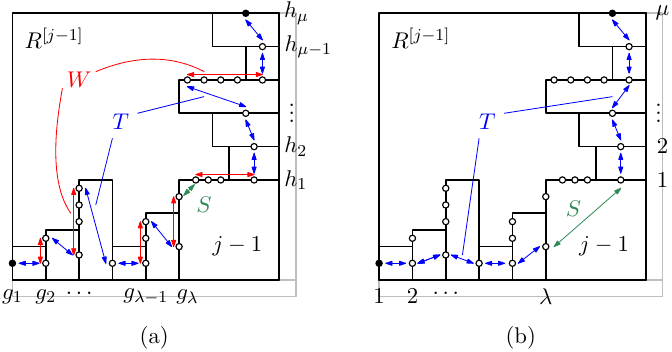}
\caption{Illustration of the proofs of (a) Theorem~\ref{thm:next-gen} for generic rectangulations and (b) Theorem~\ref{thm:next-diag} for diagonal rectangulations.}
\label{fig:next}
\end{figure}

\begin{proof}
For some fixed $j\in\{2,\ldots,n\}$, we consider all jumps of rectangle~$r_j$.
Whenever rectangle~$r_j$ jumps in a rectangulation~$R\in\cR_n$, then $R^{[k]}$ is either bottom-based or right-based for all $k=j+1,\ldots,n$.
Moreover, none of the rectangles~$k=j-1,j-2,\ldots,2$ jumps unless~$R^{[j]}$ is bottom-based or right-based.
Consequently, we can partition the jumps of~$r_j$ in the entire jump sequence into subsequences, such that in each subsequence, for each rectangulation~$R\in\cR_n$ of the subsequence, $R^{[j-1]}\in\cR_{j-1}$ is the same subrectangulation and in~$R$ the rectangle~$r_j$ jumps to the next insertion point of~$I(R^{[j-1]})$.
By Lemma~\ref{lem:gen-time}, the total time for visiting the $\nu=\nu(R^{[j-1]})$ many rectangulations of this subsequence is $\cO(\nu)$, which is $\cO(1)$ on average.
\end{proof}

\begin{remark}
By slightly modifying our data structures, we could even obtain a loopless algorithm for generic rectangulations.
The idea is to introduce an additional data structure called \emph{sides}.
Each rectangle is subdivided into four sides, and in the incidence relations, sides sit between edges and rectangles, i.e., edges do not point to the two touching rectangles directly, but to the relevant sides of those rectangles, and each side points to the rectangle it belongs to.
During S-jumps and T-jumps, a rectangle can be broken up into its four sides and the sides of two rectangles can be interchanged in constant time, avoiding the while-loops in the functions $\Sjump$ and $\Tjump$ that need to update possibly linearly many incidences between edges and rectangles.
To keep the presentation simple, we do not show these modifications.
Also, the resulting improvement is not substantial, and sides are a somewhat artificial concept.
\end{remark}

\subsection{Diagonal rectangulations}
\label{sec:oracle-diag}

Recall that in a diagonal rectangulation~$R\in\cD_n$, every rectangle intersects the main diagonal, or equivalently, $R$ avoids the patterns~$\zvu$ and~$\zhr$.
Consequently, during a jump of rectangle~$r_j$ in the current rectangulation~$R$, we need to consider precisely the insertion points from~$I(R^{[j-1]})$ that are the first insertion point of a vertical group, or the last insertion point of a horizontal group, as any other insertion point from each group would create one of the forbidden pattern; see Figure~\ref{fig:next}~(b).
Consequently, if the sequence $I(R^{[j-1]})$ has $\lambda$ vertical groups and $\mu$ horizontal groups, then the jump sequence that specifies the types of jumps with rectangle~$r_j$ from the first to the last insertion point is
\begin{equation*}
T^{\lambda-1}\,S\,T^{\mu-1}.
\end{equation*}
In particular, we do not perform any wall slides.

An implementation of this is provided in the function $\nextjump_{\cD_n}(R,j,d)$.

\begin{algo}{$\nextjump_{\cD_n}(R,j,d)$}{Minimal jump oracle for diagonal rectangulations}
\begin{enumerate}[label={\bfseries N\arabic*.}, leftmargin=8mm, noitemsep, topsep=3pt plus 3pt]
\item{}[Prepare] Set $a\gets r_j.\rnw$.
If $d=\dirl$ and $v_a.\vtype=\bottomT$, set $\alpha\gets v_a.\vsouth$ and $b\gets e_\alpha.\etail$ and goto~N2.
If $d=\dirr$ and $v_a.\vtype=\bottomT$, set $\alpha\gets v_a.\veast$ and goto~N3.
If $d=\dirr$ and $v_a.\vtype=\rightT$, set $\alpha\gets v_a.\veast$ and $b\gets e_\alpha.\ehead$ and goto~N4.
If $d=\dirl$ and $v_a.\vtype=\rightT$, set $\alpha\gets v_a.\vsouth$ and goto~N5.
\item{}[Horizontal left jump] If $v_b.\vtype=\leftT$, set $\gamma\gets v_b.\vwest$ and call $\Tjumph(R,j,\dirl,\gamma)$.
Otherwise we have $v_b.\vtype=\topT$, set $c\gets r_{j-1}.\rsw$ and $\gamma\gets v_c.\vnorth$ and call $\Sjump(R,j,\dirl,\gamma)$.
Return.
\item{}[Horizontal right jump] Set $k\gets e_\alpha.\eleft$, $b\gets r_k.\rne$ and $\gamma\gets v_b.\vwest$ and call $\Tjumph(R,j,\dirr,\gamma)$. Return.
\item{}[Vertical right jump] If $v_b.\vtype=\topT$, set $\gamma\gets v_b.\vnorth$ and call $\Tjumpv(R,j,\dirr,\gamma)$.
Otherwise we have $v_b.\vtype=\leftT$, set $c\gets r_{j-1}.\rne$ and $\gamma\gets v_c.\vwest$ and call $\Sjump(R,j,\dirr,\gamma)$.
Return.
\item{}[Vertical left jump] Set $k\gets e_\alpha.\eleft$, $b\gets r_k.\rsw$ and $\gamma\gets v_b.\vnorth$ and call $\Tjumpv(R,j,\dirl,\gamma)$. Return.
\end{enumerate}
\end{algo}
Similarly to before, the code in lines~N2 and~N4, and in lines~N3 and~N5 is symmetric by reflecting all directions at the main diagonal.
For diagonal rectangulations the runtime analysis is straightforward and gives a loopless algorithm.

\begin{lemma}
\label{lem:diag-time}
Each call $\nextjump_{\cD_n}(R,j,d)$ takes time~$\cO(1)$.
\end{lemma}

\begin{proof}
Let $R'$ be the rectangulation after the call $\nextjump_{\cD_n}(R,j,d)$, which differs from~$R$ in an S-jump or T-jump.
As we only consider the first insertion point of each vertical group and the last insertion point of each horizontal group of~$I(R^{[j-1]})$, we have $v(R,R')=0$ and $h(R,R')=0$.
The claim now follows from Lemma~\ref{lem:jump-time}~(b)+(c).
\end{proof}

Lemma~\ref{lem:diag-time} immediately yields the following result.

\begin{theorem}
\label{thm:next-diag}
Algorithm~\Mrect{} with the minimal jump oracle~$\nextjump_{\cD_n}$ takes time~$\cO(1)$ to visit each diagonal rectangulation.
\end{theorem}

\begin{remark}
Jumps as performed by the oracles~$\nextjump_{\cR_n}$ and~$\nextjump_{\cD_n}$ and shown in Figure~\ref{fig:next} correspond to cover relations in the lattice of generic rectangulations and the lattice of diagonal rectangulations described by Meehan~\cite{meehan_2019} and~Law and Reading~\cite{MR2871762}, respectively, which both arise as lattice quotients of the weak order on the symmetric group.
Consequently, our cyclic Gray codes correspond to Hamilton cycles in the cover graphs of those lattices.
In~\cite{perm_series_ii} we showed that our permutation language framework can be used to generate a Hamilton path on every lattice quotient of the weak order, which also yields a Hamilton path on the skeleton of the corresponding polytope~\cite{MR3964495} (see also~\cite{padrol_pilaud_ritter_2021}).
\end{remark}

\subsection{Pattern-avoiding rectangulations}
\label{sec:oracle-pattern}

For any zigzag set of rectangulations~$\cC_n\seq\cR_n$, and any set of tame patterns~$\cP$, Theorem~\ref{thm:pattern} guarantees that the set~$\cC_n(\cP)$ of rectangulations that avoid all patterns from~$\cP$ is also a zigzag set.
We now describe how we can obtain a minimal jump oracle for~$\cC_n(\cP)$ from a minimal jump oracle~$\nextjump_{\cC_n}(R,j,d)$ for~$\cC_n$.
The idea is simply to perform a minimal jump of~$r_j$ w.r.t.~$\cC_n$, and to test after each jump whether the resulting rectangulation contains any pattern from~$\cP$, repeating this process until we arrive at a rectangulation that avoids all patterns from~$\cP$.
This is guaranteed to terminate after at most~$j\leq n$ iterations, as the first and last insertion point of~$I(R^{[j-1]})$ will produce rectangulations that avoid all patterns from~$\cP$, due to the zigzag property.

\begin{algo}{$\nextjump_{\cC_n(\cP)}(R,j,d)$}{Minimal jump oracle for pattern-avoiding permutations}
\begin{enumerate}[label={\bfseries N\arabic*.}, leftmargin=8mm, noitemsep, topsep=3pt plus 3pt]
\item{}[Fast forward] While $R$ contains a pattern from~$\cP$ repeat $\nextjump_{\cC_n}(R,j,d)$.
\end{enumerate}
\end{algo}
We immediately obtain the following generic runtime bounds.

\begin{theorem}
\label{thm:nextjumpP-time}
Let $\cP$ be a finite set of tame rectangulation patterns.
If the zigzag set $\cC_n$ has a minimal jump oracle $\nextjump_{\cC_n}$ that runs in time~$f_n$, and containment of any pattern from~$\cP$ in $R$ can be tested in time~$t_n$, then $\nextjump_{\cC_n(\cP)}$ is a minimal jump oracle for $\cC_n(\cP)$ that runs in time~$\cO(n\cdot (f_n+t_n))$.
\end{theorem}

In some cases the runtime bound for at most~$n$ consecutive calls of $\nextjump_{\cC_n}(R,j,d)$ or several consecutive pattern containment tests can be improved upon the trivial bounds~$\cO(n\cdot f_n)$ and~$\cO(n\cdot t_n)$, respectively (see the proof of Theorem~\ref{thm:nextjumpP-RD} below).
Moreover, for some patterns further optimizations of the function $\nextjump_{\cC_n(\cP)}(R,j,d)$ are possible.
For example, the property of $R$ to be guillotine is invariant under W-jumps and S-jumps, so if $R$ is found to contain one of the windmill patterns, then we only need to check containment after the next T-jump performed by the call $\nextjump_{\cC_n}(R,j,d)$.
Most importantly, when testing for containment of a pattern, we only need to check incidences of walls involving the rectangle~$r_j$.
In the following, we provide functions $\contains(R,j,P)$ that test whether $R$ contains one of the tame patterns~$P$ listed in Lemma~\ref{lem:tame} after a sequence of jumps of rectangle~$r_j$ from a rectangulation that avoids~$P$.
We emphasize here that these functions only work under these assumptions, and are not suitable for general pattern containment testing of arbitrary rectangulations, but only for use within our algorithm~\Mrect{}.

We first present an implementation of such a containment testing function $\contains(R,j,P)$ for the clockwise windmill $P=\millr$; see Figure~\ref{fig:contains}~(a).
It uses the wall data structure~$w_1,\ldots,w_{n+3}$ to quickly move to the end vertex of a wall (without traversing the possibly many edges along the wall).

\begin{algo}{$\contains(R,j,\millr)$}{Check for clockwise windmill pattern after jump of rectangle~$r_j$}
\begin{enumerate}[label={\bfseries C\arabic*.}, leftmargin=8mm, noitemsep, topsep=3pt plus 3pt]
\item{}[Prepare] Set $a\gets r_j.\rnw$.
If $v_a.\vtype\!=\!\bottomT$, return \bfalse.
Otherwise we have $v_a.\vtype\!=\!\rightT$ and proceed with C2.
\item{}[Check] Set $\alpha\gets v_a.\vnorth$, $x\gets e_\alpha.\ewall$, $b\gets w_x.\wlast$, $\beta\gets v_b.\veast$, $y\gets e_\beta.\ewall$, $c\gets w_y.\wlast$, $\gamma\gets v_c.\vsouth$, $z\gets e_\gamma.\ewall$, $d\gets w_z.\wfirst$ and $\delta\gets v_d.\vwest$.
If $e_\delta.\eright=j$, return \btrue, otherwise return \bfalse.
\end{enumerate}
\end{algo}
The function that tests for the counterclockwise windmill $P=\milll$ is symmetric, and is not shown here for simplicity.

\begin{figure}
\includegraphics[page=1]{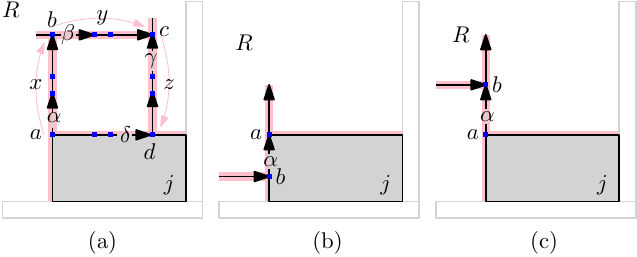}
\caption{Testing containment of the patterns (a) $\millr$, (b) $\zvu$ and (c) $\zvd$.
The forbidden configuration of walls is highlighted.
}
\label{fig:contains}
\end{figure}

The next two functions test for containment of the patterns $P=\zvu$ and $P=\zvd$, respectively; see Figure~\ref{fig:contains}~(b)+(c).

\begin{algo}{$\contains(R,j,\zvu)$}{Check for $\zvu$ after jump of rectangle~$r_j$}
\begin{enumerate}[label={\bfseries C\arabic*.}, leftmargin=8mm, noitemsep, topsep=3pt plus 3pt]
\item{}[Prepare] Set $a\gets r_j.\rnw$.
If $v_a.\vtype\!=\!\bottomT$, return \bfalse.
Otherwise we have $v_a.\vtype\!=\!\rightT$ and proceed with C2.
\item{}[Check] Set $\alpha\gets v_a.\vsouth$ and $b\gets e_\alpha.\etail$.
If $v_b.\vtype=\leftT$, return~\btrue, otherwise return~\bfalse.
\end{enumerate}
\end{algo}

\begin{algo}{$\contains(R,j,\zvd)$}{Check for $\zvd$ after jump of rectangle~$r_j$}
\begin{enumerate}[label={\bfseries C\arabic*.}, leftmargin=8mm, noitemsep, topsep=3pt plus 3pt]
\item{}[Prepare] Set $a\gets r_j.\rnw$.
If $v_a.\vtype\!=\!\bottomT$, return \bfalse.
Otherwise we have $v_a.\vtype\!=\!\rightT$ and proceed with C2.
\item{}[Check] Set $\alpha\gets v_a.\vnorth$ and $b\gets e_\alpha.\ehead$.
If $v_b.\vtype=\leftT$, return~\btrue, otherwise return~\bfalse.
\end{enumerate}
\end{algo}
Similarly to before, testing for the patterns~$P=\zhr$ and~$P=\zhl$ is symmetric to the previous two cases, so we omit those implementations.

It remains to provide containment testing for the patterns~$P=\hvert$ and~$P=\hhor$.
We only show the first case, as the other is symmetric; see Figure~\ref{fig:containsH}.

\begin{algo}{$\contains(R,j,\hvert)$}{Check for $\hvert$ after jump of rectangle~$r_j$}
\begin{enumerate}[label={\bfseries C\arabic*.}, leftmargin=8mm, noitemsep, topsep=3pt plus 3pt]
\item{}[Prepare] Set $a\gets r_j.\rnw$.
If $v_a.\vtype\!=\!\bottomT$, return \bfalse.
Otherwise we have $v_a.\vtype\!=\!\rightT$, set $b\gets r_j.\rsw$ and proceed with~C2.
\item{}[Go up] While $v_b.\vtype\notin\{\bottomT,0\}$ repeat: goto~C3; [*] set $\beta\gets v_b.\vnorth$ and $b\gets e_\beta.\ehead$.
Return \bfalse.
\item{}[Go left] Set $c\gets b$.
While $v_c.\vtype\notin\{\rightT,0\}$ repeat: if $v_c.\vtype=\topT$ goto~C4; [**] set $\gamma\gets v_c.\vwest$ and $c\gets e_\gamma.\etail$.
Go back to [*].
\item{}[Go up] Set $d\gets c$.
While $d\neq b$ and $v_d.\vtype\notin\{\bottomT,0\}$ repeat: if $v_d.\vtype=\leftT$ return \btrue; set $\delta\gets v_d.\vnorth$ and $d\gets e_\delta.\ehead$.
Go back to [**].
\end{enumerate}
\end{algo}
Lines~C2--C4 are essentially a triply nested loop that moves along the edges of the vertical wall~$x$ that contains the left side of~$r_j$ (line~C2), the edges of each horizontal wall~$y$ whose right end vertex lies on~$x$ (line~C3), and the edges of each vertical wall~$z$ whose bottom end vertex lies on~$y$, searching for a vertex of type~$\leftT$ on~$z$ (line~C4); see Figure~\ref{fig:containsH}.
This is realized by repeatedly calling line~C3 from within the while-loop in line~C2, and upon completion returning to from where the call occurred.
Similarly, line~C4 is repeatedly called from within the while-loop in line~C3, and upon completion it returns to from where it was called.

\begin{figure}
\includegraphics[page=2]{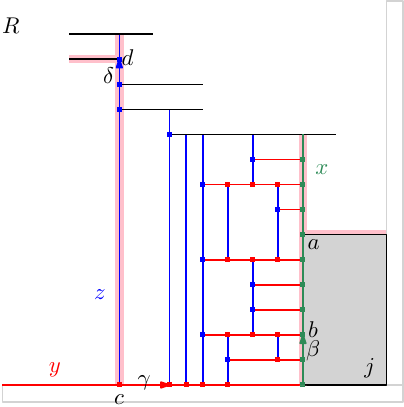}
\caption{Testing containment of the pattern $\hvert$.
The forbidden configuration of walls is highlighted.
}
\label{fig:containsH}
\end{figure}

The aforementioned functions have the following runtime guarantees.

\begin{lemma}
\label{lem:contains-time}
The function $\contains(R,j,P)$ takes time $\cO(1)$ for the patterns $P=\millr,\milll,\zvu,\linebreak\zhr,\zvd,\zhl$ and time $\cO(n)$ for the patterns $P=\hvert,\hhor$.
\end{lemma}

\begin{proof}
For the first six patterns the statement is obvious, as the specified functions only make constantly many changes to our data structures.
For the pattern $P=\hvert$, note that each edge index of the rectangulation~$R$ is assigned to one of the variables $\beta,\gamma,\delta$ at most once during the call  $\contains(R,j,P)$ in lines~C2--C4, and the number of edges of~$R$ is $3n+1=\cO(n)$.
For the pattern $P=\hhor$ the argument is the same.
\end{proof}

The next theorem combines all the observations from this section, thus establishing most of the runtime bounds stated in Table~\ref{tab:rect}.

\begin{theorem}
\label{thm:nextjumpP-RD}
Let $\cP_1:=\big\{\millr,\milll,\zvu,\zhr,\zvd,\zhl\big\}$ and $\cP_2:=\big\{\hvert,\hhor\big\}$.
Algorithm~\Mrect{} with the minimal jump oracle~$\nextjump_{\cR_n(\cP)}$ or $\nextjump_{\cD_n(\cP)}$ visits each rectangulation from~$\cR_n(\cP)$ or~$\cD_n(\cP)$, respectively, in time~$\cO(n)$ for any set of patterns $\cP\seq\cP_1$ and in time~$\cO(n^2)$ for any set of patterns $\cP\seq\cP_1\cup\cP_2$.
\end{theorem}

All the bounds stated in Theorem~\ref{thm:nextjumpP-RD} hold in the worst case (not just on average).

\begin{proof}
We first consider the minimal jump oracle $\nextjump_{\cD_n(\cP)}$, which repeatedly calls the function~$\nextjump_{\cD_n}$ described in Section~\ref{sec:oracle-diag}.
Applying Theorem~\ref{thm:nextjumpP-time} with the bound $f_n=\cO(1)$ from Lemma~\ref{lem:diag-time} and the bounds $t_n=\cO(1)$ for $\cP\seq\cP_1$ and $t_n=\cO(n)$ for $\cP\seq\cP_1\cup\cP_2$ from Lemma~\ref{lem:contains-time}, the term $n\cdot(f_n+t_n)$ evaluates to $\cO(n)$ or $\cO(n^2)$, respectively, as claimed.

We now consider the minimal jump oracle $\nextjump_{\cR_n(\cP)}$, which repeatedly calls~$\nextjump_{\cR_n}$ described in Section~\ref{sec:oracle-gen}.
In this case applying Theorem~\ref{thm:nextjumpP-time} directly would not give the desired bounds, so we have to refine the analysis of the while-loop in the algorithm~$\nextjump_{\cR_n(\cP)}$.
Specifically, we consider the sequence of calls to $\nextjump_{\cR_n}(R,j,d)$ from one rectangulation~$R\in\cR_n$ avoiding all patterns from~$\cP$ until the next one.
The length of this sequence is at most~$\nu=\nu(R^{[j-1]})\leq n$ (recall Lemma~\ref{lem:nu}), and by Lemma~\ref{lem:gen-time} the total time of all calls to~$\nextjump_{\cR_n}(R,j,d)$ is~$\cO(\nu)=\cO(n)$.
The total time of all pattern containment tests is at most $\nu\cdot t_n\leq n\cdot t_n$, which is $\cO(n)$ for $\cP\seq\cP_1$ and $\cO(n^2)$ for $\cP\seq\cP_1\cup\cP_2$ by Lemma~\ref{lem:contains-time}.
This proves the claimed bounds, completing the proof of the theorem.
\end{proof}

\begin{remark}
For $\cP:=\big\{\zvu,\zhr\big\}$ we have $\cR_n(\cP)=\cD_n$, so we could use, or rather `misuse', the minimal jump oracle~$\nextjump_{\cR_n(\cP)}$ to generate~$\cD_n$.
However, this would give a worse guarantee of~$\cO(n)$ time per visited diagonal rectangulation, rather than~$\cO(1)$ for~$\nextjump_{\cD_n}$ as guaranteed by Theorem~\ref{thm:next-diag}.
\end{remark}

\begin{remark}
\label{rem:lb}
We write $c(\cC_{n-1}(\cP))$ for the set of all rectangulations from~$\cC_n$ that are obtained by inserting a rectangle into a rectangulation from~$\cC_{n-1}(\cP)$.
To assess the runtime bounds stated in Theorem~\ref{thm:nextjumpP-RD}, one may try to investigate the quantity $|c(\cC_{n-1}(\cP)|/|\cC_n(\cP)|$.
This is a lower bound for the average number of iterations of the while-loop of the algorithm~$\nextjump_{\cC_n(\cP)}$ before it returns a rectangulation from~$\cC_n(\cP)$.
Experimentally, we found that this ratio grows with~$n$ in many cases, though maybe not linearly with~$n$, hinting at the possibility that the time bounds stated in Theorem~\ref{thm:nextjumpP-RD} are too pessimistic and can be improved in an average case analysis.
\end{remark}

\section{Proofs of Theorems~\ref{thm:jump-rect} and~\ref{thm:algo-rect}}
\label{sec:proofs}

In this section we present the proofs of Theorems~\ref{thm:jump-rect} and~\ref{thm:algo-rect}.
For this purpose we first recap the exhaustive generation framework for permutation languages developed in~\cite{DBLP:conf/soda/HartungHMW20,perm_series_i}.
Definitions and terminology intentionally parallel the corresponding definitions given for rectangulations before, and the connection between rectangulations and permutations will be made precise in Lemma~\ref{lem:gamma-jump} below.

\subsection{Permutation basics}

For any two integers $a\leq b$ we define $[a,b]:=\{a,a+1,\ldots,b\}$, and we refer to a set of this form as an \emph{interval}.
We also define $[n]:=[1,n]=\{1,\ldots,n\}$.
We write $S_n$ for the set of permutations on~$[n]$, and we write $\pi\in S_n$ in one-line notation as $\pi=\pi(1)\pi(2)\cdots \pi(n)=a_1a_2\cdots a_n$.
Moreover, we use $\varepsilon\in S_0$ to denote the empty permutation, and $\ide_n=12\cdots n\in S_n$ to denote the identity permutation.

Given two permutations~$\pi$ and~$\tau$, we say that $\pi$ \emph{contains the pattern~$\tau$}, if there is a subsequence of~$\pi$ whose elements have the same relative order as in~$\tau$.
Otherwise we say that $\pi$ \emph{avoids} $\tau$.
For example, $\pi=6\colorbox{black!30}{\!35\!}41\colorbox{black!30}{\!2\!}$ contains the pattern $\tau=231$, as the highlighted entries show, whereas $\pi=654123$ avoids $\tau=231$.
In a \emph{vincular} pattern~$\tau$, there is exactly one underlined pair of consecutive entries, with the interpretation that the underlined entries must match adjacent positions in~$\pi$.
For instance, the permutation $\pi=\colorbox{black!30}{\!3\!}1\colorbox{black!30}{\!4\!}\colorbox{black!30}{\!2\!}$ contains the pattern~$\tau=231$, but it avoids the vincular pattern~$\tau=\ul{23}1$.

\subsection{Deletion, insertion and jumps in permutations}
\label{sec:perm-ops}

For~$\pi\in S_n$, $n\geq 1$, we write $p(\pi)\in S_{n-1}$ for the permutation obtained from~$\pi$ by deleting the largest entry~$n$.
We also define $\pi^{[i]}:=p^{n-i}(\pi)$ for $i=1,\ldots,n$.
Moreover, for any $\pi\in S_{n-1}$ and any $1\leq i\leq n$, we write $c_i(\pi)\in S_n$ for the permutation obtained from~$\pi$ by inserting the new largest value~$n$ at position~$i$ of~$\pi$, i.e., if $\pi=a_1\cdots a_{n-1}$ then $c_i(\pi)=a_1\cdots a_{i-1} \, n\, a_i \cdots a_{n-1}$.
For example, for $\pi=412563$ we have $p(\pi)=41253$ and $c_1(\pi)=7412563$, $c_5(\pi)=4125763$ and $c_7(\pi)=4125637$.
Given a permutation $\pi=a_1\cdots a_n$ with a substring $a_i\cdots a_{i+d}$ with $d>0$ and $a_i>a_{i+1},\ldots,a_{i+d}$, a \emph{right jump of the value~$a_i$ by $d$~steps} is a cyclic left rotation of this substring by one position to $a_{i+1}\cdots a_{i+d} a_i$.
Similarly, given a substring $a_{i-d}\cdots a_i$ with $d>0$ and $a_i>a_{i-d},\ldots,a_{i-1}$, a \emph{left jump of the value~$a_i$ by $d$~steps} is a cyclic right rotation of this substring to $a_i a_{i-d}\cdots a_{i-1}$.
For example, a right jump of the value~5 in the permutation~$265134$ by 2~steps yields~$261354$.

We say that a jump is \emph{minimal} w.r.t.\ a set of permutations~$L_n\seq S_n$, if every jump of the same value in the same direction by fewer steps creates a permutation that is not in~$L_n$.

\subsection{Generating permutations by minimal jumps}

Consider the following analogue of Algorithm~\Jrect{} for greedily generating a set of permutations $L_n\seq S_n$ using minimal jumps.

\begin{algo}{Algorithm~J}{Greedy minimal jumps}
This algorithm attempts to greedily generate a set of permutations $L_n\seq S_n$ using minimal jumps starting from an initial permutation~$\pi_0 \in L_n$.
\begin{enumerate}[label={\bfseries J\arabic*.}, leftmargin=8mm, noitemsep, topsep=3pt plus 3pt]
\item{} [Initialize] Visit the initial permutation~$\pi_0$.
\item{} [Jump] Generate an unvisited permutation from~$L_n$ by performing a minimal jump of the largest possible value in the most recently visited permutation.
If no such jump exists, or the jump direction is ambiguous, then terminate.
Otherwise visit this permutation and repeat~J2.
\end{enumerate}
\end{algo}

The following results were proved in~\cite{perm_series_i}.
A set of permutations~$L_n\seq S_n$ is called a \emph{zigzag language}, if either $n=0$ and $L_0=\{\varepsilon\}$, or if $n\geq 1$ and $L_{n-1}:=\{p(\pi)\mid \pi\in L_n\}$ is a zigzag language and for every $\pi\in L_{n-1}$ we have~$c_1(\pi)\in L_n$ and~$c_n(\pi)\in L_n$.

We now define a sequence~$J(L_n)$ of all permutations from a zigzag language~$L_n\seq S_n$.
For any $\pi\in L_{n-1}$ we let $\rvec{c}(\pi)$ be the sequence of all $c_i(\pi)\in L_n$ for $i=1,2,\ldots,n$, starting with~$c_1(\pi)$ and ending with~$c_n(\pi)$, and we let $\lvec{c}(\pi)$ denote the reverse sequence, i.e., it starts with~$c_n(\pi)$ and ends with~$c_1(\pi)$.
In words, those sequences are obtained by inserting into~$\pi$ the new largest value~$n$ in all possible positions from left to right, or from right to left, respectively, in all possible positions that yield a permutation from~$L_n$, skipping the positions that yield a permutation that is not in~$L_n$.
If $n=0$ then we define $J(L_0):=\varepsilon$, and if $n\geq 1$ then we consider the finite sequence $J(L_{n-1})=:\pi_1,\pi_2,\ldots$ and define
\begin{equation}
\label{eq:JLn}
J(L_n):=\lvec{c}(\pi_1),\rvec{c}(\pi_2),\lvec{c}(\pi_3),\rvec{c}(\pi_4),\ldots,
\end{equation}
i.e., this sequence is obtained from the previous sequence by inserting the new largest value~$n$ in all possible positions alternatingly from right to left, or from left to right.

\begin{theorem}[{\cite[Thm.~1+Lemma~4]{perm_series_i}}]
\label{thm:jump}
Given any zigzag language of permutations~$L_n$ and initial permutation~$\pi_0=\ide_n$, Algorithm~J visits every permutation from~$L_n$ exactly once, in the order~$J(L_n)$ defined by~\eqref{eq:JLn}.
Moreover, if $|L_i|$ is even for all $2\leq i\leq n-1$, then the sequence~$J(L_n)$ is cyclic, i.e., the first and last permutation differ in a minimal jump.
\end{theorem}

A permutation $\pi$ is called \emph{2-clumped} if it avoids each of the vincular patterns $3\ul{51}24$, $3\ul{51}42$, $24\ul{51}3$, and $42\ul{51}3$.
We write $S_n'\seq S_n$ for the set of 2-clumped permutations.

\begin{lemma}[{\cite[Thm.~8+Lemma~10]{perm_series_i}}]
\label{lem:2clumped-zigzag}
We have $S_0'=\{\varepsilon\}$, and for every $n\geq 1$ we have $S_{n-1}'=\{p(\pi)\mid \pi\in S_n'\}$ and $S_n'\supseteq\{c_1(\pi),c_n(\pi)\mid \pi\in S_{n-1}'\}$.
In particular, $S_n'$ is a zigzag language for all~$n\geq 0$.
\end{lemma}

We also state the following observations for further reference.

\begin{lemma}
\label{lem:JLn-prop}
The sequence of permutations $J(L_n)$ defined in~\eqref{eq:JLn} has the following properties:
\begin{enumerate}[label=(\alph*), leftmargin=8mm, noitemsep, topsep=3pt plus 3pt]
\item The first permutation in $J(L_n)$ is the identity permutation $\ide_n$.
\item For $j=2,\ldots,n$, the first jump of the value~$j$ in $J(L_n)$ is a left jump.
\item Every jump in the sequence $J(L_n)$ is minimal w.r.t.~$L_n$.
\item Given two consecutive permutations $\pi,\rho$ in~$J(L_n)$ that differ in a jump of some value~$j$, then we have $\pi^{[k]}=c_1(\pi^{[k-1]})$ and $\rho^{[k]}=c_1(\rho^{[k-1]})$, or $\pi^{[k]}=c_k(\pi^{[k-1]})$ and $\rho^{[k]}=c_k(\rho^{[k-1]})$ for all $k=j+1,\ldots,n$.
\item Let $\pi,\rho$ be two consecutive permutations in~$J(L_n)$ such that $\rho$ is obtained from~$\pi$ by a left jump of some value~$j$, and let $\pi',\rho'$ be the next two consecutive permutations in~$J(L_n)$ that differ in a jump of~$j$.
If $j$ is not at the first position in~$\rho$ and the value left of it is smaller than~$j$, then $\rho'$ is obtained from~$\pi'$ by a left jump.
Conversely, if $j$ is at the first position in~$\rho$ or the value left of it is bigger than~$j$, then $\rho'$ is obtained from~$\pi'$ by a right jump.
An analogous statement holds with left and right interchanged.
\end{enumerate}
\end{lemma}

\begin{proof}
Properties~(a) and~(b) follow easily from the definition~\eqref{eq:JLn}.

Property~(c) follows from Theorem~\ref{thm:jump} and line~J2 of Algorithm~J.

We prove~(d) and~(e) by induction on~$n$.
Both statements are trivial for~$n=0$ and~$n=1$, which settles the induction basis.
We now assume that~$n\geq 2$, and suppose that $J(L_{n-1})=:\pi_1,\pi_2,\ldots$ satisfies~(a) and~(b) for jumps of all values~$j\in\{2,\ldots,n-1\}$.

We start with the induction step for~(d).
If $\pi,\rho$ in~$J(L_n)$ differ in a jump of the value~$j=n$, then (d) is satisfied trivially, so it suffices to consider jumps of values~$j\in\{2,\ldots,n-1\}$ in~$J(L_n)$.
However, by~\eqref{eq:JLn}, such jumps only occur at the transitions between~$\lvec{c}(\pi_k)$ and~$\rvec{c}(\pi_{k+1})$ or between~$\rvec{c}(\pi_k)$ and~$\lvec{c}(\pi_{k+1})$ for some~$k$.
In the first case, $\pi=\pi^{[n]}:=n \pi_k=c_1(\pi_k)=c_1(\pi^{[n-1]})$ is followed by $\rho=\rho^{[n]}:=n \pi_{k+1}=c_1(\pi_{k+1})=c_1(\rho^{[n-1]})$, and in the second case $\pi=\pi^{[n]}:=\pi_k n=c_n(\pi_k)=c_n(\pi^{[n-1]})$ is followed by $\rho=\rho^{[n]}:=\pi_{k+1}n=c_n(\pi_{k+1})=c_n(\rho^{[n-1]})$, completing the proof.

We proceed with the induction step for~(e).
If $\pi,\rho$ in~$J(L_n)$ differ in a jump of the value~$j=n$, then (e) follows from the definition~\eqref{eq:JLn} and of the sequences~$\lvec{c}(\pi_i)$ and $\rvec{c}(\pi_i)$.
On the other hand, consider $\pi,\rho$ in~$J(L_{n-1})$ that differ in a jump of some value~$j\in\{2,\ldots,n-1\}$, and let~$\pi',\rho'$ be the next two permutations in~$J(L_{n-1})$ that differ in a jump of~$j$, satisfying~(e) by induction.
Then in $J(L_n)$, after the jump of~$j$ from~$c_1(\pi)$ to~$c_1(\rho)$ or from~$c_n(\pi)$ to~$c_n(\rho)$, the next jump of the value~$j$ is from~$c_1(\pi')$ to~$c_1(\rho')$ or from~$c_n(\pi')$ to~$c_n(\rho')$.
If $j$ is not at the first position in~$\rho$ and the value left of it is smaller than~$j$, then $j$ is not at the first position and the value left of it is smaller in both~$c_1(\rho)$ and~$c_n(\rho)$, and $c_1(\rho')$ and~$c_n(\rho')$ are obtained from $c_1(\pi')$ and~$c_n(\pi')$, respectively, by a left jump.
On the other hand, if $j$ is at the first position in~$\rho$ or the value left of it is bigger than~$j$, then $j$ is at the first position or the value left of it is bigger than~$j$ in both~$c_1(\rho)$ and~$c_n(\rho)$, and $c_1(\rho')$ and~$c_n(\rho')$ are obtained from $c_1(\pi')$ and~$c_n(\pi')$, respectively, by a right jump.

This completes the proof.
\end{proof}

\subsection{A surjection from permutations to generic rectangulations}
\label{sec:surjection}

Observe that a diagonal rectangulation with $n$ rectangles can be laid out canonically so that each rectangle intersects the main diagonal in a $1/n$-fraction.
Specifically, rectangle~$r_i$ intersects the main diagonal in the $i$th such line segment counted from top-left to bottom-right, for $i=1,\ldots,n$.
We say that $r_i$ is \emph{left-fixed} or \emph{left-extended}, if its left side touches or does not touch the diagonal, respectively.
These notions are defined analogously for all the other three sides right, bottom, and top.

\begin{figure}
\makebox[0cm]{ 
\includegraphics[page=1]{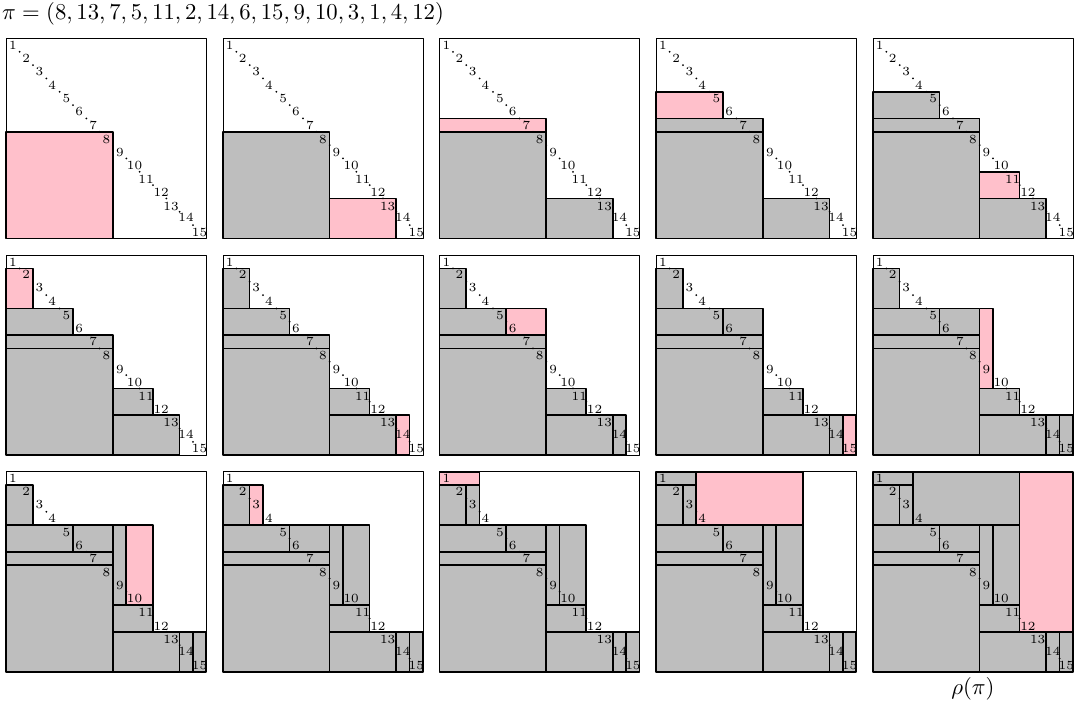}
}
\caption{Illustration of the mapping~$\rho:S_n\to\cD_n$ for the permutation $\pi=(8,13,7,5,11,2,14,6,15,9,10,3,1,4,12)$ (example from~\cite{MR2864445}).}
\label{fig:rho}
\end{figure}

We begin by reviewing a mapping $\rho:S_n\to\cD_n$, $n\geq 1$, from permutations to diagonal rectangulations, first described by Law and Reading~\cite{MR2871762}.
Maps closely related to~$\rho$ have appeared previously in the literature, see e.g.~\cite{MR2244135,MR2763051}.
We consider the outer rectangle and divide its main diagonal into $n$ equally sized line segments numbered $1,\ldots,n$ from top-left to bottom-right.
Given a permutation $\pi=a_1\cdots a_n \in S_n$, the diagonal rectangulation $\rho(\pi)$ is obtained as follows; see Figure~\ref{fig:rho}:
For $i=1,\ldots,n$, in step~$i$ we add the rectangle~$r_{a_i}$ such that it intersects the main diagonal precisely in the $a_i$th line segment, and such that the rectangle is maximal w.r.t.\ the property that the rectangles $r_{a_1}\cup \cdots\cup r_{a_i}$ form a \emph{staircase}, which means that the bottom-left boundary of $r_{a_1}\cup\cdots\cup r_{a_i}$ is an~L-shape, and the top-right boundary is a non-increasing polygonal line.

With any wall~$w$ of a generic rectangulation~$R\in\cR_n$ we associate a \emph{wall shuffle} $\sigma(w)$, which is a permutation of a subset of the rectangles that share a side with~$w$, defined as follows; see Figure~\ref{fig:shuffle}.
If the wall~$w$ is horizontal, we move from the left endpoint of~$w$ to the right endpoint, and whenever we encounter a vertical wall~$w'$ that is incident to~$w$ from the bottom, we record the rectangle whose top side lies on~$w$ and left side lies on~$w'$, and if we encounter a vertical wall~$w'$ that is incident to~$w$ from the top, we record the rectangle whose bottom side lies on~$w$ and right side lies on~$w'$.
Clearly, we record all rectangles whose bottom or top side lies on~$w$, except the first rectangle below~$w$ and the last rectangle above~$w$.
On the other hand, if the wall~$w$ is vertical, we move from the bottom endpoint of~$w$ to the top endpoint, and whenever we encounter a horizontal wall~$w'$ that is incident to~$w$ from the left, we record the rectangle whose right side lies on~$w$ and bottom side lies on~$w'$, and if we encounter a horizontal wall~$w'$ that is incident to~$w$ from the right, we record the rectangle whose left side lies on~$w$ and top side lies on~$w'$.
In this case we record all rectangles whose left or right side lies on~$w$, except the first rectangle to the left of~$w$ and the last rectangle to the right of~$w$.
Observe that wall slides do not affect the rectangles that appear in a wall shuffle, but only their relative order in the shuffle.

\begin{figure}
\includegraphics{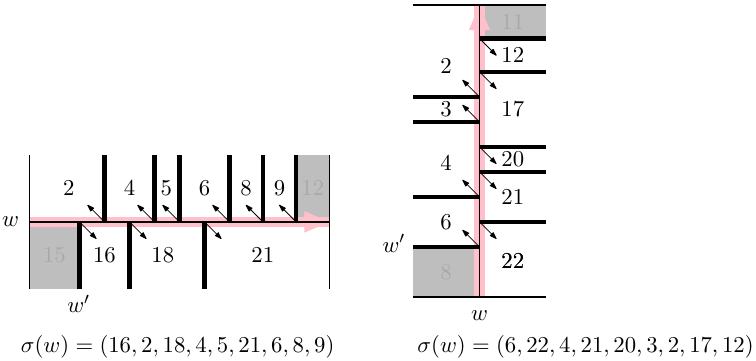}
\caption{Illustration of wall shuffles.}
\label{fig:shuffle}
\end{figure}

\begin{figure}
\includegraphics[page=2]{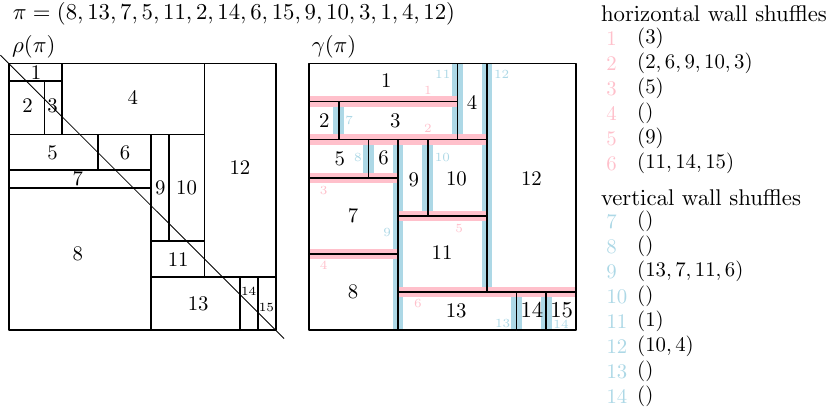}
\caption{Illustration of the surjection $\gamma:S_n\to\cR_n$.
This example continues Figure~\ref{fig:rho}.}
\label{fig:gamma}
\end{figure}

We are now in position to define the mapping $\gamma:S_n\rightarrow \cR_n$, $n\geq 1$, from permutations to generic rectangulations; see Figure~\ref{fig:gamma}.
From $\pi\in S_n$ we first construct the diagonal rectangulation $R:=\rho(\pi)\in\cD_n$.
Let $w$ be a horizontal wall in~$R$, and consider the rectangles in~$R$ whose bottom side lies on~$w$ from left to right.
By construction of~$\rho$, these rectangles form an increasing subsequence of~$\pi$.
Similarly, the rectangles in~$R$ whose top side lies on~$w$ from left to right form an increasing subsequence of~$\pi$.
Thus, we can specify a wall shuffle~$\sigma(w)$ by taking the subsequence of~$\pi$ that contains the appropriate rectangle numbers.
On the other hand, for a vertical wall~$w$ in~$R$, the rectangles in~$R$ whose left side lies on~$w$ from bottom to top form a decreasing subsequence of~$\pi$, and the rectangles whose right side lies on~$w$ form a decreasing subsequence of~$\pi$, so we can specify a wall shuffle~$\sigma(w)$ by taking the subsequence of~$\pi$ containing the appropriate rectangle numbers.
The rectangulation $\gamma(\pi)\in\cR_n$ is obtained from~$R\in\cD_n$ by applying wall slides to it, so as to obtain the wall shuffles specified by~$\pi$.

\begin{lemma}[{\cite[Prop.~4.2]{MR2864445}}]
The map $\gamma:S_n \rightarrow \cR_n$ is surjective.
\end{lemma}

Even though $\gamma$ is not a bijection, Reading~\cite{MR2864445} showed that it becomes a bijection when restricting the domain to 2-clumped permutations.

\begin{theorem}[{\cite[Thm.~4.1]{MR2864445}}]
\label{thm:2clumped-rect}
The map $\gamma$ is a bijection when restricted to the set~$S_n'$ of 2-clumped permutations.
\end{theorem}

\subsection{The connection between permutations and rectangulations}

The key lemma of this section, Lemma~\ref{lem:gamma-jump} below, asserts that deletion, insertion and jumps in permutations as defined in Section~\ref{sec:perm-ops} are in bijective correspondence under~$\gamma$ to deletion, insertion and jumps in generic rectangulations as defined in Sections~\ref{sec:deletion}, \ref{sec:insertion} and~\ref{sec:jumps}.
In order to prove it, we first establish a coarser version of this statement for the mapping~$\rho$.

\begin{lemma}
\label{lem:rho-jump}
Let $\pi=a_1\cdots a_n\in S_n$, $n\geq 1$, and define $P:=\rho(\pi)\in\cD_n$.
In the sequence~$I(P)=(q_1,\ldots,q_\nu)$, consider the subsequence~$I'(P)=(q_{j_1},\ldots,q_{j_\mu})$ consisting of the first insertion point of every vertical group, and the last insertion point of every horizontal group.
Then we have $p(\rho(c_i(\pi)))=P$ for all $i=1,\ldots,n+1$, and the sequence~$1,\ldots,n+1$ can be partitioned into consecutive nonempty intervals~$I_1,\ldots,I_\mu$ with the following properties:
\begin{enumerate}[label=(\alph*), leftmargin=8mm, noitemsep, topsep=3pt plus 3pt]
\item for every $k=1,\ldots,\mu$ and every $i\in I_k$ we have $\rho(c_i(\pi))=c_{j_k}(P)$;
\item for any interval $I_k=[\ihat,\icheck]$, $1<k<\mu$, such that the top-left vertex of~$r_{n+1}$ in~$\rho(c_i(\pi))$, $i\in I_k$, has type~$\rightT$, we have that the rectangle~$r_{a_{\ihat-1}}$ is the unique rectangle left of~$r_{n+1}$, and the rectangle~$r_{a_\icheck}$ is the leftmost rectangle above~$r_{n+1}$;
\item for any interval $I_k=[\ihat,\icheck]$, $1<k<\mu$, such that the top-left vertex of~$r_{n+1}$ in~$\rho(c_i(\pi))$, $i\in I_k$, has type~$\bottomT$, we have that the rectangle~$r_{a_{\ihat-1}}$ is the topmost rectangle left of~$r_{n+1}$, and the rectangle~$r_{a_\icheck}$ is the unique rectangle above~$r_{n+1}$;
\item we have $I_1=\{1\}$ and $I_\mu=\{n+1\}$.
\end{enumerate}
\end{lemma}

\begin{proof}
For the reader's convenience, the proof is illustrated in Figure~\ref{fig:rho-jump}.
Recall the definition of the mapping~$\rho$ via the process described in Section~\ref{sec:surjection} and illustrated in Figure~\ref{fig:rho}.
Using the definition of the rectangle insertion from Section~\ref{sec:insertion}, we first observe that $\rho(c_1(\pi))=c_1(P)$ and $\rho(c_{n+1}(\pi))=c_\nu(P)$.
We consider the sequence of permutations $c_i(\pi)$ for $i=1,\ldots,n+1$ and their images under~$\rho$.
Observe that $c_{i+1}(\pi)$ is obtained from~$c_i(\pi)$ by the adjacent transposition $(n+1)a_i\longrightarrow a_i(n+1)$.
We consider the construction of~$R:=\rho(c_i(\pi))$, specifically the two steps in which the rectangles~$r_{n+1}$ and~$r_{a_i}$ are added, and we analyze how the swapped insertion order of these two rectangles changes the resulting rectangulation~$R':=\rho(c_{i+1}(\pi))$.
Clearly, in both~$R$ and~$R'$, the rectangle~$r_{n+1}$ is either left-extended and top-fixed, or left-fixed and top-extended, and we treat both cases separately.
In particular, these two cases are not symmetric.

\begin{figure}
\includegraphics{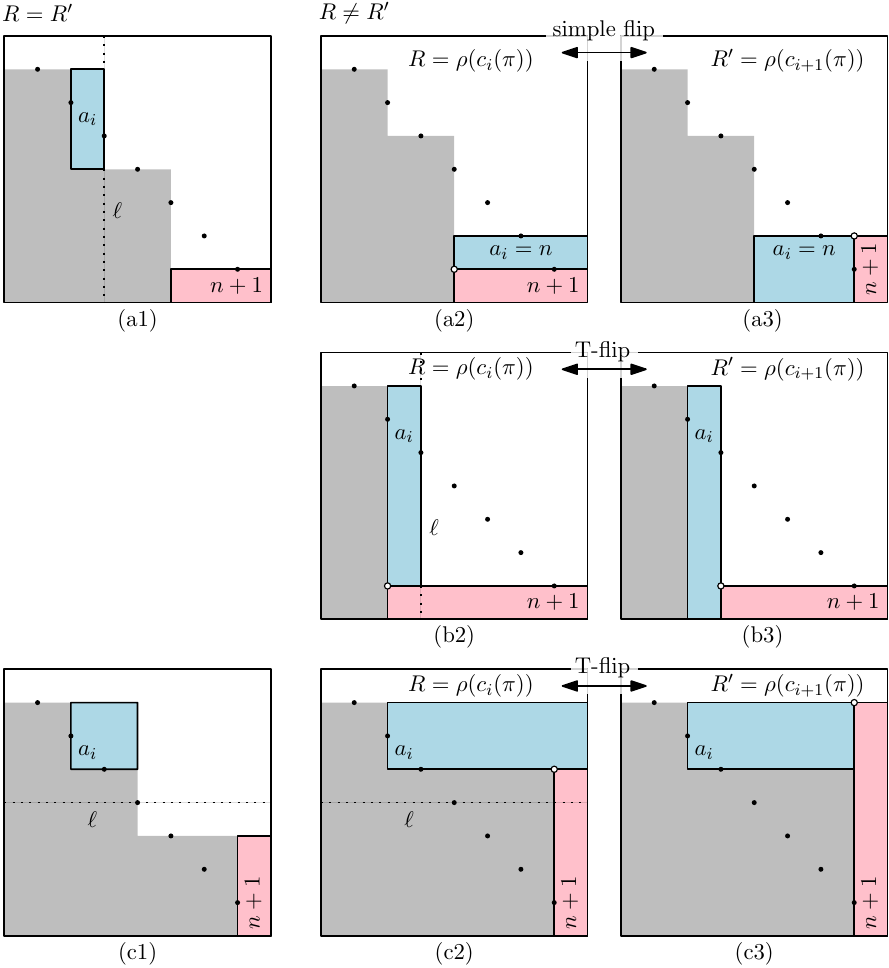}
\caption{Illustration of the proof of Lemma~\ref{lem:rho-jump}.}
\label{fig:rho-jump}
\end{figure}

\textbf{Case~(i):} $r_{n+1}$ is left-extended and top-fixed in~$R$.
This case is shown in the top and middle part of Figure~\ref{fig:rho-jump}.
We distinguish two subcases, namely $a_i=n$ and $a_i<n$.

\textbf{Case~(ia):} $a_i=n$.
In this case the top side of~$r_{n+1}$ coincides with the bottom side of~$r_{a_i}=r_n$ in~$R$, as both rectangles left-extend to the same vertical line by the staircase property and the fact that every rectangle intersects the main diagonal; see Figure~\ref{fig:rho-jump}~(a2).
Consequently, inserting them in the swapped order, first $r_n$ and then~$r_{n+1}$, produces a sub-rectangulation that differs in a simple flip of the wall between these two rectangles; see Figure~\ref{fig:rho-jump}~(a3).
As the height of the top side of~$r_n$ is not determined by~$r_{n+1}$, but only by previous rectangles $r_{a_1},\ldots,r_{a_{i-1}}$, both insertion orders produce the same staircase $r_{a_1}\cup\cdots \cup r_n\cup r_{n+1}$, which by the definition of~$\rho$ is all that matters for the next construction steps.
Consequently, $R$ and~$R'$ differ in a simple flip of the rectangles~$r_{n+1}$ and~$r_n$.
Moreover, by the definition of rectangle deletion given in Section~\ref{sec:deletion} we have $p(R)=p(R')=P$, and by the definition of rectangle insertion and the definition of~$I'(P)$ we have $R=c_{j_k}(P)$ and $R'=c_{j_{k+1}}(P)$ for some index~$k$.
Note that in~$R$, the top-left vertex of~$r_{n+1}$ has type~$\rightT$ and $r_{a_\icheck}$, $\icheck:=i$, is the leftmost rectangle above~$r_{n+1}$.
Furthermore, in~$R'$, the top-left vertex of~$r_{n+1}$ has type~$\bottomT$ and $r_{a_{\ihat-1}}=r_{a_i}$, $\ihat:=i+1$, is the topmost rectangle left of~$r_{n+1}$.

\textbf{Case~(ib):} $a_i<n$.
If $r_{a_i}$ is bottom-fixed in~$R$, or if $r_{n+1}$ does not left-extend beyond the vertical line~$\ell$ through the right endpoint of the $a_i$th interval on the main diagonal, then the rectangles $r_{n+1}$ and $r_{a_i}$ do not touch; see Figure~\ref{fig:rho-jump}~(a1).
Consequently, inserting them in the swapped order, first~$r_{a_i}$ and then~$r_{n+1}$, produces the same sub-rectangulation.
It follows that $R=R'$, i.e., $\rho(c_i(\pi))=\rho(c_{i+1}(\pi))$, and therefore trivially $p(R)=p(R')$.

On the other hand, if $r_{a_i}$ is bottom-extended and $r_{n+1}$ left-extends beyond~$\ell$, then $r_{a_i}$ must be right-fixed because of~$a_i<n$ (otherwise $r_{a_i}$ would reach into the $(a_i+1)$st line segment on the main diagonal).
Moreover, because of the staircase property the top side of~$r_{n+1}$ contains the bottom side of~$r_{a_i}$, and both rectangles left-extend to the same vertical line; see Figure~\ref{fig:rho-jump}~(b2).
Consequently, inserting them in the swapped order, first~$r_{a_i}$ and then~$r_{n+1}$ produces a sub-rectangulation that differs in a T-flip from~$\topT$ to~$\rightT$ around the top-right common vertex of these two rectangles; see Figure~\ref{fig:rho-jump}~(b3).
As the height of the top side of~$r_{a_i}$ is not determined by~$r_{n+1}$, and as $r_{a_i}$ is right-fixed and $r_{n+1}$ is top-fixed in both insertion orders, we obtain that~$R$ and~$R'$ differ in a T-flip around the top-right common vertex of the rectangles~$r_{n+1}$ and~$r_{a_i}$.
Moreover, by the definition of rectangle insertion and deletion and the definition of~$I'(P)$, we have $p(R)=p(R')=P$, and $R=c_{j_k}(P)$ and $R'=c_{j_{k+1}}(P)$ for some index~$k$.
Note that in~$R$, the top-left vertex of~$r_{n+1}$ has type~$\rightT$ and $r_{a_\icheck}$, $\icheck:=i$, is the leftmost rectangle above~$r_{n+1}$.
Furthermore, in~$R'$, the top-left vertex of~$r_{n+1}$ has type~$\rightT$ and $r_{a_{\ihat-1}}=r_{a_i}$, $\ihat:=i+1$, is the unique rectangle left of~$r_{n+1}$.

\textbf{Case~(ii):} $r_{n+1}$ is left-fixed and top-extended in~$R$.
This case is shown in the bottom part of Figure~\ref{fig:rho-jump}.
As $r_{n+1}$ is top-extended, the staircase property implies that~$a_i<n$, as $n$ must be among the first entries $a_1,\ldots,a_{i-1}$ of~$\pi$.

If $r_{a_i}$ is right-fixed in~$R$, or if $r_{n+1}$ does not top-extend beyond the horizontal line~$\ell$ through the right endpoint of the $(a_i+1)$st interval on the main diagonal, then the rectangles~$r_{n+1}$ and~$r_{a_i}$ do not touch; see Figure~\ref{fig:rho-jump}~(c1).
Consequently, inserting them in the swapped order, first~$r_{a_i}$ and then~$r_{n+1}$, produces the same sub-rectangulation.
It follows that $R=R'$, i.e., $\rho(c_i(\pi))=\rho(c_{i+1}(\pi))$, and therefore trivially $p(R)=p(R')$.

On the other hand, if $r_{a_i}$ is right-extended and $r_{n+1}$ top-extends beyond~$\ell$, then $r_{a_i}$ must be bottom-fixed because of~$a_i<n$ (otherwise $r_{a_i}$ would reach into the $(a_i+1)$st line segment on the main diagonal).
Moreover, because of the staircase property the top side of~$r_{n+1}$ must lie on the horizontal line through the bottom side of~$r_{a_i}$, and therefore~$r_{a_i}$ right-extends to the right outer boundary, i.e., the bottom side of $r_{a_i}$ contains the top side of~$r_{n+1}$; see Figure~\ref{fig:rho-jump}~(c2).
Consequently, inserting them in the swapped order, first~$r_{a_i}$ and then~$r_{n+1}$ produces a sub-rectangulation that differs in a T-flip from~$\bottomT$ to~$\leftT$ around the bottom-left common vertex of these two rectangles; see Figure~\ref{fig:rho-jump}~(c3).
As the height of the top side of~$r_{a_i}$ is not determined by~$r_{n+1}$ in the two insertion orders, we obtain that~$R$ and~$R'$ differ in a T-flip around the bottom-left common vertex of the rectangles~$r_{n+1}$ and~$r_{a_i}$.
Moreover, by the definition of rectangle insertion and deletion and the definition of~$I'(P)$, we have $p(R)=p(R')=P$, and $R=c_{j_k}(P)$ and $R'=c_{j_{k+1}}(P)$ for some index~$k$.
Note that in~$R$, the top-left vertex of~$r_{n+1}$ has type~$\bottomT$ and $r_{a_\icheck}$, $\icheck:=i$, is the unique rectangle above~$r_{n+1}$.
Furthermore, in~$R'$, the top-left vertex of~$r_{n+1}$ has type~$\bottomT$ and $r_{a_{\ihat-1}}=r_{a_i}$, $\ihat:=i+1$, is the topmost rectangle left of~$r_{n+1}$.

This proves~(a), (b) and~(c).
For~(d) observe that the rectangle in the bottom-left corner of~$\rho(\pi)$ is $r_{a_1}$, whereas the rectangle in the top-right corner is~$r_{a_n}$, so $\rho(c_1(\pi))\neq \rho(c_2(\pi))$ and $\rho(c_n(\pi))\neq \rho(c_{n+1}(\pi))$.
This completes the proof of the lemma.
\end{proof}

\begin{lemma}
\label{lem:gamma-jump}
Let $\pi=a_1\cdots a_n\in S_n$, $n\geq 1$, and consider the rectangulation $P:=\gamma(\pi)\in\cR_n$ with $\nu=\nu(P)$ insertion points.
Then we have $p(\gamma(c_i(\pi)))=P$ for all $i=1,\ldots,n+1$, and the sequence~$i=1,\ldots,n+1$ can be partitioned into consecutive nonempty intervals~$I_1,\ldots,I_\nu$, such that for every $k=1,\ldots,\nu$ and every $i\in I_k$ we have $\gamma(c_i(\pi))=c_k(P)$.
Furthermore, if $\pi$ is 2-clumped, then each interval~$I_k$ contains exactly one 2-clumped permutation.
\end{lemma}

The proof of Lemma~\ref{lem:gamma-jump} shows that $I_1=\{1\}$ and $I_\nu=\{n+1\}$.

\begin{proof}
For any $i=1,\ldots,n+1$, consider the permutation~$c_i(\pi)$, and the two diagonal rectangulations~$P':=\rho(\pi)\in\cD_n$ and $R':=\rho(c_i(\pi))\in\cD_n$.
From Lemma~\ref{lem:rho-jump} we know that $P'=p(R')$ and $R'=c_j(P')$ for some index~$j$ such that the $j$th insertion point in~$I(P')$ is the first of a vertical group, or the last of a horizontal group.
By the definitions from Section~\ref{sec:surjection}, $\gamma$ consists of applying~$\rho$ plus wall slides that are determined by the wall shuffles of~$P'$ and~$R'$ and the relative order of the rectangle indices in those shuffles in~$\pi$ and~$c_i(\pi)$, i.e., $P=\gamma(\pi)$ and $R:=\gamma(c_i(\pi))$ are obtained from~$P'$ and~$R'$ by wall slides.
Observe that in~$R'$ and~$R$, the bottom-right rectangle~$r_{n+1}$ is contained in at most one wall shuffle.
Specifically, if the top-left corner of~$r_{n+1}$ has type~$\rightT$ and does not lie on the left boundary of the rectangulation, then $r_{n+1}$ participates in a single wall shuffle of the vertical wall that contains the left side of~$r_{n+1}$, whereas if the top-left corner of~$r_{n+1}$ has type~$\bottomT$ and does not lie on the upper boundary of the rectangulation, then $r_{n+1}$ participates in a single wall shuffle of the horizontal wall that contains the top side of~$r_{n+1}$.
In particular, in the cases $j=1$ and $j=\nu(P')$ the rectangle~$r_{n+1}$ is not contained in any wall shuffle.
As the elements of all subsequences of~$\pi$ and~$c_i(\pi)$ that do not contain~$n+1$ appear in the same relative order in both permutations, all wall shuffles of~$P$ and $R$ are the same, except the wall shuffle of~$R$ containing~$n+1$, which is obtained from a wall shuffle of~$P$ by inserting the value~$n+1$.
We conclude that $p(R)=P$ and $R=c_k(P)$ for some index~$k$.

The desired interval partition $I_1,\ldots,I_\nu$ of the indices $1,\ldots,n+1$ is obtained by refining the partition $I_1',\ldots,I_\mu'$ guaranteed by Lemma~\ref{lem:rho-jump}, where $\mu$ is the number of vertical and horizontal groups of~$P'$, which is the same for~$P$, as wall slides do not affect it.
As $I_1'=\{1\}$ and $I_\mu'=\{n+1\}$ by Lemma~\ref{lem:rho-jump}~(d), these two intervals are not refined, so we have $I_1:=I_1'=\{1\}$ and $I_\nu:=I_\mu=\{n+1\}$.
It remains to consider the remaining intervals~$I_k'$, $1<k<\mu$.

First consider an interval~$I_k'=:[\ihat,\icheck]$, $1<k<\mu$, such that in $R':=\rho(c_i(\pi))\in\cD_n$, $i\in I_k$, the top-left vertex has type~$\rightT$ (recall Lemma~\ref{lem:rho-jump}~(a)).
By Lemma~\ref{lem:rho-jump}~(b), in~$R'$ the rectangle~$r_{a_{\ihat-1}}$ is the unique rectangle left of~$r_{n+1}$, and the rectangle~$r_{a_\icheck}$ is the leftmost rectangle above~$r_{n+1}$.
Consider the wall shuffle~$\sigma(w)$ of the vertical wall~$w$ between~$r_{a_{\ihat-1}}$ and~$r_{a_\icheck}$ in~$P$.
It has the form $\sigma(w)=(b_1,\ldots,b_\lambda,a_{\icheck},\ldots)$, for some $\lambda\geq 0$, i.e., the rectangles~$r_{b_1},\ldots,r_{b_\lambda}$ are stacked on top of~$r_{a_{\ihat-1}}$ and their bottom sides are incident with~$w$ below the incidence of the top side of~$r_\icheck$ with~$w$.
It follows that $\pi$ contains the subsequence $a_{\ihat-1},b_1,\ldots,b_\lambda,a_\icheck$, and so the permutations $c_i(\pi)$ for $i=\ihat,\ldots,\icheck$ have the value~$n+1$ appear at every possible position within the subsequence $b_1,\ldots,b_\lambda$.
Consequently, the interval~$I_k'$ is refined into $\lambda$ subintervals such that $\gamma(c_i(\pi))=\gamma(c_{i+1}(\pi))$ if $i,i+1$ are in the same subinterval and $\gamma(c_i(\pi))=c_\ell(P)$ and $\gamma(c_{i+1}(\pi))=c_{\ell+1}(P)$ for some index~$\ell$ if $i,i+1$ are in consecutive subintervals.

Now consider an interval~$I_k'=:[\ihat,\icheck]$, $1<k<\mu$, such that in $R':=\rho(c_i(\pi))\in\cD_n$, $i\in I_k$, the top-left vertex has type~$\bottomT$ (recall Lemma~\ref{lem:rho-jump}~(a)).
By Lemma~\ref{lem:rho-jump}~(c), in~$R'$ the rectangle~$r_{a_{\ihat-1}}$ is the topmost rectangle left of~$r_{n+1}$, and the rectangle~$r_{a_\icheck}$ is the unique rectangle above~$r_{n+1}$.
Consider the wall shuffle~$\sigma(w)$ of the horizontal wall~$w$ between~$r_{a_{\ihat-1}}$ and~$r_{a_\icheck}$ in~$P$.
It has the form $\sigma(w)=(\ldots,a_{\ihat-1},b_1,\ldots,b_\lambda)$, for some $\lambda\geq 0$, i.e., the rectangles~$r_{b_\lambda},\ldots,r_{b_1}$ are stacked to the left of~$r_{a_\icheck}$ and their right sides are incident with~$w$ to the right of the incidence of the left side of~$r_{a_{\ihat-1}}$ with~$w$.
It follows that $\pi$ contains the subsequence $a_{\ihat-1},b_1,\ldots,b_\lambda,a_\icheck$, and so the permutations $c_i(\pi)$ for $i=\ihat,\ldots,\icheck$ have the value~$n+1$ appear at every possible position within the subsequence $b_1,\ldots,b_\lambda$.
Consequently, the interval~$I_k'$ is refined into $\lambda$ subintervals such that $\gamma(c_i(\pi))=\gamma(c_{i+1}(\pi))$ if $i,i+1$ are in the same subinterval and $\gamma(c_i(\pi))=c_\ell(P)$ and $\gamma(c_{i+1}(\pi))=c_{\ell+1}(P)$ for some index~$\ell$ if $i,i+1$ are in consecutive subintervals.

It remains to argue that each set of permutations~$C_k:=\{c_i(\pi)\mid i\in I_k\}$, $k=1,\ldots,\nu$, contains exactly one 2-clumped permutation.
Indeed, $C_k$ can contain at most one 2-clumped permutation by~Theorem~\ref{thm:2clumped-rect}, as all $\pi\in C_k$ have the same image under~$\gamma$.
Suppose for the sake of contradiction that~$C_k$ contains no 2-clumped permutation.
Then by Theorem~\ref{thm:2clumped-rect} there is another 2-clumped permutation~$\rho\in S_{n+1}$ with $\rho\notin C:=\{c_i(\pi)\mid i=1,\ldots,n+1\}$ and~$\gamma(\rho)=\gamma(c_i(\pi))$ for all $i\in C_k$.
However, by Lemma~\ref{lem:2clumped-zigzag}, the permutation $\rho':=p(\rho)\in S_n$ is also 2-clumped, i.e., we have $\rho'\in S_n'$.
Moreover, we have $\rho'\neq\pi$ as $\rho\notin C$, and by Lemma~\ref{lem:gamma-jump} we have $\gamma(\rho')=\gamma(\pi)=R$, a contradiction to the fact that~$\gamma$ is a bijection between~$S_n'$ and~$\cR_n$ by Theorem~\ref{thm:2clumped-rect}.
\end{proof}

\subsection{Proof of Theorem~\ref{thm:jump-rect}}

With Lemma~\ref{lem:gamma-jump} in hand, we are now in position to present the proof of Theorem~\ref{thm:jump-rect}.

\begin{proof}[Proof of Theorem~\ref{thm:jump-rect}]
Consider a zigzag set of rectangulations~$\cC_n\seq\cR_n$, and consider the zigzag sets~$\cC_{i-1}:=\{p(R)\mid R\in\cC_i\}$ for $i=n,n-1,\ldots,2$.
By Lemma~\ref{lem:gamma-jump}, for all $i=1,\ldots,n$ there is a set of 2-clumped permutations~$L_i\seq S_i'$ such that $\gamma$ restricted to~$L_i$ is a bijection between~$L_i$ and~$\cC_i$, and such that $L_{i-1}=\{p(\pi)\mid \pi\in L_i\}$ for all $i=2,\ldots,n$.
Moreover, as~$\cC_i$ is a zigzag set, we know that for all~$R\in\cC_{i-1}$ we have $c_1(R)\in\cC_i$ and $c_{\nu(R)}(R)\in\cC_i$, for all $i=2,\ldots,n$.
By Lemma~\ref{lem:gamma-jump}, this implies that for all~$\pi\in L_{i-1}$ we have $c_1(\pi)\in L_i$ and $c_i(R)\in L_i$, for all $i=2,\ldots,n$, i.e., $L_n$ is a zigzag language of 2-clumped permutations (using that $L_1=\{1\}$ and $L_0:=\{p(\pi)\mid \pi\in L_1\}=\{\varepsilon\}$ are zigzag languages).

By Theorem~\ref{thm:jump}, Algorithm~J visits every permutation from~$L_n$ exactly once, in the order~$J(L_n)$ defined by~\eqref{eq:JLn}.
From Lemma~\ref{lem:JLn-prop}~(d) we obtain that if Algorithm~J performs a jump of some value~$j$ in the current permutation~$\pi\in L_n$, then the corresponding rectangulations $R^{[k]}:=\gamma(\pi^{[k]})$ for $k=j+1,\ldots,n$ are either bottom-based or right-based.
Using the definition of jumps from Section~\ref{sec:jumps} and Lemma~\ref{lem:gamma-jump}, a minimal left/right jump of the value~$j$ in~$\pi\in L_n$, as performed by Algorithm~J, corresponds to a minimal left/right jump of the rectangle~$r_j$ in~$\gamma(\pi)\in\gamma(L_n)=\cC_n$, as performed by Algorithm~\Jrect{}.
This together with the observation that $\gamma(\ide_n)=\idrect$ proves the first part of the theorem.
Specifically, the ordering of rectangulations generated by Algorithm~\Jrect{} is
\begin{equation}
\label{eq:JrectLn}
\Jrectm(\cC_n)=\gamma(J(L_n)).
\end{equation}
To prove the second part of the theorem, by Theorem~\ref{thm:jump} it suffices to show that $|\cC_i|=|\gamma(L_i)|$ is even for all $2\leq i\leq n-1$.
For any rectangulation~$R\in\cR_n$, we write $\lambda(R)\in\cR_n$ for the rectangulation obtained by reflection at the main diagonal.
First observe that if $R,R'\in\cR_n$, $n\geq 2$, satisfy $R'=\lambda(R)$, then we also have $p(R')=\lambda(p(R))$.
Consequently, the assumption that $\cC_n$ is symmetric implies that $\cC_i$ is symmetric for all $i=1,\ldots,n$.
Consider a rectangulation $R\in\cR_n$, $n\geq 2$, and observe that if the top-left vertex of the bottom-right rectangle~$r_n$ of~$R$ has type~$\rightT$, then it has type~$\bottomT$ in~$\lambda(R)$, and vice versa.
It follows that $\lambda$ is an involution without fixed points on~$\cC_i$ for all $i=2,\ldots,n$, proving that $|\cC_i|$ is even.

This completes the proof.
\end{proof}

\subsection{Memoryless generation of permutations}

Consider Algorithm~M below, which takes as input a zigzag language of permutations~$L_n\seq S_n$ and generates them exhaustively by minimal jumps in the same order as Algorithm~J, i.e., in the order~$J(L_n)$.

\begin{algo}{\bfseries Algorithm~M}{Memoryless minimal jumps}
This algorithm generates all permutations of a zigzag language $L_n\seq S_n$ by minimal jumps in the same order as Algorithm~J.
It maintains the current permutation in the variable~$\pi$, and auxiliary arrays $o=(o_1,\ldots,o_n)$ and $s=(s_1,\ldots,s_n)$.
\begin{enumerate}[label={\bfseries M\arabic*.}, leftmargin=8mm, noitemsep, topsep=3pt plus 3pt]
\item{} [Initialize] Set $\pi\gets \ide_n=12\cdots n$, and $o_j\gets \dirl$, $s_j\gets j$ for $j=1,\ldots,n$.
\item{} [Visit] Visit the current permutation $\pi$.
\item{} [Select value] Set $j\gets s_n$, and terminate if $j=1$.
\item{} [Jump value] In the current permutation~$\pi$, perform a jump of the value~$j$ that is minimal w.r.t.~$L_n$, where the jump direction is left if $o_j=\dirl$ and right if $o_j=\dirr$.
\item{} [Update $o$ and $s$] Set $s_n\gets n$.
If $o_j=\dirl$ and $j$ is at the first position in~$\pi$ or the value left of it is bigger than~$j$ set $o_j\gets \dirr$, or if $o_j=\dirr$ and $j$ is at the last position in~$\pi$ or the value right of it is bigger than~$j$ set $o_j\gets \dirl$, and in both cases set $s_j\gets s_{j-1}$ and $s_{j-1}\gets j-1$. Go back to~M2.
\end{enumerate}
\end{algo}

\begin{theorem}
\label{thm:algo}
For any zigzag language of permutations~$L_n\seq S_n$, $n\geq 1$, Algorithm~M visits every rectangulation from~$L_n$ exactly once, in the order~$J(L_n)$ defined by~\eqref{eq:JLn}.
\end{theorem}

The rest of this section is devoted to proving Theorem~\ref{thm:algo}.

For any $\pi$ in the sequence~$J(L_n)$ we define a sequence $s_n^\pi=(s_{n,1}^\pi,\ldots,s_{n,n}^\pi)$ as follows:
If $n=1$ we have $J(L_1)=\pi$ with $\pi:=1$ and we define $s_1^\pi=(1)$.
If $n\geq 2$, we consider the permutation~$\pi':=p(\pi)\in S_{n-1}$ in the sequence $J(L_{n-1})$, and we define $\ol{c}(\pi'):=\lvec{c}(\pi')$ if $\pi'$ appears at an odd position in~$J(L_{n-1})$, or $\ol{c}(\pi'):=\rvec{c}(\pi')$ if $\pi'$ appears at an even position.
If $\pi$ is not the last permutation in~$\ol{c}(\pi')$ we define
\begin{subnumcases}{s_{n,i}^{\pi} := \label{eq:spi1}}
s_{n-1,i}^{\pi'} & $\text{if } i\leq n-1,$ \label{eq:spi11} \\
n & $\text{if } i = n,$ \label{eq:spi12}
\end{subnumcases}
for $i=1,\ldots,n$, and if $\pi$ is the last permutation in~$\ol{c}(\pi')$ we define
\begin{subnumcases}{s_{n,i}^{\pi} := \label{eq:spi2}}
s_{n-1,i}^{\pi'} & $\text{if } i\leq n-2,$ \label{eq:spi21} \\
n-1 & $\text{if } i=n-1,$ \label{eq:spi22} \\
s_{n-1,n-1}^{\pi'} & $\text{if } i=n,$ \label{eq:spi23}
\end{subnumcases}
for $i=1,\ldots,n$.

The following lemma captures important properties of the sequences defined in this way.

\begin{lemma}
\label{lem:sn-prop}
The sequences defined in~\eqref{eq:spi1} and~\eqref{eq:spi2} have the following properties.
\begin{enumerate}[label=(\alph*),leftmargin=7mm, noitemsep, topsep=3pt plus 3pt]
\item For the first permutation~$\pi=\ide_n$ in the sequence~$J(L_n)$, we have $s_n^\pi=(1,2,\ldots,n)$.
\item For any two consecutive permutations~$\pi,\rho$ in the sequence~$J(L_n)$, $\rho$ is obtained from~$\pi$ by a jump of the value~$s_{n,n}^\pi$.
\item For the last permutation~$\pi$ in~$J(L_n)$ we have~$s_{n,n}^\pi=1$.
\end{enumerate}
Moreover, for any three consecutive permutations~$\pi,\rho,\sigma$ in~$J(L_n)$ we have:
\begin{enumerate}[label=(\alph*),leftmargin=7mm, noitemsep, topsep=3pt plus 3pt]
\setcounter{enumi}{2}
\item If $\pi$ and~$\rho$ differ in a jump of~$n$, and $\rho$ and~$\sigma$ differ in a jump of~$n$, then we have $s_{n,i}^\rho=s_{n,i}^\pi$ for $i=1,\ldots,n-1$.
\item If $\pi$ and~$\rho$ differ in a jump of~$n$, and $\rho$ and~$\sigma$ differ in a jump of~$j<n$, then we have $s_{n,i}^\rho=s_{n,i}^\pi$ for $i\in\{1,\ldots,n-2\}$ and $s_{n,n-1}^\rho=n-1$.
\item If $\pi$ and~$\rho$ differ in a jump of~$j<n$, $\rho$ and~$\sigma$ differ in a jump of~$n$, and $j$ is not at a boundary position in~$p^{n-j}(\rho)$, then we have $s_{n,i}^\rho=s_{n,i}^\pi$ for $i=1,\ldots,n-1$.
\item If $\pi$ and~$\rho$ differ in a jump of~$j<n$, $\rho$ and~$\sigma$ differ in a jump of~$n$, and $j$ is at a boundary position in~$p^{n-j}(\rho)$, then we have $s_{n,i}^\rho=s_{n,i}^\pi$ for $i\in\{1,\ldots,n-1\}\setminus\{j-1,j\}$, $s_{n,j-1}^\rho=j-1$ and $s_{n,j}^\rho=s_{n,j-1}^\pi$.
\end{enumerate}
\end{lemma}

\begin{proof}
We prove these properties by induction on~$n$.
The induction basis $n=1$ is trivial.
For the induction step let $n\geq 2$ and assume that all properties hold for the sequence~$J(L_{n-1})$.

We first show the induction step for~(a), (b) and~(c).

Consider a permutation~$\pi$ in~$J(L_n)$ and define $\pi':=p(\pi)\in S_{n-1}$.

To prove~(a), let $\pi=\ide_n$ be the first permutation in~$J(L_n)$ (recall Lemma~\ref{lem:JLn-prop}~(a)).
We have $\pi'=\ide_{n-1}$ and by induction and~(a) we hence have $s_n^{\pi'}=(1,2,\ldots,n-1)$.
Using~\eqref{eq:spi1} we obtain~$s_n^\pi=(1,2,\ldots,n)$, as claimed.

To prove~(b), let $\pi$ be a permutation that is not the last in the sequence~$J(L_n)$, and let $\rho$ be the permutation succeeding~$\pi$ in~$J(L_n)$.
If $\pi$ is not the last permutation in the subsequence~$\ol{c}(\pi')$ of~$J(L_n)$, then by~\eqref{eq:JLn} the permutation~$\rho$ is obtained from~$\pi$ by a jump of the value~$n$, and then (b) follows directly from~\eqref{eq:spi12}.
On the other hand, if $\pi$ is the last permutation in the subsequence~$\ol{c}(\pi')$, then $\rho$ is obtained from~$\pi$ by a jump of the value~$s_{n-1,n-1}^{\pi'}$ by induction and~(b), and by~\eqref{eq:spi23} we have $s_{n,n}^\pi=s_{n-1,n-1}^{\pi'}$, as desired.

To prove~(c), let $\pi$ be the last permutation in~$J(L_n)$.
Then the permutation~$\pi'$ is also the last permutation in~$J(L_{n-1})$, so by induction we have $s_{n-1,n-1}^{\pi'}=1$.
Using~\eqref{eq:spi23} we see that $s_{n,n}^\pi=s_{n-1,n-1}^{\pi'}=1$, as desired.

To prove~(d), note that $p(\pi)=p(\rho)=p(\sigma)$ and therefore $\rho$ is not the last permutation in the subsequence~$\ol{c}(p(\rho))$, so the claim follows directly from~\eqref{eq:spi11}.

To prove~(e), note that $p(\pi)=p(\rho)\neq p(\sigma)$ and therefore $\rho$ is the last permutation in the subsequence~$\ol{c}(p(\rho))$, so the claim follows directly from~\eqref{eq:spi11}, \eqref{eq:spi21} and~\eqref{eq:spi22}.

To prove~(f) and~(g), let $\pi':=p(\pi)$, $\rho':=p(\rho)$ and $\sigma':=p(\sigma)$.
Note that $\pi'\neq \rho'=\sigma'$ and therefore $\pi$ is the last permutation in the subsequence~$\ol{c}(\pi')$, whereas $\rho$ is the first permutation in the subsequence~$\ol{c}(\rho')$.
Consequently, $s_n^\pi$ is defined by~\eqref{eq:spi2} and $s_n^\rho$ is defined by~\eqref{eq:spi1}.
In particular, we have $s_{n,n-1}^\pi=n-1$ by~\eqref{eq:spi22} and $s_{n,n-1}^\rho=s_{n-1,n-1}^{\rho'}$ by~\eqref{eq:spi11}.

We first prove~(f), and we distinguish whether~$j=n-1$ or~$j<n-1$.
If $j=n-1$, then $\pi'$ and~$\rho'$ differ in a jump of~$j=n-1$, and as~$j$ is not at a boundary position in~$\rho'$, $\rho'$ and~$\sigma'$ also differ in a jump of~$n-1$ by~\eqref{eq:JLn}.
Consequently, we have $s_{n-1,i}^{\rho'}=s_{n-1,i}^{\pi'}$ for $i=1,\ldots,n-2$ by induction and~(d), and $s_{n-1,n-1}^{\rho'}=n-1$ by~\eqref{eq:spi12}, so (f) indeed holds in this case.

If $j<n-1$, then $\pi'$ and~$\rho'$ differ in a jump of~$j<n-1$, and $\rho'$ and~$\sigma'$ differ in a jump of~$n-1$ by~\eqref{eq:JLn}.
As $j$ is not at a boundary position in~$p^{n-j}(\rho)=p^{n-1-j}(\rho')$, we have $s_{n-1,i}^{\rho'}=s_{n-1,i}^{\pi'}$ for $i=1,\ldots,n-2$ by induction and~(f), and $s_{n-1,n-1}^{\rho'}=n-1$ by induction and~(b), so (f) indeed holds in this case.

We now prove~(g), and again we distinguish whether~$j=n-1$ or~$j<n-1$.
If $j=n-1$, then $\pi'$ and~$\rho'$ differ in a jump of~$j=n-1$, in particular $p(\pi')=p(\rho')$, and as~$j=n-1$ is at a boundary position in~$p^{n-j}(\rho)=p(\rho)=\rho'$, $\rho'$ and~$\sigma'$ differ in a jump of some value smaller than~$n-1$ by~\eqref{eq:JLn}.
Consequently, we have $s_{n-1,i}^{\rho'}=s_{n-1,i}^{\pi'}$ for $i=1,\ldots,n-3$ and $s_{n-1,n-2}^{\rho'}=n-2$ by induction and~(e).
Moreover, we have $s_{n-1,n-1}^{\rho'}=s_{n-2,n-2}^{p(\rho')}=s_{n-2,n-2}^{p(\pi')}=s_{n-1,n-2}^{\pi'}$ by~\eqref{eq:spi11}.
Combining these observations shows that indeed~(g) holds in this case.

If $j<n-1$, then $\pi'$ and~$\rho'$ differ in a jump of~$j<n-1$, and~$\rho'$ and~$\sigma'$ differ in a jump of~$n-1$ by~\eqref{eq:JLn}.
Clearly, $j$ is at a boundary position in~$p^{n-j}(\rho)=p^{n-1-j}(\rho')$, and by induction and~(g) we have $s_{n-1,i}^{\rho'}=s_{n-1,i}^{\pi'}$ for $i\in\{1,\ldots,n-2\}\setminus \{j-1,j\}$, $s_{n-1,j-1}^{\rho'}=j-1$ and $s_{n-1,j}^{\rho'}=s_{n-1,j-1}^{\pi'}$.
Moreover, we have $s_{n-1,n-1}^{\rho'}=n-1$ by~\eqref{eq:spi11}.
Combining these observations proves~(g) in this last case.
\end{proof}

\begin{proof}[Proof of Theorem~\ref{thm:algo}]
We establish the following invariants about the permutation~$\pi$ visited in line~M2 of the algorithm:
\begin{enumerate}[label={(\Alph*)}, leftmargin=8mm, noitemsep, topsep=3pt plus 3pt]
\item For all $j=2,\ldots,n$, the direction of the next jump of the value~$j$ after the permutation~$\pi$ in~$J(L_n)$ is left if~$o_j=\dirl$ and right if~$o_j=\dirr$.
\item The values in the array $s=(s_1,\ldots,s_n)$ satisfy $s=s_n^\pi$ with $s_n^\pi$ as defined in~\eqref{eq:spi1} and~\eqref{eq:spi2}.
\end{enumerate}

By Lemma~\ref{lem:JLn-prop}~(a), the identity permutation~$\pi:=\ide_n$ is the first permutation in the sequence~$J(L_n)$.
Moreover, by the initialization of~$\pi$ in line~M1, $\pi=\ide_n$ is also the first permutation visited in line~M2.
Combining this with the above invariants, we obtain by induction on the length of~$J(L_n)$ that after visiting a permutation~$\pi\in L_n$, the next permutation visited by Algorithm~M is the permutation that succeeds~$\pi$ in~$J(L_n)$.
Indeed, by the instructions in line~M3 and~M4, the next permutation~$\rho$ visited by the algorithm is obtained from~$\pi$ by a jump of the value~$j:=s_n$ that is minimal w.r.t.~$L_n$, and the jump direction is left if $o_j=\dirl$ and right if $o_j=\dirr$.
Applying (B) and Lemma~\ref{lem:sn-prop}~(b), and~(A) and Lemma~\ref{lem:JLn-prop}~(c), we obtain that~$\rho$ is indeed the permutation that succeeds~$\pi$ in~$J(L_n)$.
Also, the algorithm terminates correctly after visiting the last permutation in the sequence~$J(L_n)$ by the condition in line~M3 and Lemma~\ref{lem:sn-prop}~(c).

We prove (A)+(B) by double induction on~$n$ and the number of iterations of Algorithm~M.
The induction basis $n=1$ is trivial.
For the induction step, let $n\geq 2$, and assume that the invariants hold for the zigzag language~$L_{n-1}=\{p(\pi)\mid \pi\in L_n\}$.
We first verify that (A)+(B) hold in line~M2 during the first iteration of the algorithm when~$\pi=\ide_n$.
By line~M1 we have $o_j=\dirl$ for $j=2,\ldots,n$, so (A) is satisfied by Lemma~\ref{lem:JLn-prop}~(b).
By line~M1 we also have $s=(s_1,\ldots,s_n)=(1,\ldots,n)$, which equals $s_n^\pi=s_n^{\ide_n}$ by Lemma~\ref{lem:sn-prop}~(a).

For the induction step, consider three consecutive permutations~$\pihat,\rhohat,\sigmahat$ in the sequence~$J(L_n)$, and suppose that (A)+(B) are satisfied when Algorithm~M visits~$\pi=\pihat$.
We need to verify that (A)+(B) still hold after one iteration through lines~M2---M5, after which the algorithm visits~$\pi=\rhohat$ by the instructions in lines~M3 and~M4, as argued before.

\textbf{Case~(i):}
We first consider the case that~$\pihat,\rhohat$ satisfy $p(\pihat)=p(\rhohat)=:\pihat'\in L_{n-1}$ and therefore both are contained in the subsequence $\ol{c}(\pihat')$ of~$J(L_n)$.
We only treat the case $\ol{c}(\pihat')=\lvec{c}(\pihat')$, as the other case $\ol{c}(\pihat')=\rvec{c}(\pihat')$ is symmetric.
In this case $\rhohat$ is obtained from~$\pihat$ by a left jump of the value~$n$.
In particular, the variable~$j$ has the value~$j=s_n=n$ and $o_n=\dirl$ in this iteration of the algorithm.

\textbf{Case~(ia):}
The value~$n$ is not at the first position in~$\rhohat$.
Then by~\eqref{eq:JLn} the permutation~$\sigmahat$ is obtained from~$\rhohat$ by another left jump of the value~$n$.
In line~M5, the value of~$s_n$ is set to~$n$, which was the previous value, but none of the conditions in line~M5 holds for $\pi=\rhohat$, so overall none of the arrays~$s$ and~$o$ is modified.
We conclude that~(A) holds after this iteration for $\pi=\rhohat$.
Moreover, (B) holds by Lemma~\ref{lem:sn-prop}~(c) and~\eqref{eq:spi12}.

\textbf{Case~(ib):}
The value~$n$ is at the first position in~$\rhohat$, i.e., we have $\rhohat=c_1(\pihat')=n\,\pihat'$.
Then by~\eqref{eq:JLn} we have $\sigmahat=c_1(\rhohat')=n\,\rhohat'$ where $\rhohat'\in L_{n-1}$ succeeds $\pihat'$ in the sequence~$J(L_{n-1})$.
In line~M5, the value of~$s_n$ is set to~$n$, which is the same as the previous value, but since the first if-condition is satisfied, $s_n$ is then overwritten by~$s_{n-1}$, and $s_{n-1}$ is set to~$n-1$.
Consequently, the new values are $s_{n-1}=n-1$ and $s_n=s_{n,n-1}^\pihat=s_{n-1,n-1}^{\pihat'}$ by induction and~(B) and~\eqref{eq:spi11}.
Moreover, the value of~$o_n$ is flipped to~$o_n=\dirr$.
Using Lemma~\ref{lem:JLn-prop}~(e), we conclude that~(A) holds after this iteration for $\pi=\rhohat$.
Applying Lemma~\ref{lem:sn-prop}~(d) and using that $s_{n,n}^\pi=s_{n-1,n-1}^{\pihat'}$ by~\eqref{eq:spi23}, we obtain that (B) holds as well.

\textbf{Case~(ii):}
It remains to consider the case that~$p(\pihat)\neq p(\rhohat)$, i.e., both permutations have $n$ at the first or last position.
By symmetry, it suffices to consider the case that $n$ is at the first position, i.e., $\pihat=c_1(\pihat')=n\,\pihat'$ and $\rhohat=c_1(\rhohat')=n\,\rhohat'$, where~$\pihat'$ and~$\rhohat'$ are consecutive permutations in~$J(L_{n-1})$.
They differ in a jump of the value~$j:=s_n<n$ by Lemma~\ref{lem:sn-prop}~(b) and~(B).
Then by~\eqref{eq:JLn}, the permutation $\sigmahat$ is obtained from~$\rhohat$ by a right jump of~$n$.
We proceed to show that~(A) and~(B) hold for $\pi=\rhohat$, and for this we distinguish subcases.

\textbf{Case~(iia):}
The value~$j$ is not at a boundary position in~$p^{n-j}(\rhohat)$.
Then by Lemma~\ref{lem:JLn-prop}~(d) in~$\rhohat$ the value~$j$ is surrounded by the same values as in~$p^{n-j}(\rhohat)$, both smaller than~$j$.
In line~M5, the value of~$s_n$ is set to~$n$, and no other entries of~$s$ and~$o$ are modified.
Using Lemma~\ref{lem:JLn-prop}~(e), we conclude that~(A) holds after this iteration for~$\pi=\rhohat$.
Moreover, (B) holds by Lemma~\ref{lem:sn-prop}~(e), also using that $s_{n,n}^\rhohat=n$ by~\eqref{eq:spi12}.

\textbf{Case~(iib):}
The value~$j$ is at a boundary position in~$p^{n-j}(\rhohat)$, and the value~$k$ next to it is smaller than~$j$.
Then by Lemma~\ref{lem:JLn-prop}~(d) in~$\rhohat$ the value~$k$ is also next to~$j$, and either $j$ is at a boundary position (the right boundary, as $n$ is at the first position in~$\rhohat$) or the other value next to in the direction~$o_j$ of the jump is bigger than~$j$.
In line~M5, the value of~$s_n$ is set to~$n$, the value of~$s_j$ is set to~$s_{j-1}$, and $s_{j-1}$ is set to~$j-1$.
Moreover, the value of~$o_j$ is flipped.
Using Lemma~\ref{lem:JLn-prop}~(e), we conclude that~(A) holds after this iteration for~$\pi=\rhohat$.
Moreover, (B) holds by Lemma~\ref{lem:sn-prop}~(f), also using that $s_{n,n}^\rhohat=n$ by~\eqref{eq:spi12}.
This completes the proof of the theorem.
\end{proof}

\subsection{Proof of Theorem~\ref{thm:algo-rect}}

\begin{proof}[Proof of Theorem~\ref{thm:algo-rect}]
In the proof of Theorem~\ref{thm:jump-rect} we showed that the ordering of rectangulations generated by Algorithm~\Jrect{} is given by~\eqref{eq:JrectLn} for some zigzag language~$L_n$ of 2-clumped permutations.
The theorem hence follows by applying Theorem~\ref{thm:algo}.
\end{proof}

\section{S-equivalence of rectangulations}
\label{sec:Sequiv}

Recall that R-equivalence is the equivalence relation on~$\cR_n$ obtained from wall slides, i.e., any two generic rectangulations that differ in a sequence of wall slides are equivalent.
It is well known that every equivalence class contains exactly one diagonal rectangulation, i.e., $\cD_n$ is a set of representatives for R-equivalence (see e.g.~\cite{MR3878132}).

\begin{figure}[b!]
\includegraphics{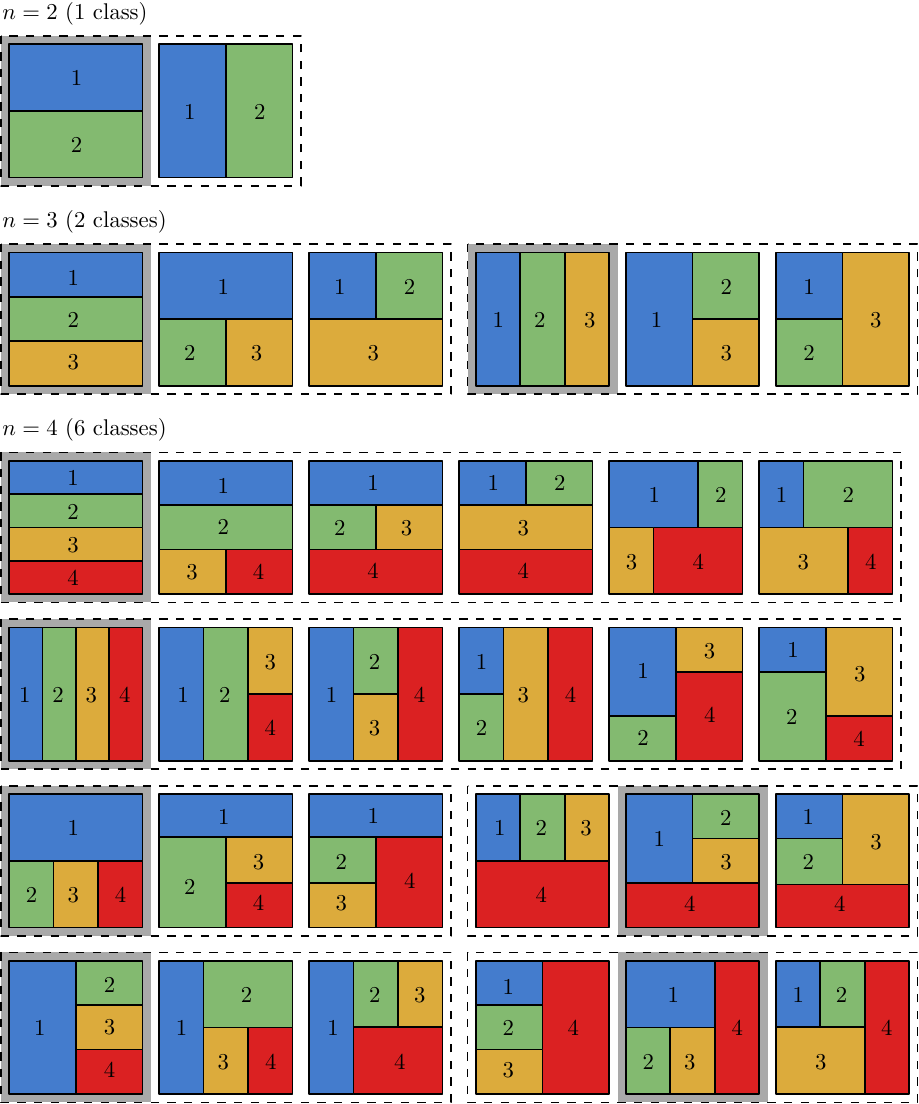}
\caption{Equivalence classes of generic rectangulations under S-equivalence for $n=2,3,4$.
Block-aligned rectangulations as equivalence class representatives are highlighted.}
\label{fig:sequiv}
\end{figure}

We aim to do something analogous for S-equivalence, and to pick a suitable set of representatives for our generation algorithm.
Recall that S-equivalence is the equivalence relation on~$\cR_n$ obtained from wall slides and simple flips, i.e., any two generic rectangulations that differ in a sequence of wall slides or simple flips are equivalent.
Figure~\ref{fig:sequiv} shows the equivalence classes of generic rectangulations under S-equivalence for $n=2,3,4$.
Unfortunately, one can check that there is no choice of representatives~$\cP_2\seq\cR_2$, $\cP_3\seq\cR_3$, $\cP_4\seq\cR_4$ for those equivalence classes (i.e., $|\cP_2|=1$, $|\cP_3|=2$, $|\cP_3|=6$) that is consistent with the operations of rectangle deletion and insertion, i.e., such that $\cP_2=\{p(R)\mid R\in\cP_3\}$ and $\cP_3=\{p(R)\mid R\in\cP_4\}$.
Consequently, for S-equivalence there is no set of unique representatives, one for each equivalence class, that would form a zigzag set, so our generation algorithms cannot be applied directly.
However, we will show how to choose representatives for each equivalence class (highlighted in Figure~\ref{fig:sequiv}), such that those representatives \emph{and} the rectangulations obtained from them by a simple flip of the rectangle~$r_n$ admit a generation tree approach with our algorithms.

\subsection{Representatives for S-equivalence}

By definition, S-equivalence is a coarsening of R-equivalence, and we will therefore choose a subset of diagonal rectangulations as representatives.

We start with some definitions; see Figure~\ref{fig:aligned}.
A rectangulation is \emph{horizontally aligned}, or \emph{H-aligned} for short, if all of its walls are horizontal.
Moreover, a rectangulation is \emph{almost horizontally aligned}, or \emph{AH-aligned} for short, if all of its walls except one at the bottom are horizontal.
Equivalently, it is obtained by gluing copies of~$\square$ on top of~$\boxbar$.

Similarly, a rectangulation is \emph{vertically aligned}, or \emph{V-aligned} for short, if all of its walls are vertical.
Moreover, a rectangulation is \emph{almost vertically aligned}, or \emph{AV-aligned} for short, if all of its walls except one at the right are vertical.
Equivalently, it is obtained by gluing copies of~$\square$ on the left of~$\boxminus$.

\begin{figure}[h!]
\includegraphics[page=1]{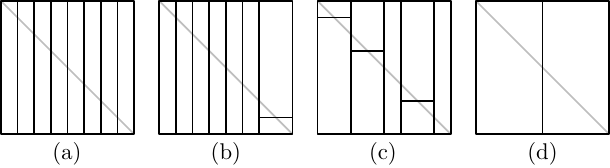}
\caption{Illustration of aligned rectangulations.
Rectangulation~(a) is V-aligned, (b) is AV-aligned, (c) is V-alignable but neither V-aligned nor AV-aligned, (d) is V-aligned and AH-aligned.}
\label{fig:aligned}
\end{figure}

The rectangulation~$\square$ is H-aligned and V-aligned.
The rectangulation~$\boxminus$ is H-aligned and AV-aligned, and the rectangulation~$\boxbar$ is V-aligned and AH-aligned.

A rectangulation is \emph{H- or V-alignable}, if we can apply a sequence of simple flips to make it \emph{H- or V-aligned}, respectively.
Clearly, a rectangulation is H-alignable if it is obtained by vertically gluing together copies of $\square$ and $\boxbar$, and it is V-alignable if it is obtained by horizontally gluing together copies of $\square$ and~$\boxminus$.

A \emph{block} in a rectangulation is a subset of rectangles whose union is a rectangle.
The \emph{size} of a block is the number of rectangles of the block.

\begin{lemma}
\label{lem:blocks}
Every diagonal rectangulation can be partitioned uniquely into maximal alignable blocks.
\end{lemma}

\begin{proof}
Suppose for the sake of contradiction that for some rectangulation~$R\in\cD_n$ there were two distinct block partitions~$P,P'$ of~$R$.
Consider a block~$B$ in~$P$ that is not a block in~$P'$.
Consider one of the rectangles in~$B$, and consider the block~$B'$ of~$P'$ containing this rectangle.
As $R$ does not have any points where 4 rectangles meet and $B'\neq B$, the block~$B'$ must be a proper subset or superset of~$B$, contradicting the maximal choice of the blocks.
\end{proof}

Lemma~\ref{lem:blocks} holds more generally for generic rectangulations and for maximal blocks with any additional property (such as alignable), but this is not needed here.

\begin{lemma}
\label{lem:blocks-flips}
For any diagonal rectangulation, the partition into maximal alignable blocks is invariant under simple flips.
\end{lemma}

\begin{proof}
Consider a wall that can be simple-flipped, and observe that the two rectangles to both sides of the wall must belong to the same alignable block due to the maximal choice of the blocks.
\end{proof}

From now on, whenever we refer to a block in a rectangulation, we mean a maximal alignable block.
A block is a \emph{base block}, if it contains the bottom boundary of the rectangulation.

Based on the partition of a diagonal rectangulation~$R\in\cD_n$ into blocks, which is unique by Lemma~\ref{lem:blocks}, we introduce the following definitions; see Figure~\ref{fig:blocks}.
We refer to each block of~$R$ as an \emph{H-block} or \emph{V-block}, if it is H-alignable or V-alignable, respectively.
We consider an H-block~$B$ of size at least~2 with rectangle~$r_i$ at the bottom-right.
If $i=n$ or if rectangle~$r_{i+1}$ of~$R$ is below~$r_i$ we say that $B$ is \emph{free}, whereas if $r_{i+1}$ is right of~$r_i$ we say that $B$ is \emph{locked}.
Similarly, we consider a V-block~$B$ of size at least~2 with rectangle~$r_i$ at the bottom-right.
If $i=n$ or if rectangle~$r_{i+1}$ of~$R$ is right of~$r_i$ we say that $B$ is \emph{free}, whereas if $r_{i+1}$ is below~$r_i$ we say that $B$ is \emph{locked}.

\begin{figure}[b!]
\vspace{2mm}
\includegraphics[page=2]{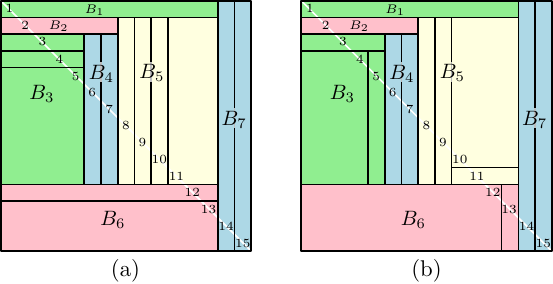}
\caption{Illustration of blocks and block-aligned rectangulations.
Blocks are highlighted by the same shading.
H-blocks are $B_1$, $B_2$, $B_3$, $B_4$, $B_6$ and~$B_7$, with $B_7$ free and $B_3$, $B_4$ and $B_6$ locked.
V-blocks are $B_1$, $B_2$, $B_4$, $B_5$, $B_6$ and~$B_7$, with $B_4$, $B_6$ and $B_7$ free and $B_5$ locked.
Rectangulation~(a) is not block-aligned, whereas (b) is block-aligned.
}
\label{fig:blocks}
\end{figure}

We say that $R\in\cD_n$ is \emph{block-aligned} if for every block~$B$ of size at least~2 in~$R$ the following conditions hold: if $B$ is a free H-block then~$B$ is H-aligned, if $B$ is a free V-block then $B$ is V-aligned, if $B$ is a locked H-block then $B$ is AH-aligned, and if $B$ is a locked V-block then $B$ is AV-aligned.
A special rule applies if the rectangle~$r_n$ is contained in a block of size~2 (which is free and both H-alignable and V-alignable), and then we require this block to be V-aligned, unless it is a base block, in which case it must be H-aligned.
Note that a block~$B$ of size exactly~2 that does not contain~$r_n$ is both an H-block and a V-block, however, if $B$ is a locked/free H-block then $B$ is a free/locked V-block, respectively, so this definition is consistent (as AH-aligned equals V-aligned and H-aligned equals AV-aligned for a block of size~2).

We write $\cB_n\seq\cD_n$ for the set of diagonal rectangulations that are block-aligned.
We partition this set into~$\cB_n^\square$ and~$\cB_n^\boxplus$, respectively, according to whether the rectangle~$r_n$ is contained in a block of size~1 or at least~2, respectively.
Note that if $R\in\cB_n^\square$, then the wall between~$r_n$ and~$r_{n-1}$ does not admit a simple flip, whereas if $R\in\cB_n^\boxplus$, then this wall admits a simple flip.
For any $R\in\cB_n^\boxplus$ we write $s(R)\in\cB_n^\boxplus$ for the rectangulation obtained from~$r_n$ by a simple flip of this wall.
The set $\cB_n^\boxplus$ is partitioned into $\cB_n^\boxminus$ and~$\cB_n^\boxbar$ according to whether this wall is horizontal or vertical, respectively.

As a consequence of Lemma~\ref{lem:blocks-flips}, every equivalence class of generic rectangulations under S-equivalence contains exactly one block-aligned diagonal rectangulation; see Figure~\ref{fig:sequiv}.
Consequently, we will use the block-aligned rectangulations $\cB_n\seq\cD_n$ as representatives for S-equivalence.

\subsection{Insertion in block-aligned rectangulations}

The next two lemmas describe how to construct block-aligned rectangulations by rectangle insertion; see Figure~\ref{fig:tree-sequiv}.

For any diagonal rectangulation~$R\in\cD_n$, we let $I_v(R)$ denote the subsequence of~$I(R)$ of the first insertion point of each vertical group.
Similarly, we let $I_h(R)$ denote the subsequence of~$I(R)$ of the last insertion point of each horizontal group.

\begin{lemma}
\label{lem:block-insert1}
Let $P\in\cB_{n-1}^\square$, $I_v(P)=:(q_{i_1},\ldots,q_{i_\lambda})$ and $I_h(P)=:(q_{j_1},\ldots,q_{j_\mu})$.
Then we have the following:
\begin{itemize}[leftmargin=5mm, noitemsep, topsep=3pt plus 3pt]
\item For any $1\leq k<\lambda$ we have $c_{i_k}(P)\in\cB_n^\square$, and every $R\in\cB_n^\square$ for which the top-left vertex of~$r_n$ has type~$\rightT$ and $r_{n-1}$ forms its own block is obtained by insertion from some $P\in\cB_{n-1}^\square$ in this way.
\item For any $1<k\leq \mu$ we have $c_{j_k}(P)\in\cB_n^\square$, and every $R\in\cB_n^\square$ for which the top-left vertex of~$r_n$ has type~$\bottomT$ and $r_{n-1}$ forms its own block is obtained by insertion from some $P\in\cB_{n-1}^\square$ in this way.
\item If $\lambda>1$ we have $c_{j_1}(P)\in\cB_n^\boxbar$, and every $R\in\cB_n^\boxbar$ for which $r_{n-1}$ and~$r_n$ form a V-aligned block of size~2 is obtained by insertion from some~$P\in\cB_{n-1}^\square$ in this way.
\item If $\lambda=1$ we have $c_{i_1}(P)\in\cB_n^\boxminus$, and every $R\in\cB_n^\boxminus$ for which $r_{n-1}$ and $r_n$ form an H-aligned base block of size~2 is obtained by insertion from some~$P\in\cB_{n-1}^\square$ in this way.
\end{itemize}
\end{lemma}

\begin{proof}
The first and second part of the lemma are symmetric, so it suffices to prove the first one.
For this we analyze how the blocks of~$R:=c_{i_k}(P)$ differ from the blocks of~$P$, and prove that they are all aligned as required.

The rectangle~$r_{n-1}$ forms a block of size~1 in~$P$, and as $k<\lambda$ this is also true in~$R$.
Similarly, as $k<\lambda$ the rectangle~$r_n$ forms a block of size~1 in~$R$.
Consequently, we only need to verify whether blocks of~$R$ not containing~$r_{n-1}$ or~$r_n$ in~$P$ are aligned as required.
If $k=1$, then the blocks of~$R$ are the same as those of~$P$, plus the block containing~$r_n$, so we are done; see Figure~\ref{fig:block-insert}~(a).
If $k>1$, we let $r_a$ and~$r_b$ be the rectangles in~$P$ to the left and right of the edge that contains the insertion point~$q_{i_k}$.
If $r_a$ and~$r_b$ belong to two distinct blocks in~$P$, then the blocks of~$R$ are the same as those of~$P$, plus the block containing~$r_n$, so we are done; see Figure~\ref{fig:block-insert}~(b).
On the other hand, if $r_a$ and~$r_b$ belong to the same block~$B$ in~$P$, then it must be a free V-block that is V-aligned or a locked H-block that AH-aligned, and we have $b=a+1$.
If $B$ is a free V-block in~$P$, then in~$R$ this block is split into two free V-blocks~$B'$ and~$B''$, one containing~$r_a$ and the other one containing~$r_b=r_{a+1}$, and both~$B'$ and~$B''$ are V-aligned; see Figure~\ref{fig:block-insert}~(c).
If $B$ is a locked H-block in~$P$, then in~$R$ this block is split into the H-block $B\setminus\{r_a,r_{a+1}\}$, and two blocks of size~1 containing~$r_a$ or~$r_{a+1}$, respectively; see Figure~\ref{fig:block-insert}~(d).
Moreover, if $|B|=3$ then $|B\setminus\{r_a,r_{a+1}\}|=1$, and otherwise $B\setminus\{r_a,r_{a+1}\}$ is a free H-block that is H-aligned in~$R$.
In all cases we obtain $R\in\cB_n^\square$, as claimed.

We continue to prove the third part of the lemma about the rectangulation $R:=c_{j_1}(P)$.
The rectangle~$r_{n-1}$ forms a block of size~1 in~$P$, and together with~$r_n$ it forms a block of size~2 in~$R$; see Figure~\ref{fig:block-insert}~(e).
This block is V-aligned in~$R$, so we have $R\in\cB_n^\boxbar$, as claimed.

It remains to prove the fourth part of the lemma about the rectangulation $R:=c_{i_1}(P)$.
The rectangle~$r_{n-1}$ forms a block of size~1 in~$P$, and together with~$r_n$ it forms a block of size~2 in~$R$; see Figure~\ref{fig:block-insert}~(f).
This block is a base block and H-aligned in~$R$, so we have $R\in\cB_n^\boxminus$, as claimed.
\end{proof}

\begin{figure}
\includegraphics[page=3]{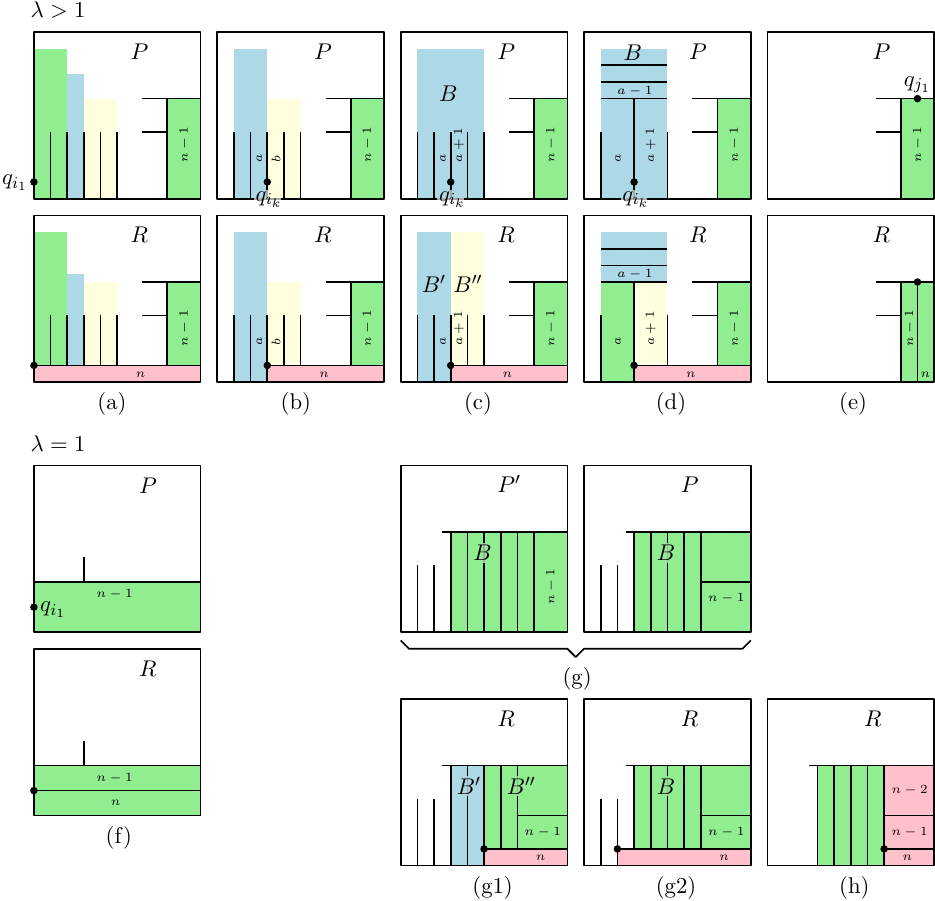}
\caption{Illustration of the proofs of Lemmas~\ref{lem:block-insert1} and~\ref{lem:block-insert2}.}
\label{fig:block-insert}
\end{figure}

\begin{lemma}
\label{lem:block-insert2}
Let $P\in\cB_{n-1}^\boxminus$ and $P':=s(P)$, or let $P'\in\cB_{n-1}^\boxbar$ and $P:=s(P')$, and define $I_v(P)=:(q_{i_1},\ldots,q_{i_\lambda})$ and $I_h(P')=:(q_{j_1},\ldots,q_{j_\mu})$.
Then we have the following:
\begin{itemize}[leftmargin=5mm, noitemsep, topsep=3pt plus 3pt]
\item For any $1\leq k<\lambda$ we have $c_{i_k}(P)\in\cB_n^\square$, and every $R\in\cB_n^\square$ for which the top-left vertex of~$r_n$ has type~$\rightT$ and $r_{n-1}$ is contained in a block of size at least~2 is obtained by insertion from some $P\in\cB_{n-1}^\boxminus$ in this way.
\item For any $1<k\leq \mu$ we have $c_{j_k}(P')\in\cB_n^\square$, and every $R\in\cB_n^\square$ for which the top-left vertex of~$r_n$ has type~$\bottomT$ and $r_{n-1}$ is contained in a block of size at least~2 is obtained by insertion from some $P'\in\cB_{n-1}^\boxbar$ in this way.
\item
We have $c_{i_\lambda}(P)\in\cB_n^\boxminus$, and every $R\in\cB_n^\boxminus$ for which $r_{n-1}$ and~$r_n$ are contained in an H-aligned block of size at least~3 is obtained by insertion from some $P\in\cB_{n-1}^\boxminus$ in this way.
\item
We have $c_{j_1}(P')\in\cB_n^\boxbar$, and every $R\in\cB_n^\boxbar$ for which $r_{n-1}$ and~$r_n$ are contained in a V-aligned block of size at least~3 is obtained by insertion from some $P'\in\cB_{n-1}^\boxbar$ in this way.
\end{itemize}
\end{lemma}

\begin{proof}
The proof for the first part in the case $P\in\cB_{n-1}^\boxminus$ and for the second part in the case $P'\in\cB_{n-1}^\boxbar$ is analogous to the proof of Lemma~\ref{lem:block-insert1}.
Therefore, by symmetry, to complete the proof of the first two parts, it suffices to argue about the case $P'\in\cB_{n-1}^\boxbar$, $P:=s(P')$ and the rectangulation $R:=c_{i_k}(P)$ for $1\leq k<\lambda$; see Figure~\ref{fig:block-insert}~(g).
The V-block~$B$ in~$P'$ containing~$r_{n-1}$, which is free and V-aligned in~$P'$, is AV-aligned and free in~$P$.
Consequently, in~$R$ the block~$B$ is either split into two blocks, a free V-block~$B'$ that is V-aligned to the left of a locked V-block~$B''$ (Figure~\ref{fig:block-insert}~(g1)) that is AV-aligned, or $B$ remains a single locked V-block that is AV-aligned in~$R$ (Figure~\ref{fig:block-insert}~(g2)), where the locking is due to the insertion of~$r_n$.
The remaining blocks of~$P'$ are treated as in the proof of Lemma~\ref{lem:block-insert1}.
In all cases we obtain that $R\in\cB_n^\square$, as claimed.

The third and fourth part of the lemma are symmetric, so it suffices to prove the third one about the rectangulation $R:=c_{i_\lambda}(P)$.
In this case $\{r_{n-2},r_{n-1},r_n\}$ is an H-block that is free and H-aligned in~$R$, and either $|B\setminus\{r_{n-2},r_{n-1}\}|=1$ or $B\setminus\{r_{n-2},r_{n-1}\}$ is a V-block that is free and V-aligned in~$R$; see Figure~\ref{fig:block-insert}~(h).
It follows that $R\in\cB_n^\boxminus$, as claimed.
\end{proof}

\begin{figure}
\includegraphics[scale=0.7]{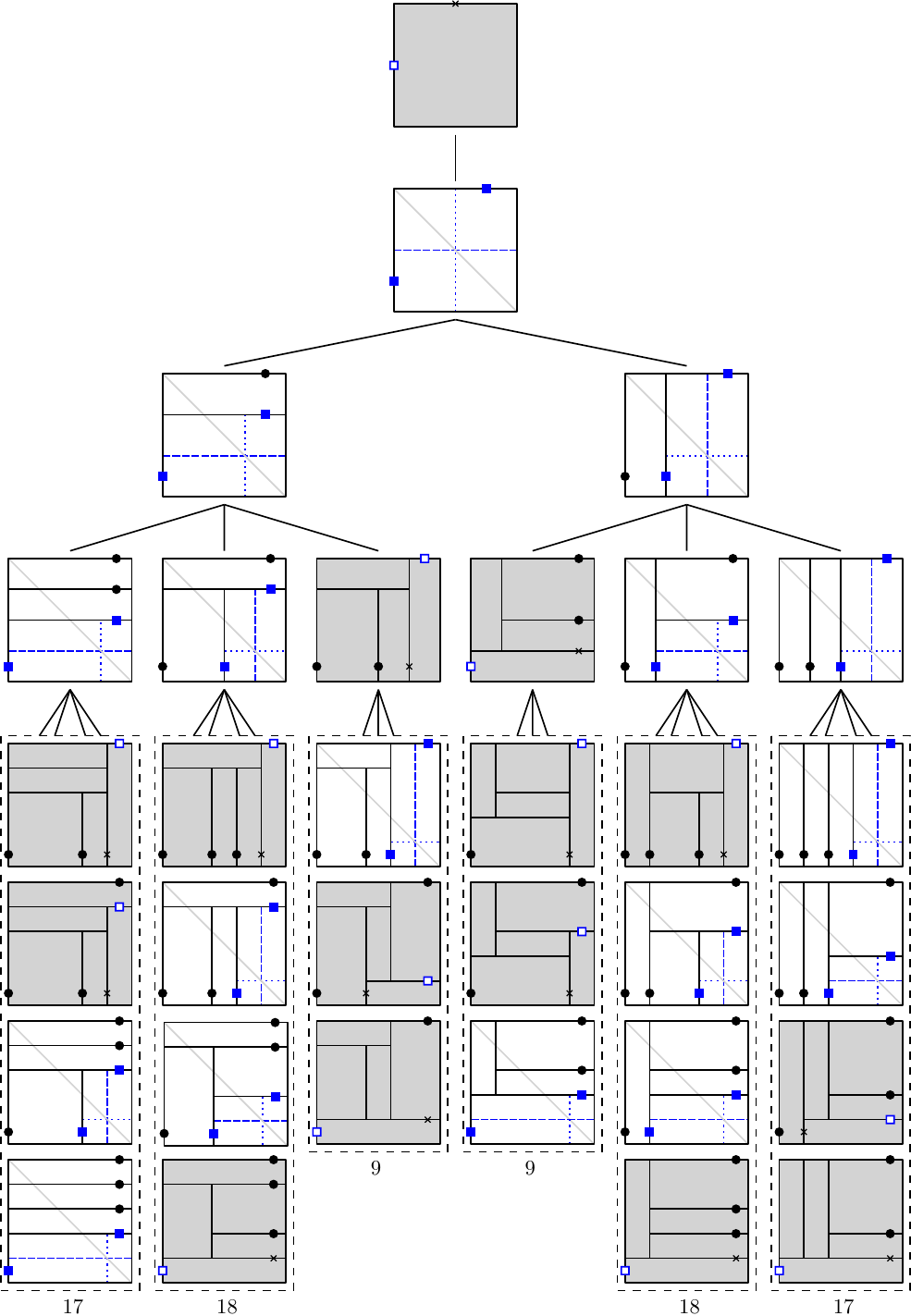}
\caption{Tree of block-aligned rectangulations.
Rectangulations from~$\cB_n^\square$ are drawn gray, and those from~$\cB_n^\boxplus$ are drawn white.
For any $R\in\cB_n^\boxplus$, the wall between rectangles~$r_n$ and~$r_{n-1}$ is drawn dashed, whereas the corresponding wall in $s(R)$ is drawn dotted.
Each insertion point marked by a disk corresponds to one child of the current node as in the first two parts of Lemmas~\ref{lem:block-insert1} and~\ref{lem:block-insert2}.
Each insertion points marked by an empty square corresponds to one child as in the third or fourth part of Lemma~\ref{lem:block-insert1}, whereas crossed insertion points are not used.
Each insertion point marked by a solid square corresponds to one child as in the third or fourth part of Lemma~\ref{lem:block-insert2}.
The numbers at the bottom indicate the number of nodes in the next level of the tree, of which there are $2(17+18+9)=88$ overall (i.e., we have $|\cB_6|=88$).
}
\label{fig:tree-sequiv}
\end{figure}

\subsection{Tree of block-aligned rectangulations}

By Lemmas~\ref{lem:block-insert1} and~\ref{lem:block-insert2}, all block-aligned rectangulations~$\cB_n$ can be obtained by suitable rectangle insertions into all block-aligned rectangulations~$\cB_{n-1}$ and $s(\cB_{n-1}^\boxplus)$.
We consider the subtree of the tree of rectangulations discussed in Section~\ref{sec:tree} induced by the rectangulations~$\cB_n$ and $s(\cB_n^\boxplus)$ for all~$n\geq 1$.
By gluing together pairs of nodes~$(R,s(R))$ for all $R\in\cB_n^\boxplus$, we obtain the tree shown in Figure~\ref{fig:tree-sequiv}.

For any $P\in\cB_{n-1}^\square$, using the notation from Lemma~\ref{lem:block-insert1} we define
\begin{subequations}
\label{eq:JSequiv}
\begin{equation}
\label{eq:JSequiv1}
c(P):=\begin{cases}
\big(c_{i_1}(P),\ldots,c_{i_{\lambda-1}}(P),c_{j_1}(P),c_{j_2}(P),\ldots,c_{j_\mu}(P)\big) & \text{ if } \lambda>1, \\
\big(c_{i_1}(P),c_{j_2}(P),\ldots,c_{j_\mu}(P)\big) & \text{ if } \lambda=1. \\
\end{cases}
\end{equation}
For any $P\in\cB_{n-1}^\boxminus\cup s(\cB_{n-1}^\boxbar)$ and~$P':=s(P)$, using the notation from Lemma~\ref{lem:block-insert2} we define
\begin{equation}
\label{eq:JSequiv2}
c(P):=\big(c_{i_1}(P),\ldots,c_{i_{\lambda-1}}(P),c_{i_\lambda}(P),c_{j_1}(P'),c_{j_2}(P'),\ldots,c_{j_\mu}(P')\big).
\end{equation}
\end{subequations}
These sequences define an ordering among the children of each node in the aforementioned (unordered) tree of block-aligned rectangulations.

Note that any two consecutive rectangulations in the sequence~\eqref{eq:JSequiv1} differ in a T-flip, except $c_{i_{\lambda-1}}(P)$ and~$c_{j_1}(P)$, and $c_{i_1}(P)$ and~$c_{j_2}(P)$, which differ in a T-flip plus a simple flip.
Similarly, any two consecutive rectangulations in the sequence~\eqref{eq:JSequiv2} differ in a T-flip, except $c_{i_\lambda}(P)$ and $c_{j_1}(P')$, which differ in a simultaneous flip of the two walls between $r_n$, $r_{n-1}$ and $r_{n-2}$.
We refer to this operation as a \emph{D-flip} (D like `double').

\subsection{Next oracle for block-aligned rectangulations}

Using~\eqref{eq:JSequiv}, we may modify the minimal jump oracle~$\nextjump_{\cD_n}$ for diagonal rectangulations described in Section~\ref{sec:oracle-diag} for the generation of block-aligned rectangulations within Algorithm~\Mrect{} as follows.
Some Gray code orderings produced by this algorithm are shown in the appendix.

\begin{algo}{$\nextjump_{\cB_n}(R,j,d)$}{Next oracle for block-aligned rectangulations}
\begin{enumerate}[label={\bfseries N\arabic*.}, leftmargin=8mm, noitemsep, topsep=3pt plus 3pt]
\item{}[Prepare] Set $a\gets r_j.\rnw$ and call $\unlock(j,d)$.
If $d=\dirl$ and $v_a.\vtype=\bottomT$, set $\alpha\gets v_a.\vsouth$ and $b\gets e_\alpha.\etail$.
If $v_b.\vtype=\leftT$ goto~N2, otherwise we have $v_b.\vtype=\topT$ and goto~N3.
If $d=\dirr$ and $v_a.\vtype=\bottomT$, set $\alpha\gets v_a.\veast$ and goto~N4.
If $d=\dirr$ and $v_a.\vtype=\rightT$, set $\alpha\gets v_a.\veast$ and $b\gets e_\alpha.\ehead$.
If $v_b.\vtype=\topT$ goto~N5, otherwise we have $v_b.\vtype=\leftT$ and goto~N6.
If $d=\dirl$ and $v_a.\vtype=\rightT$, set $\alpha\gets v_a.\vsouth$ and goto~N7.
\item{}[Horizontal left jump (T/TS)] Set $\gamma\gets v_b.\vwest$ and call $\Tjumph(R,j,\dirl,\gamma)$.
Then set $a\gets r_j.\rnw$, $\alpha\gets v_a.\vsouth$, $b\gets e_\alpha.\etail$, $c\gets r_{j-1}.\rsw$, $\gamma\gets v_c.\vnorth$, $c'\gets r_j.\rse$ and if $v_b.\vtype=\topT$ and [$v_{c'}.\vtype=\leftT$ or [$j=n$ and $e_\gamma.\eleft=0$]] call $\Sjump(R,j,\dirl,\gamma)$.
Call $\lock(R,j,\dirr)$ and return.
\item{}[Horizontal left jump (ST/D)] Set $c\!\gets\! r_{j-1}.\rsw$ and $\gamma\!\gets\! v_c.\vnorth$ and call $\Sjump(R,j,\dirl,\gamma)$.
Then set $\gamma\gets v_c.\vnorth$, $k\gets e_\gamma.\eleft$, $c'\gets r_k.\rsw$ and $\gamma'\gets v_{c'}.\vnorth$ and call $\Tjumpv(R,j,\dirl,\gamma')$.
Set $c\gets r_{j-1}.\rsw$, $\gamma\gets v_c.\vnorth$ and $a\gets e_\gamma.\ehead$.
If $v_a.\vtype=\bottomT$ we have $k=j-2$, set $c'\gets r_{j-2}.\rsw$, $\gamma'\gets v_{c'}.\vnorth$ and call $\Sjump(R,j-1,\dirl,\gamma')$.
Call $\lock(R,j-1,\dirr)$ and return.
\item{}[Horizontal right jump (T/TS)] Set $k\gets e_\alpha.\eleft$, $b\gets r_k.\rne$ and $\gamma\gets v_b.\vwest$ and call $\Tjumph(R,j,\dirr,\gamma)$.
Then set $a\gets r_j.\rnw$, $\alpha\gets v_a.\vsouth$, $b\gets e_\alpha.\etail$, $\beta\gets v_b.\vsouth$, $\gamma\gets v_b.\vwest$, $c\gets e_\beta.\etail$ and $c'\gets e_\gamma.\etail$, and if $v_c.\vtype=\topT$ and $v_{c'}.\vtype=\rightT$ we have $k=j-2$, set $\gamma'\gets v_a.\vwest$ and call $\Sjump(j-1,R,\dirr,\gamma')$.
Call $\lock(R,j,\diru)$ and return.
\item{}[Vertical right jump (T/TS)] Set $\gamma\gets v_b.\vnorth$ and call $\Tjumpv(R,j,\dirr,\gamma)$.
Then set $a\gets r_j.\rnw$, $\alpha\gets v_a.\veast$, $b\gets e_\alpha.\ehead$, $c\gets r_{j-1}.\rne$, $\gamma\gets v_c.\vwest$, $c'\gets r_j.\rse$, $e\gets r_{j-1}.\rnw$ and if $v_b.\vtype=\leftT$ and [$v_{c'}.\vtype=\topT$ or [$j=n$ and not [$v_e.\vtype=\rightT$ and $e_\gamma.\etail=e$]] call $\Sjump(R,j,\dirr,\gamma)$.
Call $\lock(R,j,\diru)$ and return.
\item{}[Vertical right jump (ST/D)] Set $c\gets r_{j-1}.\rne$ and $\gamma\gets v_c.\vwest$ and call $\Sjump(R,j,\dirr,\gamma)$.
Then set $\gamma\gets v_c.\vwest$, $k\gets e_\gamma.\eleft$, $c'\gets r_k.\rne$ and $\gamma'\gets v_{c'}.\vwest$ and call $\Tjumph(R,j,\dirr,\gamma')$.
Set $c\gets r_{j-1}.\rne$, $\gamma\gets v_c.\vwest$ and $a\gets e_\gamma.\etail$.
If $v_a.\vtype=\rightT$ we have $k=j-2$, set $c'\gets r_{j-2}.\rne$, $\gamma'\gets v_{c'}.\vwest$ and call $\Sjump(R,j-1,\dirr,\gamma')$.
Call $\lock(R,j-1,\diru)$ and return.
\item{}[Vertical left jump (T/TS)] Set $k\gets e_\alpha.\eleft$, $b\gets r_k.\rsw$ and $\gamma\gets v_b.\vnorth$ and call $\Tjumpv(R,j,\dirl,\gamma)$.
Then set $a\gets r_j.\rnw$, $\alpha\gets v_a.\veast$, $b\gets e_\alpha.\ehead$, $\beta\gets v_b.\veast$, $\gamma\gets v_b.\vnorth$, $c\gets e_\beta.\ehead$ and $c'\gets e_\gamma.\ehead$, and if $v_c.\vtype=\leftT$ and $v_{c'}.\vtype=\bottomT$ we have $k=j-2$, set $\gamma'\gets v_a.\vnorth$ and call $\Sjump(j-1,R,\dirl,\gamma')$.
Call $\lock(R,j,\dirr)$ and return.
\end{enumerate}
\end{algo}
Lines N2--N4 are symmetric to lines N5--N7, so we only consider N2--N4; see the illustrations in Figure~\ref{fig:Sequiv-next}.
Lines~N2 and~N4 perform a T-flip, possibly followed by a simple flip.
Line~N3 performs a simple flip followed by a T-flip, possibly followed by a simple flip, and in this case the combination of three flips, simple flip plus T-flip plus simple flip, yields a D-flip overall.
The function $\lock(R,j,\dir)$, $\dir\in\{\dirr,\diru\}$, called at the end of each of the lines~N2--N7 checks whether rectangle~$r_j$ participates in an H-aligned block (if $\dir=\dirr$) or V-aligned block (if $\dir=\diru$) that is locked and must be transformed to an AH-aligned block or AV-aligned block by a simple flip.
The function $\unlock(R,j,d)$, $d\in\{\dirl,\dirr\}$, called at the beginning in line~N1 does the converse, namely checking whether~$r_j$ participates in an AH-aligned block or AV-aligned block that must be made H-aligned or V-aligned, respectively, before performing a jump of rectangle~$r_j$ in direction~$d$.
The implementation of these functions is shown below for the cases $\dir=\dirr$ and $d=\dirr$.
The other variants $\dir=\diru$ and $d=\dirl$ are omitted for simplicity.

\begin{figure}
\makebox[0cm]{ 
\includegraphics[page=2]{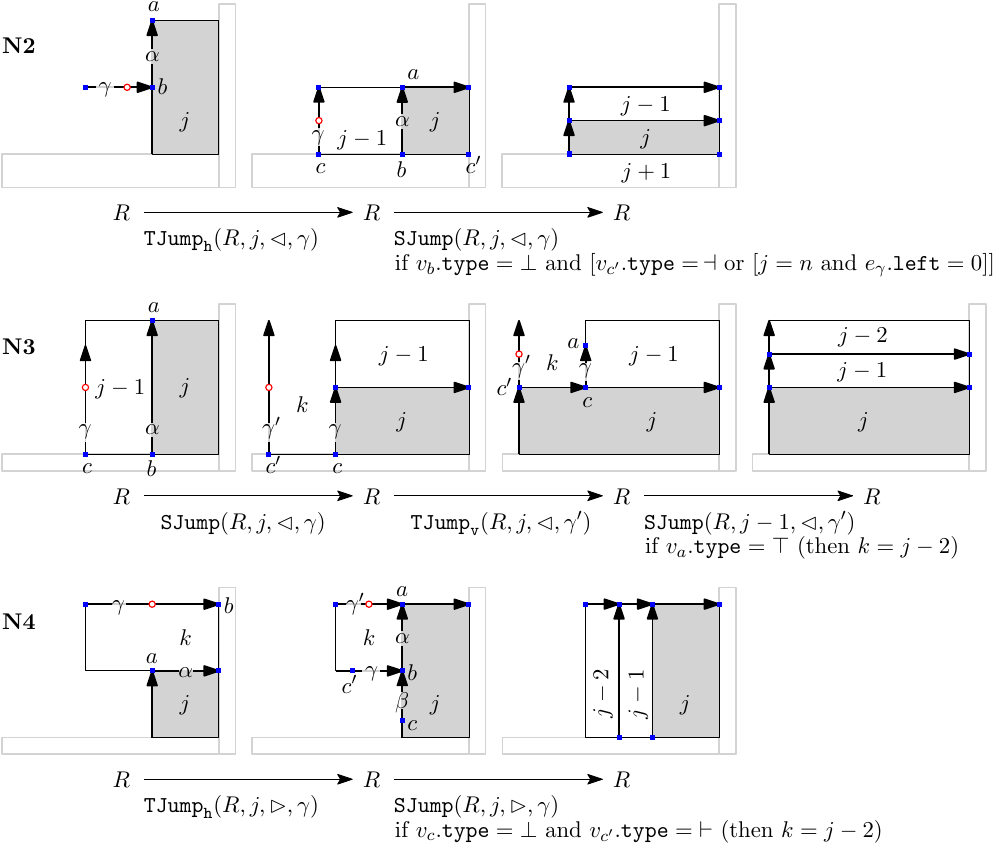}
}
\caption{Flip operations in lines~N2--N4 of the oracle $\nextjump_{\cB_n}$.
}
\label{fig:Sequiv-next}
\end{figure}

\begin{algo}{$\lock(R,j,\dirr)$}{Lock block if necessary}
\begin{enumerate}[label={\bfseries L\arabic*.}, leftmargin=8mm, noitemsep, topsep=3pt plus 3pt]
\item{}[Prepare] Set $a\gets r_j.\rne$, $b\gets r_j.\rsw$, $c\gets r_j.\rse$, $\alpha\gets v_a.\vwest$, $\beta\gets v_b.\veast$ and return if $v_b.\vtype\neq \rightT$ or $v_c.\vtype\neq \leftT$ or $e_\beta.\ehead\neq c$.
\item{}[Lock if necessary] Set $d\gets r_{j+1}.\rse$ and if $v_d.\vtype=\topT$ call $\Sjump(R,j+1,\dirr,\alpha)$.
\end{enumerate}
\end{algo}

\begin{algo}{$\unlock(R,j,\dirr)$}{Unlock block if necessary}
\begin{enumerate}[label={\bfseries U\arabic*.}, leftmargin=8mm, noitemsep, topsep=3pt plus 3pt]
\item{}[Prepare] Set $a\gets r_j.\rne$, $b\gets r_j.\rse$, $c\gets r_j.\rsw$ and $\gamma\gets v_c.\vnorth$.
\item{}[Unlock if necessary] If $v_a.\vtype=\bottomT$ and $v_b.\vtype=\topT$ call $\Sjump(R,j+1,\dirl,\gamma)$.
\end{enumerate}
\end{algo}

To use Algorithm~\Mrect{} with this oracle, in line~M5 we also need to check whether $R^{[j-1]}$ is bottom-based or right-based (in addition to $R^{[j]}$), and whether $R^{[j-1]}$ or $R^{[j]}$ are one simple flip away from such a configuration.
Similarly, to use this oracle in conjunction with the oracle $\nextjump_{\cB_n(\cP)}$ defined in Section~\ref{sec:oracle-pattern}, we need to test containment of a pattern~$P$ in the rectangulation~$R$ after a jump of rectangle~$r_j$ not only via $\contains(R,j,P)$, but also using $\contains(R,j-1,P)$ and $\contains(R,j+1,P)$, as all three rectangles~$r_{j-1}$, $r_j$ and~$r_{j+1}$ may be modified through one call of~$\nextjump_{\cB_n}(R,j,d)$.
For details see our C++ implementation~\cite{cos_rect}.

We obtain the following analogue of Theorems~\ref{thm:pattern} and~\ref{thm:algo-rect}.

\begin{theorem}
\label{thm:algo-baligned}
Let $n\geq 3$.
For any set of patterns~$\cP$ that are neither bottom-based nor right-based nor simple-flippable to a bottom-based or right-based pattern, Algorithm~\Mrect{} with the oracle $\nextjump_{\cB_n(\cP)}$ defined in Section~\ref{sec:oracle-pattern}, which calls $\nextjump_{\cB_n}$ as defined above, visits every rectangulation from~$\cB_n(\cP)$ exactly once, performing a sequence of one T- or D-flip plus at most three simple flips in each step.
\end{theorem}

It remains to analyze the running time of this algorithm.

\begin{lemma}
\label{lem:Sequiv-time}
Each call $\nextjump_{\cB_n}(R,j,d)$ takes time~$\cO(1)$.
\end{lemma}

As we are dealing with a subset of diagonal rectangulations, the proof is very similar to the proof of Lemma~\ref{lem:diag-time}.

\begin{proof}
Consider any of the calls $\Sjump(R,j,d)$, $\Tjumph(R,j,d)$, $\Tjumpv(R,j,d)$ in lines~N2--N7 and let~$R'$ be the rectangulation after the call.
As we only consider the first insertion point of each vertical group and the last insertion point of each horizontal group of~$I(R^{[j-1]})$, we have $v(R,R')=0$ and $h(R,R')=0$, so the claim follows from Lemma~\ref{lem:jump-time}~(b)+(c).
\end{proof}

Lemma~\ref{lem:Sequiv-time} immediately yields the following result.

\begin{theorem}
\label{thm:next-block}
Algorithm~\Mrect{} with the oracle~$\nextjump_{\cB_n}$ takes time~$\cO(1)$ to visit each block-aligned rectangulation.
\end{theorem}

For the pattern avoidance version of this algorithm, we obtain the following runtime bounds.

\begin{theorem}
\label{thm:nextjumpP-S}
For any set of patterns $\cP\seq\big\{\millr,\milll\big\}$, Algorithm~\Mrect{} with the oracle~$\nextjump_{\cB_n(\cP)}$ visits each rectangulation from~$\cB_n(\cP)$ in time~$\cO(n)$.
\end{theorem}

\begin{proof}
The oracle $\nextjump_{\cB_n(\cP)}$ repeatedly calls the function~$\nextjump_{\cB_n}$.
Applying Theorem~\ref{thm:nextjumpP-time} with the bound $f_n=\cO(1)$ from Lemma~\ref{lem:Sequiv-time} and the bound $t_n=\cO(1)$ from Lemma~\ref{lem:contains-time}, the term $n\cdot(f_n+t_n)$ evaluates to $\cO(n)$, as claimed.
\end{proof}

\section{Counting pattern-avoiding rectangulations}
\label{sec:counting}

In this section we report on computer experiments that count pattern-avoiding rectangulations $\cC_n(\cP)$ for all interesting subsets of patterns $\cP\seq\{P_1,\ldots,P_8\}$ where $P_1=\millr$, $P_2=\milll$, $P_3=\zvu$, $P_4=\zhr$, $P_5=\zvd$, $P_6=\zhl$, $P_7=\hvert$, $P_8=\hhor$.
Clearly, we can omit sets of patterns that are equivalent to another set of patterns under $D_4$ actions (rotations and mirroring vertically or horizontally).
Table~\ref{tab:countRn} shows the results for generic rectangulations $\cC_n=\cR_n$ as a base class, and Table~\ref{tab:countBn} for block-aligned rectangulations~$\cC_n=\cB_n$ as a base class.
The set of patterns used in each row of the table is denoted by the pattern indices, omitting curly brackets and commas.
For example, the row 1478 refers to the set $\cP=\{P_1,P_4,P_7,P_8\}$.
When counting block-aligned rectangulations with our algorithms, the patterns $P_3,\ldots,P_8$ cannot be used, as they do not satisfy the conditions of Theorem~\ref{thm:algo-baligned}.

\renewcommand{\arraystretch}{0.7}
\setlength\tabcolsep{1pt}
\begin{table}[htp!]
\caption{Counts for pattern-avoiding rectangulations with generic rectangulations as a base class.}
\label{tab:countRn}
\scriptsize
\begin{tabular}{l|rrrrrrrrrrrrr|l}
Patterns $\cP$ & \multicolumn{13}{c|}{Counts $|\cR_n(\cP)|$ for $n=1,\ldots,12$} & \hspace{3mm}OEIS \\\hline
$\emptyset$ & $1,$ & $2,$ & $6,$ & $24,$ & $116,$ & $642,$ & $3938,$ & $26194,$ & $186042,$ & $1395008,$ & $10948768,$ & $89346128,$ & $\dots $ & \href{https://oeis.org/A342141}{A342141} \\
1 & $1,$ & $2,$ & $6,$ & $24,$ & $115,$ & $624,$ & $3712,$ & $23704,$ & $160140,$ & $1132628,$ & $8321372,$ & $63129494,$ & $\dots $ &  \\
3 & $1,$ & $2,$ & $6,$ & $23,$ & $104,$ & $530,$ & $2958,$ & $17734,$ & $112657,$ & $750726,$ & $5207910,$ & $37387881,$ & $\dots $ & \href{https://oeis.org/A117106}{A117106} ? \\
7 & $1,$ & $2,$ & $6,$ & $24,$ & $115,$ & $619,$ & $3607,$ & $22265,$ & $143667,$ & $960854,$ & $6622454,$ & $46841852,$ & $\dots $ &  \\
12 & $1,$ & $2,$ & $6,$ & $24,$ & $114,$ & $606,$ & $3494,$ & $21434,$ & $138100,$ & $926008,$ & $6418576,$ & $45755516,$ & $\dots $ &  \\
13 & $1,$ & $2,$ & $6,$ & $23,$ & $103,$ & $514,$ & $2779,$ & $15983,$ & $96557,$ & $607174,$ & $3947335,$ & $26393968,$ & $\dots $ &  \\
14 & $1,$ & $2,$ & $6,$ & $23,$ & $103,$ & $514,$ & $2779,$ & $15983,$ & $96557,$ & $607174,$ & $3947335,$ & $26393968,$ & $\dots $ &  \\
17 & $1,$ & $2,$ & $6,$ & $24,$ & $114,$ & $601,$ & $3391,$ & $20070,$ & $123156,$ & $777836,$ & $5031860,$ & $33225018,$ & $\dots $ &  \\
34 & $1,$ & $2,$ & $6,$ & $22,$ & $92,$ & $422,$ & $2074,$ & $10754,$ & $58202,$ & $326240,$ & $1882960,$ & $11140560,$ & $\dots $ & \href{https://oeis.org/A001181}{A001181} \\
35 & $1,$ & $2,$ & $6,$ & $22,$ & $94,$ & $450,$ & $2349,$ & $13128,$ & $77533,$ & $479250,$ & $3077864,$ & $20421177,$ & $\dots $ &  \\
36 & $1,$ & $2,$ & $6,$ & $22,$ & $92,$ & $422,$ & $2074,$ & $10754,$ & $58202,$ & $326240,$ & $1882960,$ & $11140560,$ & $\dots $ & \href{https://oeis.org/A001181}{A001181} ? \\
37 & $1,$ & $2,$ & $6,$ & $23,$ & $103,$ & $514,$ & $2779,$ & $15987,$ & $96664,$ & $608933,$ & $3970441,$ & $26661194,$ & $\dots $ &  \\
38 & $1,$ & $2,$ & $6,$ & $23,$ & $103,$ & $507,$ & $2641,$ & $14245,$ & $78619,$ & $441174,$ & $2508688,$ & $14429287,$ & $\dots $ &  \\
78 & $1,$ & $2,$ & $6,$ & $24,$ & $114,$ & $596,$ & $3276,$ & $18396,$ & $103718,$ & $581636,$ & $3229888,$ & $17730584,$ & $\dots $ &  \\
123 & $1,$ & $2,$ & $6,$ & $23,$ & $102,$ & $498,$ & $2606,$ & $14378,$ & $82725,$ & $492520,$ & $3017043,$ & $18933201,$ & $\dots $ &  \\
127 & $1,$ & $2,$ & $6,$ & $24,$ & $113,$ & $583,$ & $3183,$ & $18077,$ & $105813,$ & $634838,$ & $3889236,$ & $24262094,$ & $\dots $ &  \\
134 & $1,$ & $2,$ & $6,$ & $22,$ & $91,$ & $408,$ & $1938,$ & $9614,$ & $49335,$ & $260130,$ & $1402440,$ & $7702632,$ & $\dots $ & \href{https://oeis.org/A000139}{A000139} ? \\
135 & $1,$ & $2,$ & $6,$ & $22,$ & $93,$ & $436,$ & $2209,$ & $11889,$ & $67159,$ & $394692,$ & $2397407,$ & $14974319,$ & $\dots $ &  \\
136 & $1,$ & $2,$ & $6,$ & $22,$ & $91,$ & $408,$ & $1938,$ & $9614,$ & $49335,$ & $260130,$ & $1402440,$ & $7702632,$ & $\dots $ & \href{https://oeis.org/A000139}{A000139} ? \\
137 & $1,$ & $2,$ & $6,$ & $23,$ & $102,$ & $498,$ & $2605,$ & $14362,$ & $82567,$ & $491285,$ & $3008821,$ & $18886524,$ & $\dots $ &  \\
138 & $1,$ & $2,$ & $6,$ & $23,$ & $102,$ & $491,$ & $2472,$ & $12763,$ & $66908,$ & $354396,$ & $1892049,$ & $10169071,$ & $\dots $ &  \\
145 & $1,$ & $2,$ & $6,$ & $22,$ & $91,$ & $408,$ & $1938,$ & $9614,$ & $49335,$ & $260130,$ & $1402440,$ & $7702632,$ & $\dots $ & \href{https://oeis.org/A000139}{A000139} ? \\
147 & $1,$ & $2,$ & $6,$ & $23,$ & $102,$ & $491,$ & $2472,$ & $12763,$ & $66908,$ & $354396,$ & $1892049,$ & $10169071,$ & $\dots $ &  \\
148 & $1,$ & $2,$ & $6,$ & $23,$ & $102,$ & $498,$ & $2605,$ & $14362,$ & $82567,$ & $491285,$ & $3008821,$ & $18886524,$ & $\dots $ &  \\
178 & $1,$ & $2,$ & $6,$ & $24,$ & $113,$ & $578,$ & $3070,$ & $16496,$ & $88378,$ & $468780,$ & $2455332,$ & $12694892,$ & $\dots $ &  \\
345 & $1,$ & $2,$ & $6,$ & $21,$ & $82,$ & $346,$ & $1547,$ & $7236,$ & $35090,$ & $175268,$ & $897273,$ & $4690392,$ & $\dots $ & \href{https://oeis.org/A281784}{A281784} ? \\
347 & $1,$ & $2,$ & $6,$ & $22,$ & $91,$ & $406,$ & $1905,$ & $9264,$ & $46288,$ & $236364,$ & $1229209,$ & $6494549,$ & $\dots $ &  \\
357 & $1,$ & $2,$ & $6,$ & $22,$ & $93,$ & $439,$ & $2257,$ & $12407,$ & $71963,$ & $436176,$ & $2742686,$ & $17791880,$ & $\dots $ &  \\
358 & $1,$ & $2,$ & $6,$ & $22,$ & $93,$ & $427,$ & $2044,$ & $9975,$ & $49089,$ & $242458,$ & $1199855,$ & $5947447,$ & $\dots $ &  \\
367 & $1,$ & $2,$ & $6,$ & $22,$ & $91,$ & $406,$ & $1905,$ & $9264,$ & $46288,$ & $236364,$ & $1229209,$ & $6494549,$ & $\dots $ &  \\
378 & $1,$ & $2,$ & $6,$ & $23,$ & $102,$ & $491,$ & $2462,$ & $12534,$ & $63842,$ & $322875,$ & $1615726,$ & $7990347,$ & $\dots $ &  \\
1234 & $1,$ & $2,$ & $6,$ & $22,$ & $90,$ & $394,$ & $1806,$ & $8558,$ & $41586,$ & $206098,$ & $1037718,$ & $5293446,$ & $\dots $ & \href{https://oeis.org/A006318}{A006318} \\
1235 & $1,$ & $2,$ & $6,$ & $22,$ & $92,$ & $422,$ & $2073,$ & $10738,$ & $58029,$ & $324648,$ & $1869482,$ & $11031813,$ & $\dots $ &  \\
1236 & $1,$ & $2,$ & $6,$ & $22,$ & $90,$ & $394,$ & $1806,$ & $8558,$ & $41586,$ & $206098,$ & $1037718,$ & $5293446,$ & $\dots $ & \href{https://oeis.org/A006318}{A006318} ? \\
1237 & $1,$ & $2,$ & $6,$ & $23,$ & $101,$ & $482,$ & $2437,$ & $12877,$ & $70514,$ & $397823,$ & $2302074,$ & $13614952,$ & $\dots $ &  \\
1238 & $1,$ & $2,$ & $6,$ & $23,$ & $101,$ & $475,$ & $2309,$ & $11409,$ & $56879,$ & $285220,$ & $1436772,$ & $7267279,$ & $\dots $ &  \\
1278 & $1,$ & $2,$ & $6,$ & $24,$ & $112,$ & $560,$ & $2872,$ & $14780,$ & $75512,$ & $381320,$ & $1901292,$ & $9366128,$ & $\dots $ &  \\
1345 & $1,$ & $2,$ & $6,$ & $21,$ & $81,$ & $334,$ & $1446,$ & $6498,$ & $30074,$ & $142556,$ & $689248,$ & $3388453,$ & $\dots $ &  \\
1346 & $1,$ & $2,$ & $6,$ & $21,$ & $81,$ & $334,$ & $1446,$ & $6498,$ & $30074,$ & $142556,$ & $689248,$ & $3388453,$ & $\dots $ &  \\
1347 & $1,$ & $2,$ & $6,$ & $22,$ & $90,$ & $392,$ & $1774,$ & $8236,$ & $38961,$ & $187093,$ & $909961,$ & $4475961,$ & $\dots $ &  \\
1348 & $1,$ & $2,$ & $6,$ & $22,$ & $90,$ & $392,$ & $1774,$ & $8236,$ & $38961,$ & $187093,$ & $909961,$ & $4475961,$ & $\dots $ &  \\
1357 & $1,$ & $2,$ & $6,$ & $22,$ & $92,$ & $425,$ & $2119,$ & $11210,$ & $62164,$ & $358200,$ & $2130760,$ & $13019572,$ & $\dots $ &  \\
1358 & $1,$ & $2,$ & $6,$ & $22,$ & $92,$ & $413,$ & $1914,$ & $8981,$ & $42310,$ & $199500,$ & $940788,$ & $4437867,$ & $\dots $ &  \\
1367 & $1,$ & $2,$ & $6,$ & $22,$ & $90,$ & $392,$ & $1774,$ & $8236,$ & $38961,$ & $187093,$ & $909961,$ & $4475961,$ & $\dots $ &  \\
1378 & $1,$ & $2,$ & $6,$ & $23,$ & $101,$ & $475,$ & $2298,$ & $11178,$ & $54030,$ & $258192,$ & $1217964,$ & $5673144,$ & $\dots $ &  \\
1457 & $1,$ & $2,$ & $6,$ & $22,$ & $90,$ & $392,$ & $1774,$ & $8236,$ & $38961,$ & $187093,$ & $909961,$ & $4475961,$ & $\dots $ &  \\
1478 & $1,$ & $2,$ & $6,$ & $23,$ & $101,$ & $475,$ & $2298,$ & $11178,$ & $54030,$ & $258192,$ & $1217964,$ & $5673144,$ & $\dots $ &  \\
3456 & $1,$ & $2,$ & $6,$ & $20,$ & $72,$ & $274,$ & $1088,$ & $4470,$ & $18884,$ & $81652,$ & $360054,$ & $1614618,$ & $\dots $ &  \\
3457 & $1,$ & $2,$ & $6,$ & $21,$ & $81,$ & $335,$ & $1461,$ & $6643,$ & $31235,$ & $150960,$ & $746522,$ & $3764017,$ & $\dots $ &  \\
3458 & $1,$ & $2,$ & $6,$ & $21,$ & $81,$ & $330,$ & $1386,$ & $5925,$ & $25614,$ & $111638,$ & $489937,$ & $2164127,$ & $\dots $ &  \\
3478 & $1,$ & $2,$ & $6,$ & $22,$ & $90,$ & $390,$ & $1736,$ & $7794,$ & $34926,$ & $155340,$ & $683920,$ & $2977794,$ & $\dots $ &  \\
3578 & $1,$ & $2,$ & $6,$ & $22,$ & $92,$ & $416,$ & $1952,$ & $9270,$ & $43986,$ & $207340,$ & $968862,$ & $4486184,$ & $\dots $ &  \\
3678 & $1,$ & $2,$ & $6,$ & $22,$ & $90,$ & $390,$ & $1736,$ & $7794,$ & $34926,$ & $155340,$ & $683920,$ & $2977794,$ & $\dots $ &  \\
12345 & $1,$ & $2,$ & $6,$ & $21,$ & $80,$ & $322,$ & $1347,$ & $5798,$ & $25512,$ & $114236,$ & $518848,$ & $2384538,$ & $\dots $ & \href{https://oeis.org/A106228}{A106228} ? \\
12347 & $1,$ & $2,$ & $6,$ & $22,$ & $89,$ & $378,$ & $1647,$ & $7286,$ & $32574,$ & $146866,$ & $667088,$ & $3050619,$ & $\dots $ &  \\
12357 & $1,$ & $2,$ & $6,$ & $22,$ & $91,$ & $411,$ & $1985,$ & $10099,$ & $53547,$ & $293602,$ & $1655170,$ & $9551440,$ & $\dots $ &  \\
12358 & $1,$ & $2,$ & $6,$ & $22,$ & $91,$ & $399,$ & $1788,$ & $8057,$ & $36291,$ & $163158,$ & $732385,$ & $3285369,$ & $\dots $ &  \\
12367 & $1,$ & $2,$ & $6,$ & $22,$ & $89,$ & $378,$ & $1647,$ & $7286,$ & $32574,$ & $146866,$ & $667088,$ & $3050619,$ & $\dots $ &  \\
12378 & $1,$ & $2,$ & $6,$ & $23,$ & $100,$ & $459,$ & $2140,$ & $9944,$ & $45676,$ & $206855,$ & $923746,$ & $4073045,$ & $\dots $ &  \\
13456 & $1,$ & $2,$ & $6,$ & $20,$ & $71,$ & $264,$ & $1018,$ & $4042,$ & $16438,$ & $68196,$ & $287724,$ & $1231514,$ & $\dots $ &  \\
13457 & $1,$ & $2,$ & $6,$ & $21,$ & $80,$ & $323,$ & $1362,$ & $5941,$ & $26628,$ & $122036,$ & $569781,$ & $2702496,$ & $\dots $ &  \\
13458 & $1,$ & $2,$ & $6,$ & $21,$ & $80,$ & $318,$ & $1290,$ & $5287,$ & $21803,$ & $90351,$ & $376174,$ & $1573975,$ & $\dots $ &  \\
13467 & $1,$ & $2,$ & $6,$ & $21,$ & $80,$ & $318,$ & $1290,$ & $5287,$ & $21803,$ & $90351,$ & $376174,$ & $1573975,$ & $\dots $ &  \\
13468 & $1,$ & $2,$ & $6,$ & $21,$ & $80,$ & $323,$ & $1362,$ & $5941,$ & $26628,$ & $122036,$ & $569781,$ & $2702496,$ & $\dots $ &  \\
13478 & $1,$ & $2,$ & $6,$ & $22,$ & $89,$ & $376,$ & $1610,$ & $6878,$ & $29094,$ & $121498,$ & $500688,$ & $2037758,$ & $\dots $ &  \\
13578 & $1,$ & $2,$ & $6,$ & $22,$ & $91,$ & $402,$ & $1824,$ & $8318,$ & $37750,$ & $169880,$ & $757488,$ & $3348274,$ & $\dots $ &  \\
13678 & $1,$ & $2,$ & $6,$ & $22,$ & $89,$ & $376,$ & $1610,$ & $6878,$ & $29094,$ & $121498,$ & $500688,$ & $2037758,$ & $\dots $ &  \\
14578 & $1,$ & $2,$ & $6,$ & $22,$ & $89,$ & $376,$ & $1610,$ & $6878,$ & $29094,$ & $121498,$ & $500688,$ & $2037758,$ & $\dots $ &  \\
34567 & $1,$ & $2,$ & $6,$ & $20,$ & $71,$ & $263,$ & $1006,$ & $3949,$ & $15839,$ & $64700,$ & $268477,$ & $1129385,$ & $\dots $ &  \\
34578 & $1,$ & $2,$ & $6,$ & $21,$ & $80,$ & $319,$ & $1300,$ & $5340,$ & $21946,$ & $89909,$ & $366626,$ & $1487463,$ & $\dots $ &  \\
123456 & $1,$ & $2,$ & $6,$ & $20,$ & $70,$ & $254,$ & $948,$ & $3618,$ & $14058,$ & $55432,$ & $221262,$ & $892346,$ & $\dots $ & \href{https://oeis.org/A078482}{A078482} \\
123457 & $1,$ & $2,$ & $6,$ & $21,$ & $79,$ & $311,$ & $1265,$ & $5275,$ & $22431,$ & $96900,$ & $424068,$ & $1876143,$ & $\dots $ & \href{https://oeis.org/A033321}{A033321} ? \\
123458 & $1,$ & $2,$ & $6,$ & $21,$ & $79,$ & $306,$ & $1196,$ & $4681,$ & $18308,$ & $71564,$ & $279820,$ & $1095533,$ & $\dots $ &  \\
123478 & $1,$ & $2,$ & $6,$ & $22,$ & $88,$ & $362,$ & $1488,$ & $6034,$ & $24024,$ & $93830,$ & $359824,$ & $1357088,$ & $\dots $ &  \\
123578 & $1,$ & $2,$ & $6,$ & $22,$ & $90,$ & $388,$ & $1700,$ & $7434,$ & $32212,$ & $138040,$ & $585246,$ & $2457712,$ & $\dots $ &  \\
123678 & $1,$ & $2,$ & $6,$ & $22,$ & $88,$ & $362,$ & $1488,$ & $6034,$ & $24024,$ & $93830,$ & $359824,$ & $1357088,$ & $\dots $ &  \\
134567 & $1,$ & $2,$ & $6,$ & $20,$ & $70,$ & $253,$ & $938,$ & $3553,$ & $13708,$ & $53736,$ & $213588,$ & $859335,$ & $\dots $ &  \\
134578 & $1,$ & $2,$ & $6,$ & $21,$ & $79,$ & $307,$ & $1206,$ & $4738,$ & $18532,$ & $72070,$ & $278718,$ & $1072739,$ & $\dots $ &  \\
134678 & $1,$ & $2,$ & $6,$ & $21,$ & $79,$ & $307,$ & $1206,$ & $4738,$ & $18532,$ & $72070,$ & $278718,$ & $1072739,$ & $\dots $ &  \\
345678 & $1,$ & $2,$ & $6,$ & $20,$ & $70,$ & $252,$ & $924,$ & $3432,$ & $12870,$ & $48620,$ & $184756,$ & $705432,$ & $\dots $ & \href{https://oeis.org/A000984}{A000984} ? \\
1234567 & $1,$ & $2,$ & $6,$ & $20,$ & $69,$ & $243,$ & $870,$ & $3159,$ & $11611,$ & $43130,$ & $161691,$ & $611065,$ & $\dots $ &  \\
1234578 & $1,$ & $2,$ & $6,$ & $21,$ & $78,$ & $295,$ & $1114,$ & $4166,$ & $15390,$ & $56167,$ & $202738,$ & $724813,$ & $\dots $ &  \\
1345678 & $1,$ & $2,$ & $6,$ & $20,$ & $69,$ & $242,$ & $858,$ & $3068,$ & $11050,$ & $40052,$ & $145996,$ & $534888,$ & $\dots $ & \href{https://oeis.org/A026029}{A026029} ? \\
12345678 & $1,$ & $2,$ & $6,$ & $20,$ & $68,$ & $232,$ & $792,$ & $2704,$ & $9232,$ & $31520,$ & $107616,$ & $367424,$ & $\dots $ & \href{https://oeis.org/A006012}{A006012} \\
\end{tabular}
\end{table}

\begin{table}[htp!]
\caption{Counts for pattern-avoiding rectangulations with block-aligned rectangulations as a base class.}
\label{tab:countBn}
\scriptsize
\begin{tabular}{l|rrrrrrrrrrrrrr|l}
Patterns $\cP$ & \multicolumn{14}{c|}{Counts $|\cB_n(\cP)|$ for $n=1,\ldots,13$} & \hspace{1.5mm}OEIS \\\hline
$\emptyset$ & $1,$ & $1,$ & $2,$ & $6,$ & $22,$ & $88,$ & $374,$ & $1668,$ & $7744,$ & $37182,$ & $183666,$ & $929480,$ & $4803018,$ & $\dots $ & \href{https://oeis.org/A214358}{A214358} \\
1 & $1,$ & $1,$ & $2,$ & $6,$ & $21,$ & $79,$ & $312,$ & $1280,$ & $5416,$ & $23506,$ & $104198,$ & $470192,$ & $2154204,$ & $\dots $ &  \\
12 & $1,$ & $1,$ & $2,$ & $6,$ & $20,$ & $70,$ & $254,$ & $948,$ & $3618,$ & $14058,$ & $55432,$ & $221262,$ & $892346,$ & $\dots $ & \href{https://oeis.org/A078482}{A078482} \\
\end{tabular}
\end{table}

Several of these counting sequences appear in the OEIS~\cite{oeis}, and are related to pattern-avoiding permutations (see e.g.~\cite{MR3882946}).
The matching OEIS entries marked with ? are observed through are numerical experiments, but no formal bijective proof has been obtained yet, even though finding one should be straightforward in some cases.
The last two rows in Table~\ref{tab:countRn} with ? are interesting, as the correspondence to the objects mentioned in those OEIS entries is not obvious.
This is true in particular for OEIS sequence~A000984, which are the central binomial coefficients~$\binom{2n}{n}$.

\section{Open questions}
\label{sec:open}

The subject of pattern-avoiding rectangulations deserves further systematic investigation, and may still hold many undiscovered gems; recall Table~\ref{tab:countRn}.
Understanding the number of pattern-avoiding rectangulations that are obtained by rectangle insertion may also help to improve the runtime bounds for our generation algorithms (recall Remark~\ref{rem:lb}).
Moreover, does the avoidance of a rectangulation pattern always correspond to the avoidance of a particular permutation pattern, and what is this correspondence?

In our paper we considered R-equivalence and S-equivalence of generic rectangulations~$\cR_n$, and these equivalence relations are induced by wall slides, or by wall slides and simple flips, respectively.
Considering all three basic flip operations~$F=\{W,S,T\}$, namely wall slides, simple flips, and T-flips, there are $2^3=8$ possible subsets of~$F$ to induce an equivalence relation on~$\cR_n$.
Which of these equivalence relations are interesting (apart from $\emptyset$, $\{W\}$ and $\{W,S\}$ considered here), and what are suitable representatives that can be generated efficiently?

Another interesting question to investigate would be Gray codes for rectangulations of point sets as introduced by Ackerman, Barequet and Pinter~\cite{MR2244135}.
Some first results in this direction have been obtained by Yamanaka, Rahman and Nakano~\cite{DBLP:journals/ieicet/YamanakaRN18}.
In particular, can we apply our permutation-based generation framework for this task?

\section*{Acknowledgements}

We thankfully acknowledge several discussions in the early phases of this manuscript with Hung P.~Hoang, which took place at the 17th Gremo Workshop on Open Problems in Switzerland.
We thank the organizers for the invitation to the workshop, and the other participants for the pleasant and stimulating working atmosphere.
Furthermore, we thank the reviewers for their numerous helpful comments, which helped improving the manuscript.

\bibliographystyle{alpha}
\bibliography{refs}

\newcommand{\etalchar}[1]{$^{#1}$}
\begin{thebibliography}{ABBM{\etalchar{+}}13}

\bibitem[ABBM{\etalchar{+}}13]{MR3084577}
A.~Asinowski, G.~Barequet, M.~Bousquet-M{\'e}lou, T.~Mansour, and R.~Y. Pinter.
\newblock Orders induced by segments in floorplans and (2-14-3,
  3-41-2)-avoiding permutations.
\newblock {\em Electron. J. Combin.}, 20(2):Paper 35, 43, 2013.

\bibitem[ABP06a]{MR2233287}
E.~Ackerman, G.~Barequet, and R.~Y. Pinter.
\newblock A bijection between permutations and floorplans, and its
  applications.
\newblock {\em Discrete Appl. Math.}, 154(12):1674--1684, 2006.

\bibitem[ABP06b]{MR2244135}
E.~Ackerman, G.~Barequet, and R.~Y. Pinter.
\newblock On the number of rectangulations of a planar point set.
\newblock {\em J. Combin. Theory Ser. A}, 113(6):1072--1091, 2006.

\bibitem[AF96]{MR1380066}
D.~Avis and K.~Fukuda.
\newblock Reverse search for enumeration.
\newblock {\em Discrete Appl. Math.}, 65(1-3):21--46, 1996.
\newblock First International Colloquium on Graphs and Optimization (GOI), 1992
  (Grimentz).

\bibitem[AM10]{MR2601798}
A.~Asinowski and T.~Mansour.
\newblock Separable {$d$}-permutations and guillotine partitions.
\newblock {\em Ann. Comb.}, 14(1):17--43, 2010.

\bibitem[ANY07]{amano_nakano_yamanaka_2007}
K.~Amano, S.~Nakano, and K.~Yamanaka.
\newblock On the number of rectangular drawings: Exact counting and lower and
  upper bounds.
\newblock IPSJ SIG Technical Report 2007-AL-115 (5), 2007.

\bibitem[BER76]{MR0424386}
J.~R. Bitner, G.~Ehrlich, and E.~M. Reingold.
\newblock Efficient generation of the binary reflected {G}ray code and its
  applications.
\newblock {\em Comm. ACM}, 19(9):517--521, 1976.

\bibitem[BGRR18]{MR3882946}
M.~Bouvel, V.~Guerrini, A.~Rechnitzer, and S.~Rinaldi.
\newblock Semi-{B}axter and strong-{B}axter: two relatives of the {B}axter
  sequence.
\newblock {\em SIAM J. Discrete Math.}, 32(4):2795--2819, 2018.

\bibitem[CM14]{MR3167602}
J.~Conant and T.~Michaels.
\newblock On the number of tilings of a square by rectangles.
\newblock {\em Ann. Comb.}, 18(1):21--34, 2014.

\bibitem[cos]{cos_rect}
The Combinatorial Object Server: Generate rectangulations.
  \url{http://www.combos.org/rect}.

\bibitem[CSS18]{MR3878132}
J.~Cardinal, V.~Sacrist{\'a}n, and R.~I. Silveira.
\newblock A note on flips in diagonal rectangulations.
\newblock {\em Discrete Math. Theor. Comput. Sci.}, 20(2):Paper No. 14, 22,
  2018.

\bibitem[Ehr73]{MR0366085}
G.~Ehrlich.
\newblock Loopless algorithms for generating permutations, combinations, and
  other combinatorial configurations.
\newblock {\em J. Assoc. Comput. Mach.}, 20:500--513, 1973.

\bibitem[EMSV12]{MR2967467}
D.~Eppstein, E.~Mumford, B.~Speckmann, and K.~Verbeek.
\newblock Area-universal and constrained rectangular layouts.
\newblock {\em SIAM J. Comput.}, 41(3):537--564, 2012.

\bibitem[Fel13]{MR3205156}
S.~Felsner.
\newblock Rectangle and square representations of planar graphs.
\newblock In {\em Thirty essays on geometric graph theory}, pages 213--248.
  Springer, New York, 2013.

\bibitem[FFNO11]{MR2763051}
S.~Felsner, {\'E}.~Fusy, M.~Noy, and D.~Orden.
\newblock Bijections for {B}axter families and related objects.
\newblock {\em J. Combin. Theory Ser. A}, 118(3):993--1020, 2011.

\bibitem[FIT09]{fujimaki_inoue_takahashi_2009}
R.~Fujimaki, Y.~Inoue, and T.~Takahashi.
\newblock An asymptotic estimate of the numbers of rectangular drawings or
  floorplans.
\newblock In {\em 2009 IEEE International Symposium on Circuits and Systems
  (ISCAS)}, pages 856--859, 2009.

\bibitem[FNT21]{felsner_nathenson_toth_2021}
S.~Felsner, A.~Nathenson, and C.~D. T{\'o}th.
\newblock Aspect ratio universal rectangular layouts.
\newblock Manuscript, 2021.

\bibitem[Fus09]{MR2509359}
{\'E}.~Fusy.
\newblock Transversal structures on triangulations: a combinatorial study and
  straight-line drawings.
\newblock {\em Discrete Math.}, 309(7):1870--1894, 2009.

\bibitem[He14]{MR3192492}
B.~D. He.
\newblock A simple optimal binary representation of mosaic floorplans and
  {B}axter permutations.
\newblock {\em Theoret. Comput. Sci.}, 532:40--50, 2014.

\bibitem[HHC{\etalchar{+}}00]{DBLP:conf/iccad/HongHCGDCG00}
X.~Hong, G.~Huang, Y.~Cai, J.~Gu, S.~Dong, C.{-}K. Cheng, and J.~Gu.
\newblock Corner block list: An effective and efficient topological
  representation of non-slicing floorplan.
\newblock In E.~Sentovich, editor, {\em Proceedings of the 2000 {IEEE/ACM}
  International Conference on Computer-Aided Design, 2000, San Jose,
  California, USA, November 5-9, 2000}, pages 8--12. {IEEE} Computer Society,
  2000.

\bibitem[HHMW20]{DBLP:conf/soda/HartungHMW20}
E.~Hartung, H.~P. Hoang, T.~M{\"u}tze, and A.~Williams.
\newblock Combinatorial generation via permutation languages.
\newblock In Shuchi Chawla, editor, {\em Proceedings of the 2020 {ACM-SIAM}
  Symposium on Discrete Algorithms, {SODA} 2020, Salt Lake City, UT, USA,
  January 5-8, 2020}, pages 1214--1225. {SIAM}, 2020.

\bibitem[HHMW21]{perm_series_i}
E.~Hartung, H.~P. Hoang, T.~M{\"u}tze, and A.~Williams.
\newblock Combinatorial generation via permutation languages. {I}.
  {F}undamentals.
\newblock To appear in {\it Trans.\ Amer.\ Math.\ Soc.}; preprint available at
  \url{https://arxiv.org/abs/1906.06069}, 2021.

\bibitem[HM21]{perm_series_ii}
H.~P. Hoang and T.~M{\"u}tze.
\newblock Combinatorial generation via permutation languages. {II}. {L}attice
  congruences.
\newblock To appear in {\it Israel J.\ Math.}; preprint available at
  \url{https://arxiv.org/abs/1911.12078}, 2021.

\bibitem[ITF09]{DBLP:journals/ieicet/InoueTF09}
Y.~Inoue, T.~Takahashi, and R.~Fujimaki.
\newblock Counting rectangular drawings or floorplans in polynomial time.
\newblock {\em {IEICE} Trans. Fundam. Electron. Commun. Comput. Sci.},
  92-A(4):1115--1120, 2009.

\bibitem[KH97]{MR1432861}
G.~Kant and X.~He.
\newblock Regular edge labeling of {$4$}-connected plane graphs and its
  applications in graph drawing problems.
\newblock {\em Theoret. Comput. Sci.}, 172(1-2):175--193, 1997.

\bibitem[Knu11]{MR3444818}
D.~E. Knuth.
\newblock {\em The Art of Computer Programming. {V}ol. 4{A}. {C}ombinatorial
  algorithms. {P}art 1}.
\newblock Addison-Wesley, Upper Saddle River, NJ, 2011.

\bibitem[Lei21]{leifheit_2021}
L.~J. Leifheit.
\newblock Combinatorial properties of rectangulations.
\newblock Master's thesis, TU Berlin, 2021.

\bibitem[LR12]{MR2871762}
S.~Law and N.~Reading.
\newblock The {H}opf algebra of diagonal rectangulations.
\newblock {\em J. Combin. Theory Ser. A}, 119(3):788--824, 2012.

\bibitem[Mee19]{meehan_2019}
E.~Meehan.
\newblock The {H}opf algebra of generic rectangulations.
\newblock \url{https://arxiv.org/abs/1903.09874}, 2019.

\bibitem[MSL76]{mitchell_steadman_liggett_1976}
W.~J. Mitchell, J.~P. Steadman, and R.~S. Liggett.
\newblock Synthesis and optimization of small rectangular floor plans.
\newblock {\em Environment and Planning B: Planning and Design}, 3(1):37--70,
  1976.

\bibitem[Nak01]{MR1917735}
S.~Nakano.
\newblock Enumerating floorplans with {$n$} rooms.
\newblock In {\em Algorithms and computation ({C}hristchurch, 2001)}, volume
  2223 of {\em Lecture Notes in Comput. Sci.}, pages 107--115. Springer,
  Berlin, 2001.

\bibitem[{oei}20]{oeis}
{OEIS} {F}oundation {I}nc. {T}he on-line encyclopedia of integer sequences,
  2020.
\newblock \url{http://oeis.org}.

\bibitem[Ott82]{DBLP:conf/dac/Otten82}
R.~H. J.~M. Otten.
\newblock Automatic floorplan design.
\newblock In J.~S. Crabbe, C.~E. Radke, and H.~Ofek, editors, {\em Proceedings
  of the 19th Design Automation Conference, {DAC} '82, Las Vegas, Nevada, USA,
  June 14-16, 1982}, pages 261--267. {ACM/IEEE}, 1982.

\bibitem[PPR21]{padrol_pilaud_ritter_2021}
A.~Padrol, V.~Pilaud, and J.~Ritter.
\newblock Shard polytopes.
\newblock To appear in {\it Proceedings of the 33rd International Conference on
  Formal Power Series and Algebraic Combinatorics}; preprint available at
  \url{https://arxiv.org/abs/2007.01008}, 2021.

\bibitem[PS19]{MR3964495}
V.~Pilaud and F.~Santos.
\newblock Quotientopes.
\newblock {\em Bull. Lond. Math. Soc.}, 51(3):406--420, 2019.

\bibitem[Rea12]{MR2864445}
N.~Reading.
\newblock Generic rectangulations.
\newblock {\em European J. Combin.}, 33(4):610--623, 2012.

\bibitem[Rus16]{MR3523863}
F.~Ruskey.
\newblock Combinatorial {G}ray code.
\newblock In M.-Y. Kao, editor, {\em Encyclopedia of Algorithms}, pages
  342--347. Springer, 2016.

\bibitem[Sav97]{MR1491049}
C.~Savage.
\newblock A survey of combinatorial {G}ray codes.
\newblock {\em SIAM Rev.}, 39(4):605--629, 1997.

\bibitem[SC03]{shen_chu_2003}
Z.~C. Shen and C.~C.~N. Chu.
\newblock Bounds on the number of slicing, mosaic, and general floorplans.
\newblock {\em IEEE Transactions on Computer-Aided Design of Integrated
  Circuits and Systems}, 22(10):1354--1361, 2003.

\bibitem[TN04]{DBLP:journals/scjapan/TakagiN04}
M.~Takagi and S.~Nakano.
\newblock Listing all rectangular drawings with certain properties.
\newblock {\em Systems and Computers in Japan}, 35(4):1--8, 2004.

\bibitem[vKS07]{MR2331063}
M.~van Kreveld and B.~Speckmann.
\newblock On rectangular cartograms.
\newblock {\em Comput. Geom.}, 37(3):175--187, 2007.

\bibitem[Wil13]{DBLP:conf/wads/Williams13}
A.~Williams.
\newblock The greedy {G}ray code algorithm.
\newblock In {\em Algorithms and Data Structures - 13th International
  Symposium, {WADS} 2013, London, ON, Canada, August 12-14, 2013. Proceedings},
  pages 525--536, 2013.

\bibitem[YCCG03]{DBLP:journals/todaes/YaoCCG03}
B.~Yao, H.~Chen, C.-K. Cheng, and R.~L. Graham.
\newblock Floorplan representations: Complexity and connections.
\newblock {\em {ACM} Trans. Design Autom. Electr. Syst.}, 8(1):55--80, 2003.

\bibitem[YCYN06]{DBLP:journals/ieicet/YoshiiCYN06}
S.~Yoshii, D.~Chigira, K.~Yamanaka, and S.~Nakano.
\newblock Constant time generation of rectangular drawings with exactly
  \emph{n} faces.
\newblock {\em {IEICE} Trans. Fundam. Electron. Commun. Comput. Sci.},
  89-A(9):2445--2450, 2006.

\bibitem[YRN18]{DBLP:journals/ieicet/YamanakaRN18}
K.~Yamanaka, M.~S. Rahman, and S.~Nakano.
\newblock Enumerating floorplans with columns.
\newblock {\em {IEICE} Trans. Fundam. Electron. Commun. Comput. Sci.},
  101-A(9):1392--1397, 2018.

\end{thebibliography}
\appendix

\newpage

\section{Visualization of Gray codes}

In this section we visualize the Gray codes obtained from our algorithms for generic rectangulations, diagonal rectangulations and block-aligned rectangulations (without any forbidden patterns).
The corresponding 2-clumped permutations are shown below each rectangulation.

\subsection{Generic rectangulations}
\hphantom{x}

\renewcommand{\arraystretch}{0.5}
\setlength\tabcolsep{3pt}

\begin{figure}[h!]
\begin{tabular}{c}
\begin{tikzpicture}[scale=120/100,font=\tiny]
\fill[p1col1] (0.0000,1.0000) rectangle (1.0000,0.0000);
\end{tikzpicture}
 \\
1 \\[1mm]
\end{tabular}
\caption{$n=1$}
\end{figure}

\begin{figure}[h!]
\begin{tabular}{cc}
\begin{tikzpicture}[scale=60/100,font=\tiny]
\fill[p1col1] (0.0000,2.0000) rectangle (1.0000,0.0000);
\fill[p1col2] (1.0000,2.0000) rectangle (2.0000,0.0000);
\end{tikzpicture}
 & \begin{tikzpicture}[scale=60/100,font=\tiny]
\fill[p1col1] (0.0000,2.0000) rectangle (2.0000,1.0000);
\fill[p1col2] (0.0000,1.0000) rectangle (2.0000,0.0000);
\end{tikzpicture}
 \\
12 & 21 \\
\end{tabular}
\caption{$n=2$}
\end{figure}

\begin{figure}[h!]
\begin{tabular}{cccccc}
\begin{tikzpicture}[scale=40/100,font=\tiny]
\fill[p1col1] (0.0000,3.0000) rectangle (1.0000,0.0000);
\fill[p1col2] (1.0000,3.0000) rectangle (2.0000,0.0000);
\fill[p1col3] (2.0000,3.0000) rectangle (3.0000,0.0000);
\end{tikzpicture}
 & \begin{tikzpicture}[scale=40/100,font=\tiny]
\fill[p1col1] (0.0000,3.0000) rectangle (1.5000,0.0000);
\fill[p1col2] (1.5000,3.0000) rectangle (3.0000,1.5000);
\fill[p1col3] (1.5000,1.5000) rectangle (3.0000,0.0000);
\end{tikzpicture}
 & \begin{tikzpicture}[scale=40/100,font=\tiny]
\fill[p1col1] (0.0000,3.0000) rectangle (1.5000,1.5000);
\fill[p1col2] (1.5000,3.0000) rectangle (3.0000,1.5000);
\fill[p1col3] (0.0000,1.5000) rectangle (3.0000,0.0000);
\end{tikzpicture}
 & \begin{tikzpicture}[scale=40/100,font=\tiny]
\fill[p1col1] (0.0000,3.0000) rectangle (3.0000,2.0000);
\fill[p1col2] (0.0000,2.0000) rectangle (3.0000,1.0000);
\fill[p1col3] (0.0000,1.0000) rectangle (3.0000,0.0000);
\end{tikzpicture}
 & \begin{tikzpicture}[scale=40/100,font=\tiny]
\fill[p1col1] (0.0000,3.0000) rectangle (3.0000,1.5000);
\fill[p1col2] (0.0000,1.5000) rectangle (1.5000,0.0000);
\fill[p1col3] (1.5000,1.5000) rectangle (3.0000,0.0000);
\end{tikzpicture}
 & \begin{tikzpicture}[scale=40/100,font=\tiny]
\fill[p1col1] (0.0000,3.0000) rectangle (1.5000,1.5000);
\fill[p1col2] (0.0000,1.5000) rectangle (1.5000,0.0000);
\fill[p1col3] (1.5000,3.0000) rectangle (3.0000,0.0000);
\end{tikzpicture}
 \\
123 & 132 & 312 & 321 & 231 & 213 \\
\end{tabular}
\caption{$n=3$}
\end{figure}

\begin{figure}[h!]
\begin{tabular}{ccccccccccc}
\begin{tikzpicture}[scale=30/100,font=\tiny]
\fill[p1col1] (0.0000,4.0000) rectangle (1.0000,0.0000);
\fill[p1col2] (1.0000,4.0000) rectangle (2.0000,0.0000);
\fill[p1col3] (2.0000,4.0000) rectangle (3.0000,0.0000);
\fill[p1col4] (3.0000,4.0000) rectangle (4.0000,0.0000);
\end{tikzpicture}
 & \begin{tikzpicture}[scale=30/100,font=\tiny]
\fill[p1col1] (0.0000,4.0000) rectangle (1.3333,0.0000);
\fill[p1col2] (1.3333,4.0000) rectangle (2.6667,0.0000);
\fill[p1col3] (2.6667,4.0000) rectangle (4.0000,2.0000);
\fill[p1col4] (2.6667,2.0000) rectangle (4.0000,0.0000);
\end{tikzpicture}
 & \begin{tikzpicture}[scale=30/100,font=\tiny]
\fill[p1col1] (0.0000,4.0000) rectangle (1.3333,0.0000);
\fill[p1col2] (1.3333,4.0000) rectangle (2.6667,2.0000);
\fill[p1col3] (2.6667,4.0000) rectangle (4.0000,2.0000);
\fill[p1col4] (1.3333,2.0000) rectangle (4.0000,0.0000);
\end{tikzpicture}
 & \begin{tikzpicture}[scale=30/100,font=\tiny]
\fill[p1col1] (0.0000,4.0000) rectangle (1.3333,2.0000);
\fill[p1col2] (1.3333,4.0000) rectangle (2.6667,2.0000);
\fill[p1col3] (2.6667,4.0000) rectangle (4.0000,2.0000);
\fill[p1col4] (0.0000,2.0000) rectangle (4.0000,0.0000);
\end{tikzpicture}
 & \begin{tikzpicture}[scale=30/100,font=\tiny]
\fill[p1col1] (0.0000,4.0000) rectangle (2.0000,1.3333);
\fill[p1col2] (2.0000,4.0000) rectangle (4.0000,2.6667);
\fill[p1col3] (2.0000,2.6667) rectangle (4.0000,1.3333);
\fill[p1col4] (0.0000,1.3333) rectangle (4.0000,0.0000);
\end{tikzpicture}
 & \begin{tikzpicture}[scale=30/100,font=\tiny]
\fill[p1col1] (0.0000,4.0000) rectangle (2.0000,0.0000);
\fill[p1col2] (2.0000,4.0000) rectangle (4.0000,2.6667);
\fill[p1col3] (2.0000,2.6667) rectangle (4.0000,1.3333);
\fill[p1col4] (2.0000,1.3333) rectangle (4.0000,0.0000);
\end{tikzpicture}
 & \begin{tikzpicture}[scale=30/100,font=\tiny]
\fill[p1col1] (0.0000,4.0000) rectangle (1.3333,0.0000);
\fill[p1col2] (1.3333,4.0000) rectangle (4.0000,2.0000);
\fill[p1col3] (1.3333,2.0000) rectangle (2.6667,0.0000);
\fill[p1col4] (2.6667,2.0000) rectangle (4.0000,0.0000);
\end{tikzpicture}
 & \begin{tikzpicture}[scale=30/100,font=\tiny]
\fill[p1col1] (0.0000,4.0000) rectangle (1.3333,0.0000);
\fill[p1col2] (1.3333,4.0000) rectangle (2.6667,2.0000);
\fill[p1col3] (1.3333,2.0000) rectangle (2.6667,0.0000);
\fill[p1col4] (2.6667,4.0000) rectangle (4.0000,0.0000);
\end{tikzpicture}
 & \begin{tikzpicture}[scale=30/100,font=\tiny]
\fill[p1col1] (0.0000,4.0000) rectangle (1.3333,2.0000);
\fill[p1col2] (1.3333,4.0000) rectangle (2.6667,2.0000);
\fill[p1col3] (0.0000,2.0000) rectangle (2.6667,0.0000);
\fill[p1col4] (2.6667,4.0000) rectangle (4.0000,0.0000);
\end{tikzpicture}
 & \begin{tikzpicture}[scale=30/100,font=\tiny]
\fill[p1col1] (0.0000,4.0000) rectangle (1.3333,2.0000);
\fill[p1col2] (1.3333,4.0000) rectangle (4.0000,2.0000);
\fill[p1col3] (0.0000,2.0000) rectangle (2.6667,0.0000);
\fill[p1col4] (2.6667,2.0000) rectangle (4.0000,0.0000);
\end{tikzpicture}
 & \begin{tikzpicture}[scale=30/100,font=\tiny]
\fill[p1col1] (0.0000,4.0000) rectangle (2.6667,2.0000);
\fill[p1col2] (2.6667,4.0000) rectangle (4.0000,2.0000);
\fill[p1col3] (0.0000,2.0000) rectangle (1.3333,0.0000);
\fill[p1col4] (1.3333,2.0000) rectangle (4.0000,0.0000);
\end{tikzpicture}
 \\
\end{tabular}
\caption{$n=4$}
\end{figure}

\begin{figure}[h!]
\begin{tabular}{ccccccccc}
\begin{tikzpicture}[scale=28/100,font=\tiny]
\fill[p1col1] (0.0000,5.0000) rectangle (1.0000,0.0000);
\fill[p1col2] (1.0000,5.0000) rectangle (2.0000,0.0000);
\fill[p1col3] (2.0000,5.0000) rectangle (3.0000,0.0000);
\fill[p1col4] (3.0000,5.0000) rectangle (4.0000,0.0000);
\fill[p1col5] (4.0000,5.0000) rectangle (5.0000,0.0000);
\end{tikzpicture}
 & \begin{tikzpicture}[scale=28/100,font=\tiny]
\fill[p1col1] (0.0000,5.0000) rectangle (1.2500,0.0000);
\fill[p1col2] (1.2500,5.0000) rectangle (2.5000,0.0000);
\fill[p1col3] (2.5000,5.0000) rectangle (3.7500,0.0000);
\fill[p1col4] (3.7500,5.0000) rectangle (5.0000,2.5000);
\fill[p1col5] (3.7500,2.5000) rectangle (5.0000,0.0000);
\end{tikzpicture}
 & \begin{tikzpicture}[scale=28/100,font=\tiny]
\fill[p1col1] (0.0000,5.0000) rectangle (1.2500,0.0000);
\fill[p1col2] (1.2500,5.0000) rectangle (2.5000,0.0000);
\fill[p1col3] (2.5000,5.0000) rectangle (3.7500,2.5000);
\fill[p1col4] (3.7500,5.0000) rectangle (5.0000,2.5000);
\fill[p1col5] (2.5000,2.5000) rectangle (5.0000,0.0000);
\end{tikzpicture}
 & \begin{tikzpicture}[scale=28/100,font=\tiny]
\fill[p1col1] (0.0000,5.0000) rectangle (1.2500,0.0000);
\fill[p1col2] (1.2500,5.0000) rectangle (2.5000,2.5000);
\fill[p1col3] (2.5000,5.0000) rectangle (3.7500,2.5000);
\fill[p1col4] (3.7500,5.0000) rectangle (5.0000,2.5000);
\fill[p1col5] (1.2500,2.5000) rectangle (5.0000,0.0000);
\end{tikzpicture}
 & \begin{tikzpicture}[scale=28/100,font=\tiny]
\fill[p1col1] (0.0000,5.0000) rectangle (1.2500,2.5000);
\fill[p1col2] (1.2500,5.0000) rectangle (2.5000,2.5000);
\fill[p1col3] (2.5000,5.0000) rectangle (3.7500,2.5000);
\fill[p1col4] (3.7500,5.0000) rectangle (5.0000,2.5000);
\fill[p1col5] (0.0000,2.5000) rectangle (5.0000,0.0000);
\end{tikzpicture}
 & \begin{tikzpicture}[scale=28/100,font=\tiny]
\fill[p1col1] (0.0000,5.0000) rectangle (1.6667,1.6667);
\fill[p1col2] (1.6667,5.0000) rectangle (3.3333,1.6667);
\fill[p1col3] (3.3333,5.0000) rectangle (5.0000,3.3333);
\fill[p1col4] (3.3333,3.3333) rectangle (5.0000,1.6667);
\fill[p1col5] (0.0000,1.6667) rectangle (5.0000,0.0000);
\end{tikzpicture}
 & \begin{tikzpicture}[scale=28/100,font=\tiny]
\fill[p1col1] (0.0000,5.0000) rectangle (1.6667,0.0000);
\fill[p1col2] (1.6667,5.0000) rectangle (3.3333,1.6667);
\fill[p1col3] (3.3333,5.0000) rectangle (5.0000,3.3333);
\fill[p1col4] (3.3333,3.3333) rectangle (5.0000,1.6667);
\fill[p1col5] (1.6667,1.6667) rectangle (5.0000,0.0000);
\end{tikzpicture}
 & \begin{tikzpicture}[scale=28/100,font=\tiny]
\fill[p1col1] (0.0000,5.0000) rectangle (1.6667,0.0000);
\fill[p1col2] (1.6667,5.0000) rectangle (3.3333,0.0000);
\fill[p1col3] (3.3333,5.0000) rectangle (5.0000,3.3333);
\fill[p1col4] (3.3333,3.3333) rectangle (5.0000,1.6667);
\fill[p1col5] (3.3333,1.6667) rectangle (5.0000,0.0000);
\end{tikzpicture}
 & \begin{tikzpicture}[scale=28/100,font=\tiny]
\fill[p1col1] (0.0000,5.0000) rectangle (1.2500,0.0000);
\fill[p1col2] (1.2500,5.0000) rectangle (2.5000,0.0000);
\fill[p1col3] (2.5000,5.0000) rectangle (5.0000,2.5000);
\fill[p1col4] (2.5000,2.5000) rectangle (3.7500,0.0000);
\fill[p1col5] (3.7500,2.5000) rectangle (5.0000,0.0000);
\end{tikzpicture}
 \\
\end{tabular}
\caption{$n=5$. The 2 non-guillotine rectangulations are marked by~*.}
\end{figure}
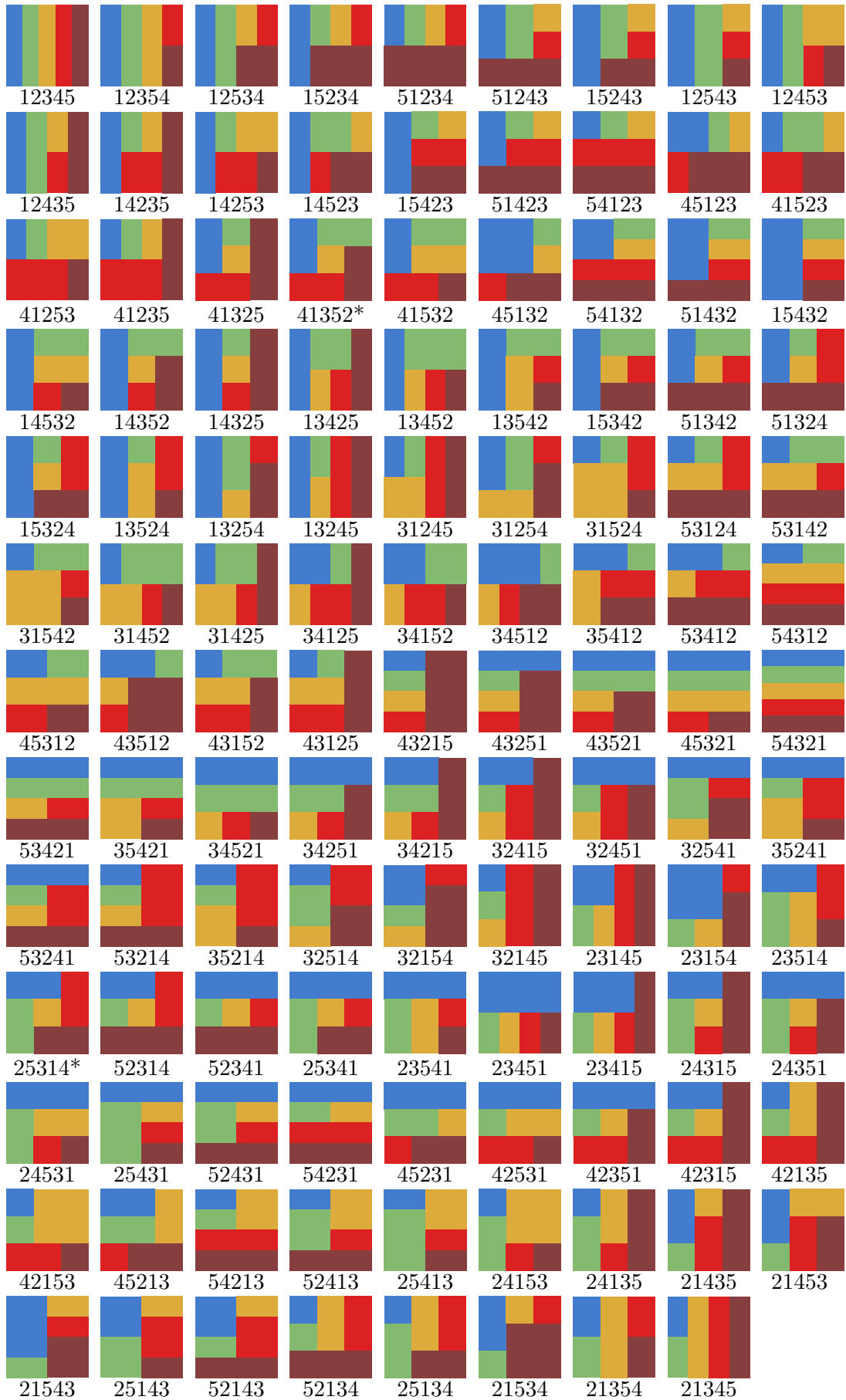

\clearpage

\subsection{Diagonal rectangulations}
\hphantom{x}

\begin{figure}[h!]
\begin{tabular}{c}
\begin{tikzpicture}[scale=120/100,font=\tiny]
\fill[p1col1] (0,1) rectangle (1,0);
\draw (0,1)--(1,0);
\end{tikzpicture}
 \\
1 \\[1mm]
\end{tabular}
\caption{$n=1$}
\end{figure}

\begin{figure}[h!]
\begin{tabular}{cc}
\begin{tikzpicture}[scale=60/100,font=\tiny]
\fill[p1col1] (0,2) rectangle (1,0);
\fill[p1col2] (1,2) rectangle (2,0);
\draw (0,2)--(2,0);
\end{tikzpicture}
 & \begin{tikzpicture}[scale=60/100,font=\tiny]
\fill[p1col2] (0,1) rectangle (2,0);
\fill[p1col1] (0,2) rectangle (2,1);
\draw (0,2)--(2,0);
\end{tikzpicture}
 \\
12 & 21 \\
\end{tabular}
\caption{$n=2$}
\end{figure}

\begin{figure}[h!]
\begin{tabular}{cccccc}
\begin{tikzpicture}[scale=40/100,font=\tiny]
\fill[p1col1] (0,3) rectangle (1,0);
\fill[p1col2] (1,3) rectangle (2,0);
\fill[p1col3] (2,3) rectangle (3,0);
\draw (0,3)--(3,0);
\end{tikzpicture}
 & \begin{tikzpicture}[scale=40/100,font=\tiny]
\fill[p1col1] (0,3) rectangle (1,0);
\fill[p1col3] (1,1) rectangle (3,0);
\fill[p1col2] (1,3) rectangle (3,1);
\draw (0,3)--(3,0);
\end{tikzpicture}
 & \begin{tikzpicture}[scale=40/100,font=\tiny]
\fill[p1col3] (0,1) rectangle (3,0);
\fill[p1col1] (0,3) rectangle (1,1);
\fill[p1col2] (1,3) rectangle (3,1);
\draw (0,3)--(3,0);
\end{tikzpicture}
 & \begin{tikzpicture}[scale=40/100,font=\tiny]
\fill[p1col3] (0,1) rectangle (3,0);
\fill[p1col2] (0,2) rectangle (3,1);
\fill[p1col1] (0,3) rectangle (3,2);
\draw (0,3)--(3,0);
\end{tikzpicture}
 & \begin{tikzpicture}[scale=40/100,font=\tiny]
\fill[p1col2] (0,2) rectangle (2,0);
\fill[p1col3] (2,2) rectangle (3,0);
\fill[p1col1] (0,3) rectangle (3,2);
\draw (0,3)--(3,0);
\end{tikzpicture}
 & \begin{tikzpicture}[scale=40/100,font=\tiny]
\fill[p1col2] (0,2) rectangle (2,0);
\fill[p1col1] (0,3) rectangle (2,2);
\fill[p1col3] (2,3) rectangle (3,0);
\draw (0,3)--(3,0);
\end{tikzpicture}
 \\
123 & 132 & 312 & 321 & 231 & 213 \\
\end{tabular}
\caption{$n=3$}
\end{figure}

\begin{figure}[h!]
\begin{tabular}{ccccccccccc}
\begin{tikzpicture}[scale=30/100,font=\tiny]
\fill[p1col1] (0,4) rectangle (1,0);
\fill[p1col2] (1,4) rectangle (2,0);
\fill[p1col3] (2,4) rectangle (3,0);
\fill[p1col4] (3,4) rectangle (4,0);
\draw (0,4)--(4,0);
\end{tikzpicture}
 & \begin{tikzpicture}[scale=30/100,font=\tiny]
\fill[p1col1] (0,4) rectangle (1,0);
\fill[p1col2] (1,4) rectangle (2,0);
\fill[p1col4] (2,1) rectangle (4,0);
\fill[p1col3] (2,4) rectangle (4,1);
\draw (0,4)--(4,0);
\end{tikzpicture}
 & \begin{tikzpicture}[scale=30/100,font=\tiny]
\fill[p1col1] (0,4) rectangle (1,0);
\fill[p1col4] (1,1) rectangle (4,0);
\fill[p1col2] (1,4) rectangle (2,1);
\fill[p1col3] (2,4) rectangle (4,1);
\draw (0,4)--(4,0);
\end{tikzpicture}
 & \begin{tikzpicture}[scale=30/100,font=\tiny]
\fill[p1col4] (0,1) rectangle (4,0);
\fill[p1col1] (0,4) rectangle (1,1);
\fill[p1col2] (1,4) rectangle (2,1);
\fill[p1col3] (2,4) rectangle (4,1);
\draw (0,4)--(4,0);
\end{tikzpicture}
 & \begin{tikzpicture}[scale=30/100,font=\tiny]
\fill[p1col4] (0,1) rectangle (4,0);
\fill[p1col1] (0,4) rectangle (1,1);
\fill[p1col3] (1,2) rectangle (4,1);
\fill[p1col2] (1,4) rectangle (4,2);
\draw (0,4)--(4,0);
\end{tikzpicture}
 & \begin{tikzpicture}[scale=30/100,font=\tiny]
\fill[p1col1] (0,4) rectangle (1,0);
\fill[p1col4] (1,1) rectangle (4,0);
\fill[p1col3] (1,2) rectangle (4,1);
\fill[p1col2] (1,4) rectangle (4,2);
\draw (0,4)--(4,0);
\end{tikzpicture}
 & \begin{tikzpicture}[scale=30/100,font=\tiny]
\fill[p1col1] (0,4) rectangle (1,0);
\fill[p1col3] (1,2) rectangle (3,0);
\fill[p1col4] (3,2) rectangle (4,0);
\fill[p1col2] (1,4) rectangle (4,2);
\draw (0,4)--(4,0);
\end{tikzpicture}
 & \begin{tikzpicture}[scale=30/100,font=\tiny]
\fill[p1col1] (0,4) rectangle (1,0);
\fill[p1col3] (1,2) rectangle (3,0);
\fill[p1col2] (1,4) rectangle (3,2);
\fill[p1col4] (3,4) rectangle (4,0);
\draw (0,4)--(4,0);
\end{tikzpicture}
 & \begin{tikzpicture}[scale=30/100,font=\tiny]
\fill[p1col3] (0,2) rectangle (3,0);
\fill[p1col1] (0,4) rectangle (1,2);
\fill[p1col2] (1,4) rectangle (3,2);
\fill[p1col4] (3,4) rectangle (4,0);
\draw (0,4)--(4,0);
\end{tikzpicture}
 & \begin{tikzpicture}[scale=30/100,font=\tiny]
\fill[p1col3] (0,2) rectangle (3,0);
\fill[p1col1] (0,4) rectangle (1,2);
\fill[p1col4] (3,2) rectangle (4,0);
\fill[p1col2] (1,4) rectangle (4,2);
\draw (0,4)--(4,0);
\end{tikzpicture}
 & \begin{tikzpicture}[scale=30/100,font=\tiny]
\fill[p1col4] (0,1) rectangle (4,0);
\fill[p1col3] (0,2) rectangle (4,1);
\fill[p1col1] (0,4) rectangle (1,2);
\fill[p1col2] (1,4) rectangle (4,2);
\draw (0,4)--(4,0);
\end{tikzpicture}
 \\
\end{tabular}
\caption{$n=4$}
\end{figure}

\begin{figure}[h!]
\begin{tabular}{ccccccccc}
\begin{tikzpicture}[scale=30/100,font=\tiny]
\fill[p1col1] (0,5) rectangle (1,0);
\fill[p1col2] (1,5) rectangle (2,0);
\fill[p1col3] (2,5) rectangle (3,0);
\fill[p1col4] (3,5) rectangle (4,0);
\fill[p1col5] (4,5) rectangle (5,0);
\draw (0,5)--(5,0);
\end{tikzpicture}
 & \begin{tikzpicture}[scale=30/100,font=\tiny]
\fill[p1col1] (0,5) rectangle (1,0);
\fill[p1col2] (1,5) rectangle (2,0);
\fill[p1col3] (2,5) rectangle (3,0);
\fill[p1col5] (3,1) rectangle (5,0);
\fill[p1col4] (3,5) rectangle (5,1);
\draw (0,5)--(5,0);
\end{tikzpicture}
 & \begin{tikzpicture}[scale=30/100,font=\tiny]
\fill[p1col1] (0,5) rectangle (1,0);
\fill[p1col2] (1,5) rectangle (2,0);
\fill[p1col5] (2,1) rectangle (5,0);
\fill[p1col3] (2,5) rectangle (3,1);
\fill[p1col4] (3,5) rectangle (5,1);
\draw (0,5)--(5,0);
\end{tikzpicture}
 & \begin{tikzpicture}[scale=30/100,font=\tiny]
\fill[p1col1] (0,5) rectangle (1,0);
\fill[p1col5] (1,1) rectangle (5,0);
\fill[p1col2] (1,5) rectangle (2,1);
\fill[p1col3] (2,5) rectangle (3,1);
\fill[p1col4] (3,5) rectangle (5,1);
\draw (0,5)--(5,0);
\end{tikzpicture}
 & \begin{tikzpicture}[scale=30/100,font=\tiny]
\fill[p1col5] (0,1) rectangle (5,0);
\fill[p1col1] (0,5) rectangle (1,1);
\fill[p1col2] (1,5) rectangle (2,1);
\fill[p1col3] (2,5) rectangle (3,1);
\fill[p1col4] (3,5) rectangle (5,1);
\draw (0,5)--(5,0);
\end{tikzpicture}
 & \begin{tikzpicture}[scale=30/100,font=\tiny]
\fill[p1col5] (0,1) rectangle (5,0);
\fill[p1col1] (0,5) rectangle (1,1);
\fill[p1col2] (1,5) rectangle (2,1);
\fill[p1col4] (2,2) rectangle (5,1);
\fill[p1col3] (2,5) rectangle (5,2);
\draw (0,5)--(5,0);
\end{tikzpicture}
 & \begin{tikzpicture}[scale=30/100,font=\tiny]
\fill[p1col1] (0,5) rectangle (1,0);
\fill[p1col5] (1,1) rectangle (5,0);
\fill[p1col2] (1,5) rectangle (2,1);
\fill[p1col4] (2,2) rectangle (5,1);
\fill[p1col3] (2,5) rectangle (5,2);
\draw (0,5)--(5,0);
\end{tikzpicture}
 & \begin{tikzpicture}[scale=30/100,font=\tiny]
\fill[p1col1] (0,5) rectangle (1,0);
\fill[p1col2] (1,5) rectangle (2,0);
\fill[p1col5] (2,1) rectangle (5,0);
\fill[p1col4] (2,2) rectangle (5,1);
\fill[p1col3] (2,5) rectangle (5,2);
\draw (0,5)--(5,0);
\end{tikzpicture}
 & \begin{tikzpicture}[scale=30/100,font=\tiny]
\fill[p1col1] (0,5) rectangle (1,0);
\fill[p1col2] (1,5) rectangle (2,0);
\fill[p1col4] (2,2) rectangle (4,0);
\fill[p1col5] (4,2) rectangle (5,0);
\fill[p1col3] (2,5) rectangle (5,2);
\draw (0,5)--(5,0);
\end{tikzpicture}
 \\
\end{tabular}
\caption{$n=5$. The 2 non-guillotine rectangulations are marked by~*.}
\end{figure}
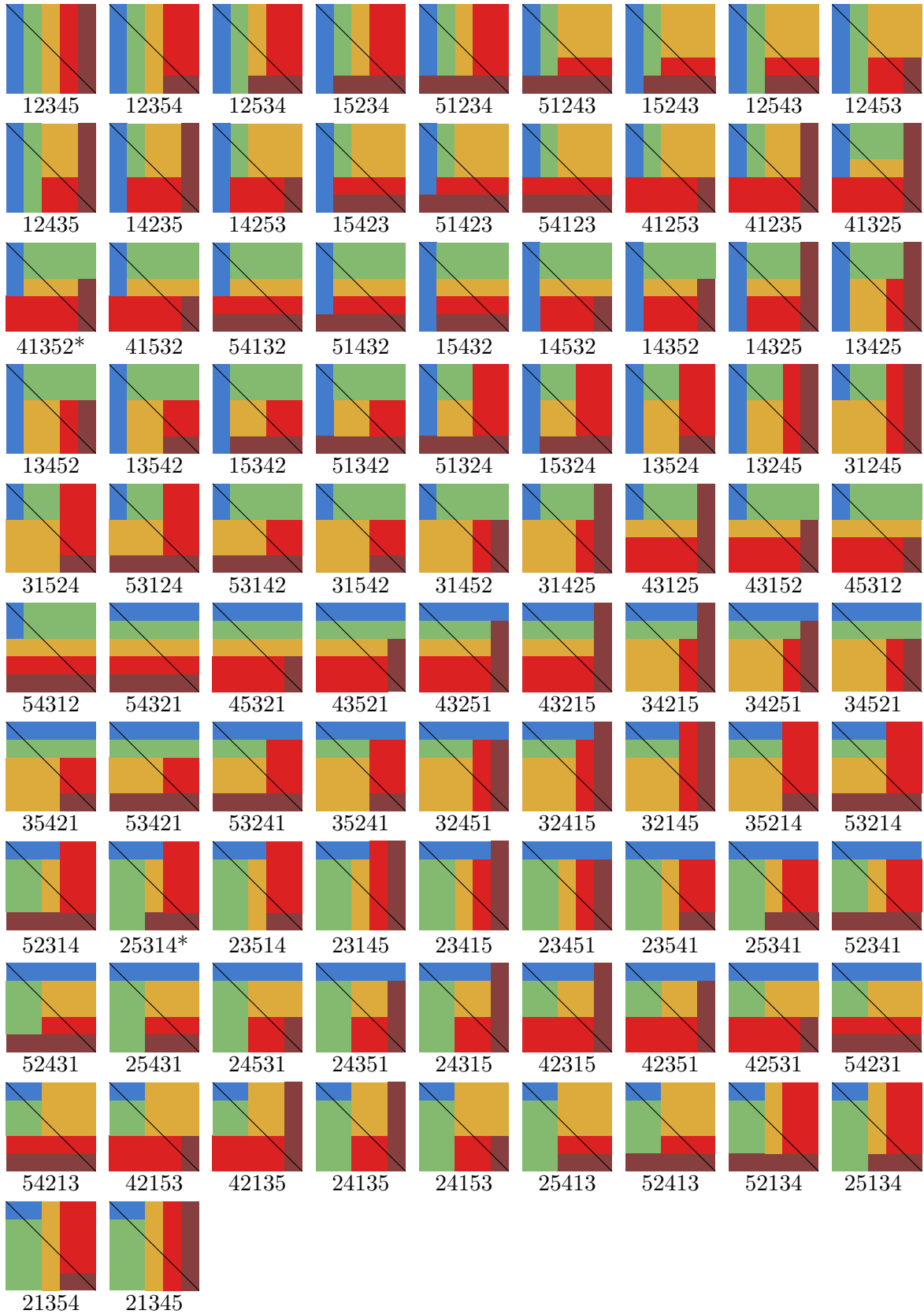

\clearpage

\subsection{Block-aligned rectangulations}
\hphantom{x}

\begin{figure}[h!]
\begin{tabular}{c}
\begin{tikzpicture}[scale=120/100,font=\tiny]
\fill[p1col1] (0,1) rectangle (1,0);
\draw (0,1)--(1,0);
\end{tikzpicture}
 \\
1 \\[1mm]
\end{tabular}
\caption{$n=1$}
\end{figure}

\begin{figure}[h!]
\begin{tabular}{c}
\begin{tikzpicture}[scale=60/100,font=\tiny]
\fill[p1col2] (0,1) rectangle (2,0);
\fill[p1col1] (0,2) rectangle (2,1);
\draw (0,2)--(2,0);
\end{tikzpicture}
 \\
21 \\
\end{tabular}
\caption{$n=2$}
\end{figure}

\begin{figure}[h!]
\begin{tabular}{cc}
\begin{tikzpicture}[scale=40/100,font=\tiny]
\fill[p1col1] (0,3) rectangle (1,0);
\fill[p1col2] (1,3) rectangle (2,0);
\fill[p1col3] (2,3) rectangle (3,0);
\draw (0,3)--(3,0);
\end{tikzpicture}
 & \begin{tikzpicture}[scale=40/100,font=\tiny]
\fill[p1col3] (0,1) rectangle (3,0);
\fill[p1col2] (0,2) rectangle (3,1);
\fill[p1col1] (0,3) rectangle (3,2);
\draw (0,3)--(3,0);
\end{tikzpicture}
  \\
123 & 321 \\
\end{tabular}
\caption{$n=3$}
\end{figure}

\begin{figure}[h!]
\begin{tabular}{cccccc}
\begin{tikzpicture}[scale=30/100,font=\tiny]
\fill[p1col1] (0,4) rectangle (1,0);
\fill[p1col2] (1,4) rectangle (2,0);
\fill[p1col3] (2,4) rectangle (3,0);
\fill[p1col4] (3,4) rectangle (4,0);
\draw (0,4)--(4,0);
\end{tikzpicture}
 & \begin{tikzpicture}[scale=30/100,font=\tiny]
\fill[p1col1] (0,4) rectangle (1,0);
\fill[p1col4] (1,1) rectangle (4,0);
\fill[p1col3] (1,2) rectangle (4,1);
\fill[p1col2] (1,4) rectangle (4,2);
\draw (0,4)--(4,0);
\end{tikzpicture}
 & \begin{tikzpicture}[scale=30/100,font=\tiny]
\fill[p1col4] (0,1) rectangle (4,0);
\fill[p1col1] (0,4) rectangle (1,1);
\fill[p1col3] (1,2) rectangle (4,1);
\fill[p1col2] (1,4) rectangle (4,2);
\draw (0,4)--(4,0);
\end{tikzpicture}
 & \begin{tikzpicture}[scale=30/100,font=\tiny]
\fill[p1col4] (0,1) rectangle (4,0);
\fill[p1col3] (0,2) rectangle (4,1);
\fill[p1col2] (0,3) rectangle (4,2);
\fill[p1col1] (0,4) rectangle (4,3);
\draw (0,4)--(4,0);
\end{tikzpicture}
 & \begin{tikzpicture}[scale=30/100,font=\tiny]
\fill[p1col2] (0,3) rectangle (2,0);
\fill[p1col3] (2,3) rectangle (3,0);
\fill[p1col4] (3,3) rectangle (4,0);
\fill[p1col1] (0,4) rectangle (4,3);
\draw (0,4)--(4,0);
\end{tikzpicture}
 & \begin{tikzpicture}[scale=30/100,font=\tiny]
\fill[p1col2] (0,3) rectangle (2,0);
\fill[p1col3] (2,3) rectangle (3,0);
\fill[p1col1] (0,4) rectangle (3,3);
\fill[p1col4] (3,4) rectangle (4,0);
\draw (0,4)--(4,0);
\end{tikzpicture}
 \\
1234 & 1432 & 4132 & 4321 & 2341 & 2314 \\
\end{tabular}
\caption{$n=4$}
\end{figure}

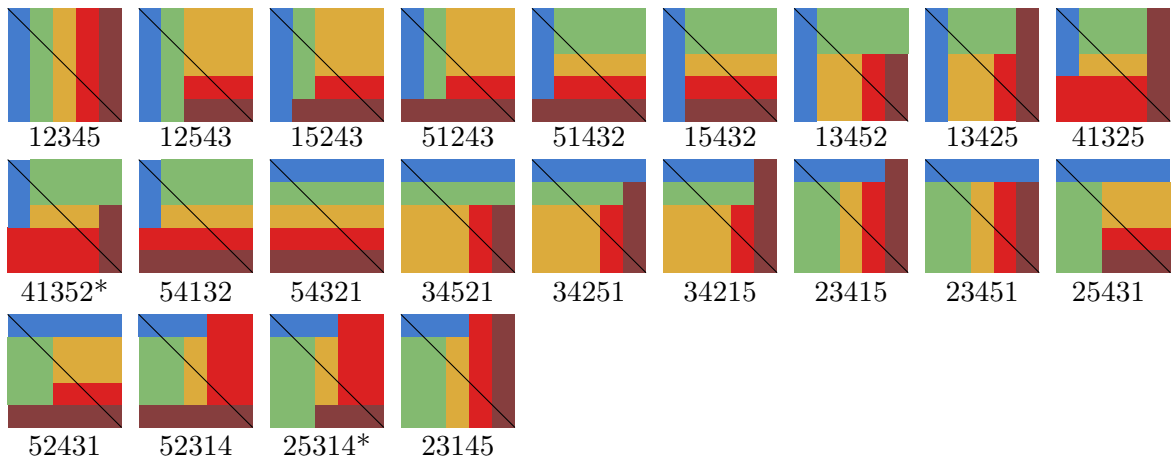
\begin{figure}[h!]
\begin{tabular}{ccccccccc}
\begin{tikzpicture}[scale=30/100,font=\tiny]
\fill[p1col1] (0,5) rectangle (1,0);
\fill[p1col2] (1,5) rectangle (2,0);
\fill[p1col3] (2,5) rectangle (3,0);
\fill[p1col4] (3,5) rectangle (4,0);
\fill[p1col5] (4,5) rectangle (5,0);
\draw (0,5)--(5,0);
\end{tikzpicture}
 & \begin{tikzpicture}[scale=30/100,font=\tiny]
\fill[p1col1] (0,5) rectangle (1,0);
\fill[p1col2] (1,5) rectangle (2,0);
\fill[p1col5] (2,1) rectangle (5,0);
\fill[p1col4] (2,2) rectangle (5,1);
\fill[p1col3] (2,5) rectangle (5,2);
\draw (0,5)--(5,0);
\end{tikzpicture}
 & \begin{tikzpicture}[scale=30/100,font=\tiny]
\fill[p1col1] (0,5) rectangle (1,0);
\fill[p1col5] (1,1) rectangle (5,0);
\fill[p1col2] (1,5) rectangle (2,1);
\fill[p1col4] (2,2) rectangle (5,1);
\fill[p1col3] (2,5) rectangle (5,2);
\draw (0,5)--(5,0);
\end{tikzpicture}
 & \begin{tikzpicture}[scale=30/100,font=\tiny]
\fill[p1col5] (0,1) rectangle (5,0);
\fill[p1col1] (0,5) rectangle (1,1);
\fill[p1col2] (1,5) rectangle (2,1);
\fill[p1col4] (2,2) rectangle (5,1);
\fill[p1col3] (2,5) rectangle (5,2);
\draw (0,5)--(5,0);
\end{tikzpicture}
 & \begin{tikzpicture}[scale=30/100,font=\tiny]
\fill[p1col5] (0,1) rectangle (5,0);
\fill[p1col1] (0,5) rectangle (1,1);
\fill[p1col4] (1,2) rectangle (5,1);
\fill[p1col3] (1,3) rectangle (5,2);
\fill[p1col2] (1,5) rectangle (5,3);
\draw (0,5)--(5,0);
\end{tikzpicture}
 & \begin{tikzpicture}[scale=30/100,font=\tiny]
\fill[p1col1] (0,5) rectangle (1,0);
\fill[p1col5] (1,1) rectangle (5,0);
\fill[p1col4] (1,2) rectangle (5,1);
\fill[p1col3] (1,3) rectangle (5,2);
\fill[p1col2] (1,5) rectangle (5,3);
\draw (0,5)--(5,0);
\end{tikzpicture}
 & \begin{tikzpicture}[scale=30/100,font=\tiny]
\fill[p1col1] (0,5) rectangle (1,0);
\fill[p1col3] (1,3) rectangle (3,0);
\fill[p1col4] (3,3) rectangle (4,0);
\fill[p1col5] (4,3) rectangle (5,0);
\fill[p1col2] (1,5) rectangle (5,3);
\draw (0,5)--(5,0);
\end{tikzpicture}
 & \begin{tikzpicture}[scale=30/100,font=\tiny]
\fill[p1col1] (0,5) rectangle (1,0);
\fill[p1col3] (1,3) rectangle (3,0);
\fill[p1col4] (3,3) rectangle (4,0);
\fill[p1col2] (1,5) rectangle (4,3);
\fill[p1col5] (4,5) rectangle (5,0);
\draw (0,5)--(5,0);
\end{tikzpicture}
 & \begin{tikzpicture}[scale=30/100,font=\tiny]
\fill[p1col4] (0,2) rectangle (4,0);
\fill[p1col1] (0,5) rectangle (1,2);
\fill[p1col3] (1,3) rectangle (4,2);
\fill[p1col2] (1,5) rectangle (4,3);
\fill[p1col5] (4,5) rectangle (5,0);
\draw (0,5)--(5,0);
\end{tikzpicture}
 \\
\end{tabular}
\caption{$n=5$. The 2 non-guillotine rectangulations are marked by~*.}
\end{figure}

\end{document}